\documentclass[12pt]{article}
\usepackage{mathtools}
\usepackage{comment}
\usepackage{amsopn,amsmath,amssymb,amsthm}
\usepackage[dvipsnames]{xcolor}
\usepackage[left=1.5cm,right=1.5cm,top=1.5cm,bottom=2cm]{geometry}
\usepackage{subcaption}
\usepackage[pdfencoding=unicode, psdextra, citecolor=blue, urlcolor=blue,
  linkcolor=blue, colorlinks=true, bookmarksopen=true]{hyperref}
\usepackage{graphicx}
\usepackage{csquotes}
\usepackage{enumerate}
\usepackage[normalem]{ulem}
\numberwithin{equation}{section}
\numberwithin{figure}{section}

\newcommand{\R}{\mathbb{R}}
\newcommand{\Z}{\mathbb{Z}}
\newcommand{\const}{\mathrm{const}}
\newcommand{\polar}{\circ}

\DeclareMathOperator{\argmax}{\mathrm{arg\,max}}

\newcommand{\ii}{\mathrm{\bf i}}
\newcommand{\jj}{\mathrm{\bf j}}
\newcommand{\kk}{\mathrm{\bf k}}

\usepackage{cite}

\def\h{\mathfrak{h}}
\def\sl{\mathfrak{sl}}
\def\su{\mathfrak{su}}
\def\se{\mathfrak{se}}
\def\sh{\mathfrak{sh}}
\def\H{\mathrm{H}}
\def\SL{\mathrm{SL}}
\def\SU{\mathrm{SU}}
\def\SE{\mathrm{SE}}
\def\SH{\mathrm{SH}}

\newcommand{\En}{\mathbb{E}}
\newcommand{\Cn}{\mathcal{C}}
\newcommand{\Un}{\mathcal{U}}
\newcommand{\Sn}{\mathbb{S}}
\def\then{\Rightarrow}
\def\co{\cos_{\Omega}}
\def\so{\sin_{\Omega}}
\def\coo{\cos_{\Omega^{\polar}}}
\def\soo{\sin_{\Omega^{\polar}}}
\def\tho{\theta^{\polar}}
\def\So{\Sn^{\polar}}
\DeclareMathOperator{\Id}{Id}
\DeclareMathOperator{\intt}{int}
\DeclareMathOperator{\spann}{span}
\newcommand{\eq}[1]{$(\protect\ref{#1})$}
\newcommand{\be}[1]{\begin{equation}\label{#1}}
\newcommand{\ee}{\end{equation}}
\DeclareMathOperator{\sgn}{sgn}
\DeclareMathOperator{\conv}{conv}
\def\om{\omega}
\def\a{\alpha}
\def\b{\beta}
\def\lam{\lambda}
\DeclareMathOperator{\Sing}{\mathfrak{S}}

\newcommand{\twofiglabel}[6]
{
\begin{figure}[htbp]
\includegraphics[width=0.47\textwidth]{#1}
\hfill
\includegraphics[width=0.47\textwidth]{#4}
\\
\parbox[t]{0.49\textwidth}{\caption{#2}\label{#3}}
\hfill
\parbox[t]{0.49\textwidth}{\caption{#5}\label{#6}}
\end{figure}
}

\newcommand{\twofiglabels}[8]
{
\begin{figure}[htbp]
\parbox[t]{0.49\textwidth}{
\begin{center}
\includegraphics[width=#7\textwidth]{#1}
\end{center}
}
\hfill
\parbox[t]{0.49\textwidth}{
\begin{center}
\includegraphics[width=#8\textwidth]{#4}
\end{center}
}
\\
\parbox[t]{0.495\textwidth}{\caption{#2}\label{#3}}
\hfill
\parbox[t]{0.495\textwidth}{\caption{#5}\label{#6}}
\end{figure}
}

\newcommand*{\arxiv}[1]{\href{http://arxiv.org/abs/#1}{arXiv: #1}}

\newcommand{\N}{\mathbb{N}}
\theoremstyle{plain}
\newtheorem{thm}{Theorem}[section]
\newtheorem{prop}{Proposition}[section]
\newtheorem{corollary}{Corollary}[section]

\theoremstyle{definition}
\newtheorem{defn}{Definition}[section]

\newtheorem{remark}{Remark}[section]

\begin{document}
\title{Extremals for a series of sub-Finsler problems with 2-dimensional control via convex trigonometry\thanks{Section 6 was written by A.A. Ardentov.  Sections 1--3, 5, 7 were written by L.V. Lokutsievskiy. Section 4 was written by Yu.L. Sachkov. All results in this paper are products of authors collaborative work. The work of A.A. Ardentov is supported by the Russian Science Foundation under grant 17-11-01387-p and performed in Ailamazyan Program Systems Institute of Russian Academy of Sciences. The work of L.V. Lokutsievskiy is supported by the Russian Science Foundation under grant 20-11-20169 and performed in Steklov Mathematical Institute of Russian Academy of Sciences. The work of Yu.L. Sachkov is supported by the Russian Foundation for Basic Research, project number 19-31-51023.}}
\author{A.A.~Ardentov\thanks{Ailamazyan Program Systems Institute, Russian Academy of Sciences, Pereslavl-Zalessky, Russia. E-mail: aaa@pereslavl.ru}, L.V.~Lokutsievskiy\thanks{Steklov Mathematical Institute of Russian Academy of Sciences, Moscow, Russia. E-mail: lion.lokut@gmail.com}, Yu.L.~Sachkov\thanks{Sirius University of Science and Technology (Sochi, Russia) and Ailamazyan Program Systems Institute of Russian Academy of Sciences (Pereslavl-Zalessky, Russia).  E-mail: yusachkov@gmail.com}}
\date{26.04.2020}

\maketitle

\begin{abstract} 
We consider a series of optimal control problems with 2-dimensional control lying in an arbitrary convex compact set~$\Omega$. The considered problems are well studied for the case when $\Omega$ is a unit disc, but barely studied for arbitrary $\Omega$. We derive extremals to these problems in general case by using machinery of convex trigonometry, which allows us to do this identically and independently on the shape of~$\Omega$. The paper  {describes geodesics in } (i)  the Finsler problem on the Lobachevsky hyperbolic plane; (ii) left-invariant sub-Finsler  {problems} on all unimodular 3D Lie groups ($\SU(2)$, $\SL(2)$, $\SE(2)$, $\SH(2)$); (iii) the problem of rolling ball on a plane with distance function given by~$\Omega$;  (iv) a series of ``yacht problems'' generalizing Euler's elastic problem, Markov-Dubins problem, Reeds-Shepp problem and a new sub-Riemannian problem on $\SE(2)$; and (v)  {the} plane dynamic motion {problem}.
\end{abstract}

\section*{Introduction}

In this paper, we obtain extremals in a series of optimal control problems with two-dimensional control lying in a compact convex set $\Omega\subset\mathbb{R}^2$, $0\in\mathrm{int}\,\Omega$. Usually Pontryagin maximum principle dictates an optimal control to move along the boundary $\partial\Omega$. In the case, when $\partial\Omega$ is the unit circle, this motion can be conveniently described by the trigonometric functions $\cos$ and $\sin$. This allows to find explicit formulas for extremals in a lot of optimal control problems with $\Omega$ being the unit disc, which in fact is very useful for any deep investigation of a problem. In the case, when $\Omega$ has an arbitrary form, finding formulas for extremals is much more difficult task also because extremals depend on the structure of $\Omega$. Hence, a lot of important problems (which are very well studied when $\Omega$ is the unit disc) are barely investigated or even not investigated at all if $\Omega$ has some different shape. It seems to us, it happens because of lack of a convenient way to describe motions of the optimal control along the boundary of an arbitrary flat convex set.

In the present paper, we use new functions $\cos_\Omega$ and $\sin_\Omega$ (which were introduced in~\cite{CT1}) to describe this motion. These new functions inherit a lot of convenient properties from the classical functions $\cos$ and $\sin$ (we usually use the term ``convex trigonometry'' for this collection of properties). On the one hand, convex trigonometry allows us to derive short formulas for extremals in terms of {$\cos_\Omega$ and $\sin_\Omega$} in a series of optimal control problems. On the other hand, the functions $\cos_\Omega$ and $\sin_\Omega$ can be expressed by classic functions for a lot of concrete sets $\Omega$ (see {Subsec.~\ref{subsec:explicit_sets}} for the detailed list {of} the sets). For example, if $\Omega=\{|x|^p+|y|^p\le 1\}$, $1\le p\le \infty$, the functions $\cos_\Omega$ and $\sin_\Omega$ coincide with well known Shelupsky's functions and can be expressed in terms of Euler beta function or hypergeometric function $_2F_1$.

So the following new optimal control problems are investigated in the present paper.

\begin{enumerate}
	
	\item In Sec.~\ref{sec:lobachavsky}, we find Finsler geodesics on the Lobachevsky hyperbolic plane. This problem is equivalent to a left-invariant Finsler problem on the group $\mathrm{Aff}(\mathbb{R})$ of  proper affine transformations of the real line $\mathbb{R}$. For centrally symmetrical sets $\Omega$, this problem was investigated in \cite{Gribanova}. In the present paper, we find all geodesics for all possible sets $\Omega$. The result is obtained by convex trigonometry, but the final answer does not contain the functions $\cos_\Omega$ and $\sin_\Omega$ at all (see Theorem~\ref{thm:FinslerLobachevsky}).
	
	\item 
Left-invariant  sub-Finsler problems on 3D unimodular Lie groups
	considered in Sec.~\ref{sec:3DLie} are a natural generalization of similar sub-Riemannian problems (recall that a problem is sub-Riemannian if $\Omega$ is an ellipse centered at the origin). Sub-Riemannian geodesics on all these groups can be parameterized by elliptic functions, see~\cite[Sec.~18]{ABB}. Moreover, in some of these problems, an optimal synthesis is known (see \cite{versh_gersh} for the Heisenberg group $\H_3$, \cite{berestovsky, boscain_rossi} for some $\mathrm{SO}(3)$ and $\SL(2)$ cases, \cite{max_sre, cut_sre1, cut_sre2} for $\SE(2)$, and \cite{sh21, sh22, sh23} for $\SH(2)$). Left invariant sub-Riemannian and sub-Finsler problems are of great importance for control theory and mathematics in general~\cite{ABB, SRG, b1, Breuillard-LeDonne1, montgomery}, and arise in a lot of practical applications, see e.g.~\cite{SCP, BDRS, boscain3level, Jurdjevic}. 
Sub-Finsler problems on $\SE(2)$ take important role for building control models of wheeled robots, e.g. famous Reeds-Shepp car~\cite{reeds-shepp} corresponds to $l_\infty$ case (see~Subsection~\ref{subsec:li}, Case~4), shortest paths for differential-drive mobile robot with the restriction on wheel-rotation speed~\cite{bvdr,chitsaz2009} are equavalent to sub-Finsler minimizers on $\SE(2)$ with $l_1$ norm (see~Subsection~\ref{subsec:l1}, Case~4). The~generalization of such models on polygonal and strictly convex smooth boundaries is explored in~\cite{bvmb}.	
Nonetheless, for sub-Finsler problems (when $\Omega$ is arbitrary), only geodesics on the Heisenberg group $\H_3$ are known (see \cite{buseman,berestovskiiHeisenberg}). In the present paper, for arbitrary $\Omega$, we find solutions to the vertical subsystem of the Pontryagin maximum principle in terms of convex trigonometry (see Sec.~\ref{sec:3DLie}). For example, this allows us to fully describe the phase portrait of the system and give a full description of singular extremals and their types in the $\ell_p$ case (see Subsecs.~\ref{subsec:li}--\ref{subsec:lp}).

	\item In Sec.~\ref{sec:rolling_ball}, we consider a sphere rolling on a plane without slipping or twisting. The problem is to find a rolling of the sphere from one given position to another such that the center draws a shortest possible path and the sphere rotates in a given way. In the classical case, when the path length is measured by the Euclidean distance, the problem is well studied (see \cite{Jurdjevic,JurdjevicBalls2}). For example, it is known that optimal motions of the center coincide with Euler elasticae. In Sec.~\ref{sec:rolling_ball}, we investigate an unsolved problem when the path length of the sphere center is measured by an arbitrary sub-Finsler metric on the plane. We find solutions to the vertical subsystem of the Pontryagin maximum principle in terms of convex trigonometry. In particular, this allows us to describe behavior of optimal sphere motions when the sub-Finsler metric is piecewise linear (see Theorem~\ref{thm:sphere_polygon}).
	
	\item In Sec.~\ref{sec:yachts}, we define a series of optimal control problems generalizing four classic problems (Euler elasticae, Dubins car, Reeds-Shepp car) and a new sub-Riemannian problem on $\SE(2)$, which express control models for a car-like robot. All the listed  problems were  investigated in literature (see \cite{euler} for Euler elasticae, \cite{dubins} for Dubins car, \cite{reeds-shepp} for Reeds-Shepp car, {\cite{max_sre, cut_sre1, cut_sre2} for the sub-Riemannian problem on $\SE(2)$}). We use arbitrary convex set $\Omega$ of unit velocities on the plane to generalize those models as yacht models moving in a sea (or a car moving on a non-horizontal plane).\footnote{Yachts problems are very popular applications of control theory nowadays (see \cite{yacht1, yacht2, yacht3, yacht4, yacht5}). Usually authors consider these problems for some specific set $\Omega$ of a given shape. Our approach allows us to investigate these problems simultaneously for all possible sets~$\Omega$.} For each case we investigate the phase portrait for the vertical subsystem of the Hamiltonian system and study possible extremals for arbitrary $\Omega$. A special type of extremals appears for not strictly convex $\Omega$ with edges. For the Markov-Dubins generalization, we describe all extremal controls explicitly in Theorems~\ref{Dubins1},~\ref{Dubins2}. For the Reeds-Shepp generalization, we find all patterns for control switchings. For other problems, we use the theory developed in Sec.~\ref{sec:second-order}  to classify all types of extremal controls. 
	
	\item In Sec.~\ref{sec:plane_dynamic}, we consider a system with drift that describes a controlled motion of a massive point on a plane, where the maximum admissible acceleration depends on the chosen direction. This problem generalizes the classic one-dimensional Pontryagin train control problem to the dimension~2.
	
\end{enumerate}

The paper has the following structure. In Sec.~\ref{sec:convex_trig}, we shortly introduce convex trigonometry, a list of sets~$\Omega$ for which the functions $\cos_\Omega$ and $\sin_\Omega$ are computed explicitly, and some new results on smoothness of these functions. Section~\ref{sec:second-order} contains important results on second order ODEs  {expressed by} functions of convex trigonometry. These results are used in all consequent sections. Sections~\ref{sec:lobachavsky}, \ref{sec:3DLie}--\ref{sec:plane_dynamic} contain investigation of the optimal control problems described above.

\section{Convex trigonometry}
\label{sec:convex_trig}

Let $\Omega\subset\R^2$ be a convex compact set and let $0\in\mathrm{int}\,\Omega$. In this section, we give a brief introduction to convex trigonometry, which was introduced in~\cite{CT1}. In the present paper, we use convex trigonometry in all the optimal control problems mentioned in the introduction.

\subsection{Notation and definitions}

Denote by $\mathbb{S}$ the area of the set $\Omega$.

\begin{defn}
	Let $\theta\in\R$ denote a generalized angle. If $0\le\theta<2\mathbb{S}$, then we choose a point $P_\theta$ on the boundary of $\Omega$ such that the area of the sector of $\Omega$ between the rays $Ox$ and $OP_\theta$ is $\frac12\theta$. By definition $\cos_\Omega\theta$ and $\sin_\Omega\theta$ are the coordinates of~$P_\theta$. If the generalized angle $\theta$ does not belong to the interval $\big[0;2\mathbb{S}\big)$, then we define the functions $\cos_\Omega$ and $\sin_\Omega$ as periodic with the period $2\mathbb{S}$; i.e., for $k\in\Z$ such that $\theta + 2\mathbb{S}k \in \big[0;2\mathbb{S}\big)$ we put
	\[
	\cos_\Omega\theta = \cos_\Omega(\theta + 2\mathbb{S}k);\qquad
	\sin_\Omega\theta = \sin_\Omega(\theta + 2\mathbb{S}k);\qquad
	P_\theta = P_{\theta+2\mathbb{S}k}.
	\]
\end{defn}

If $\Omega$ is the unit disc, $\Omega=\{x^2+y^2\le 1\}$, then the above definition produces the classic trigonometric functions $\cos$ and $\sin$.

We will use the polar set $\Omega^\polar$ together with the set $\Omega$:
\[
\Omega^\polar = \{(p,q)\in\R^{2*}:px+qy\le 1\mbox{ for all }(x,y)\in\Omega\}\subset\R^{2*}.
\]

\noindent The polar set $\Omega^\polar$ is (always) a convex and compact (as $0\in\mathrm{int}\,\Omega$) set and $0\in\mathrm{int}\,\Omega^\polar$ (as $\Omega$ is bounded). To avoid confusion we will assume that the set $\Omega$ lies in the plane with coordinates $(x,y)$ and the polar set $\Omega^\polar$ lies in the plane with coordinates $(p,q)$.

Note that $\Omega^{\circ\circ}=\Omega$ by the bipolar theorem (see \cite[Theorem 14.5]{Rockafellar}). We can apply the above definition of the generalized trigonometric functions to the polar set $\Omega^\polar$ and an arbitrary angle $\psi\in\R$ to construct $\cos_{\Omega^\polar}\psi$ and $\sin_{\Omega^\polar}\psi$, which are the coordinates of the appropriate point $Q_\psi\in\partial\Omega^\polar$. 

\begin{defn}
	We say that angles $\theta\in\R$ and $\theta^\polar\in\R$ correspond to each other and write $\theta\xleftrightarrow{\Omega\ }\theta^\polar$ if the supporting half-plane of $\Omega$ at $P_\theta$ is determined by the (co)vector~$Q_{\theta^\polar}$. When no confusing ensues we omit the symbol $\Omega$ over the arrow and write~$\theta\leftrightarrow\theta^\polar$.
\end{defn}

\subsection{Main properties of convex trigonometry}
\label{subsec:CT_properties}

Properties of the classic functions $(\cos,\sin)$ are inherited by two pairs of functions $(\cos_\Omega,\sin_\Omega)$ and $(\cos_{\Omega^\polar},\sin_{\Omega^\polar})$ for the sets $\Omega$ and $\Omega^\polar$. All the listed below properties are proved in \cite[Sec.~2]{CT1}.

\begin{enumerate}[I.]
	\item \label{property:main_trig_thm} \textit{Generalized Pythagorean identity}
	\begin{equation}
	\label{eq:main_trig}
		\theta\xleftrightarrow{\Omega\,}\theta^\polar
		\quad\Longleftrightarrow\quad
		\theta^\polar\xleftrightarrow{\Omega^\polar}\theta
		\quad\Longleftrightarrow\quad
		\cos_\Omega\theta\cos_{\Omega^\polar}\theta^\polar + 
		\sin_\Omega\theta\sin_{\Omega^\polar}\theta^\polar = 1.
	\end{equation}

	\item \textit{Differentiation formulae.} The functions $\cos_\Omega$ and $\sin_\Omega$ are Lipschitz continuous, and for a.e.\ $\theta$,
\begin{equation}\label{eq:diff}
		\cos'_\Omega\theta = -\sin_{\Omega^\polar}\theta^\polar
		\quad\mbox{ and }\quad
		\sin'_\Omega\theta = \cos_{\Omega^\polar}\theta^\polar
\end{equation}
	where $\theta^\polar\leftrightarrow\theta$ (in this case, there exists a unique $\theta^\polar\leftrightarrow\theta$). At points $\theta$ of non-differentiability, the correspondence $\theta^\polar\leftrightarrow\theta$ is not bijective, but both functions $\cos_\Omega$ and $\sin_\Omega$ have right and left derivatives, and the interval between them is given by the right-hand sides in~\eqref{eq:diff} for all $\theta^\polar\leftrightarrow\theta$.
	
	\bigskip
	
	\item \textit{Quasiperiodicity of the dependence $\theta^\polar(\theta)$}. Let $\mathbb{S}^\polar$ denote the area of $\Omega^\polar$. Obviously, the angle $\theta^\polar$ corresponding to $\theta$ is defined up to $2\mathbb{S}^\polar$, but we are able to choose a monotone function $\theta^\polar(\theta)\leftrightarrow\theta$. Then
	\[
		\theta^\polar(\theta+2\mathbb{S}k) = \theta^\polar(\theta) + 2\mathbb{S}^\polar k
		\quad\mbox{for}\quad k\in\Z.
	\]
	
	\item\textit{Polar change of coordinates.} Put $x=r\cos_\Omega\theta$ and $y=r\sin_\Omega\theta$. Then Jacobi matrix is 
	\[
		J=\begin{pmatrix}
			x'_r & x'_\theta\\
			y'_r & y'_\theta
		\end{pmatrix}
		=
		\begin{pmatrix}
			\cos_\Omega\theta & -r\sin_{\Omega^\polar}\theta^\polar\\
			\sin_\Omega\theta & r\cos_{\Omega^\polar}\theta^\polar\\
		\end{pmatrix}
	\]
	where $\theta^\polar\leftrightarrow\theta$. Moreover, $\det J = r$ by the generalized Pythagorean identity.
	
	\bigskip
	
	\item\textit{Inverse polar change of coordinates.} Let $(x(t),y(t))$ be an absolutely continuous curve that does not pass through the origin, and $x=r\cos_\Omega\theta$ and $y=r\sin_\Omega\theta$ as above. Then
	\[
		r=s_{\Omega^\polar}(x,y)\qquad\mbox{and}\qquad
		\dot\theta=\frac{x\dot y-\dot x y}{r^2}.
	\]
	\noindent The first equation holds for all $t$, and the second one holds for a.e.\ $t$.
\end{enumerate}

\subsection{List of explicitly computed functions $\cos_\Omega$ and $\sin_\Omega$.}
\label{subsec:explicit_sets}

In this subsection, we enumerate all known cases of sets $\Omega$ for which it is possible to compute functions of convex trigonometry.

\medskip

\begin{enumerate}[1.]
	\item \textit{Suppose $\Omega$ is a (convex) polyhedron.} In this case, the functions $\cos_\Omega$ and $\sin_\Omega$ are piecewise linear. They have corners when the point $P_\theta$ passes through a  vertex of $\Omega$. Explicit formulas for  $\cos_\Omega$ and $\sin_\Omega$ are given in~\cite[Sec.~3]{CT1} in terms of coordinates of vertices.
	
	\bigskip
	
	\item \textit{Suppose $\Omega=\{|x|^p+|y|^p\le 1\}$ for some $1\le p\le\infty$.} If $p=1$ or $p=\infty$, then $\Omega$ is a square and the functions $\cos_\Omega$ and $\sin_\Omega$ are computed as above. If $1<p<\infty$, then $\cos_\Omega$ and $\sin_\Omega$ coincide with Shelupcky's generalized trigonometry functions (see~\cite[Sec.~4]{CT2}). Shelupsky's functions form a 1-parametric family (determined by the parameter $p$) of pair of functions $(\cos_p,\sin_p)$, whereas the convex trigonometry functions introduced in~\cite{CT1} form an infinite-dimensional family determined by the set $\Omega$. In other words, Shelupsky's functions are a special case of convex trigonometry. Shelupsky's function are very convenient for eigenvalue problems for $p$-Laplacians (see~\cite{Shelupsky,WeiLiuElgindi} for details). In this case, the functions $\cos_\Omega$ and $\sin_\Omega$ can be expressed in terms of the inverse Euler beta function (see~\cite[Theorem~5]{CT2}).
	
	\bigskip
		
	\item\label{item:omega_paramtric}\textit{Suppose the boundary of $\Omega$ is defined parametrically} $\partial\Omega=\{(x(s),y(s)) {\in \R^2} \mid s\in\R\}$. Then, for any $s$, there is defined a generalized angle $\theta(s)$ such that $P_{\theta(s)}=(x(s),y(s))$. Obviously,
	\begin{equation}
	\label{eq:cos_sin_is_x_y_of_s}
		\cos_\Omega\theta(s)=x(s)\qquad\mbox{and}\qquad\sin_\Omega\theta(s)=y(s).
	\end{equation}
	\noindent Let us compute the function $\theta(s)$ and its inverse $s(\theta)$. Using the inverse polar change of coordinates, we get
	\begin{equation}
	\label{eq:theta_by_s}
		\theta'_s=x y'_s - x'_s y
		\quad\Leftrightarrow\quad
		\theta(s) = \int (x(s) y'_s(s) - x'_s(s) y(s))\,ds,
	\end{equation}
	since $r(s)=s_{\Omega^\polar}(x(s),y(s))=1$. For some cases this integral can be taken explicitly, but it is not possible to do it in general case. Nonetheless, we can often avoid explicit integration here! The main reason is the following: any of the above mentioned optimal control problems can be reduced to an ODE of the form $\dot\theta=f(\cos_\Omega\theta,\sin_\Omega\theta)$. From~\eqref{eq:cos_sin_is_x_y_of_s} we have the following ODE:
	\[
		\dot s = \dot\theta/\theta'_s = \frac{f(x(s),y(s))}{x(s)y'_s(s)-x'_s(s)y(s)}
		\quad\Leftrightarrow\quad
		t=\int\frac{x(s)y'_s(s)-x'_s(s)y(s)}{f(x(s),y(s))}\,ds.
	\]
	\noindent The last integral does not contain new functions $\cos_\Omega$ and $\sin_\Omega$, and its computation is reduced to classic calculus.
	
	\bigskip
	
		\item \textit{Suppose $\Omega$ is an ellipse.} In this case, its boundary is given by $x(s)=a\cos s+x_0$ and $y(s)=b\sin s+y_0$ (where $\frac{x_0^2}{a^2}+\frac{y_0^2}{b^2}<1$, since $0\in\mathrm{int}\,\Omega$). Using the previous item we obtain
		\[
		\theta = ab\,s - y_0 a\cos s + x_0 b\sin s + \const.
		\]
		\noindent We can not find $s(\theta)$ from this equation explicitly. But the function $f(x,y)$ is the square root of a linear or quadratic polynomial in all optimal control problems considered in the paper. Hence, the extremals equation $\dot\theta=f(\cos_\Omega\theta,\sin_\Omega\theta)$ in these problems takes the form
		\[
			\dot s = \frac{f(a\cos s+x_0,b\sin s+y_0)}{ab - y_0a\sin s + x_0b\cos s}
			\quad\Leftrightarrow\quad
			t=\int\frac{ab - y_0a\sin s + x_0b\cos s}{f(a\cos s+x_0,b\sin s+y_0)}\,dt,
		\]
		\noindent which can be expressed in elliptic functions.
		
		\bigskip

	\item\textit{Suppose that the Minkowski function $\mu_\Omega=s_{\Omega^\polar}$ is known: $\Omega=\{(x,y):s_{\Omega^\polar}(x,y)\le 1\}$.}	Let $\phi$ denote a classic angle. Let us parametrize the boundary $\partial\Omega$ by $\phi$, i.e.,
	\[
		x(\phi)=\cos_\Omega\theta(\phi) = \frac{\cos\phi}{s_{\Omega^\polar}(\cos \phi,\sin\phi)}
		\quad\mbox{and}\quad
		y(\phi)=\cos_\Omega\theta(\phi) = \frac{\sin\phi}{s_{\Omega^\polar}(\cos \phi,\sin\phi)}.
	\]
	
	\noindent Put $r(\phi)=s_{\Omega^\polar}(\cos\phi,\sin\phi)$ for short. Then $x+\ii y = \frac{1}{r}e^{\ii\phi}$. Thus, $\theta' = xy' - x'y$ by the inverse polar change of coordinates. So
	\[
		\theta' = \Re((x+\ii y)\ii(x'-\ii y')) = 
		\frac1r\Re\left(\ii e^{\ii\phi}\big(\frac{1}{r}e^{-\ii\phi}\Big)'\right) = 
		\frac1r\Re\left(\ii e^{\ii\phi}\Big(-\frac{\ii}{r}e^{-\ii\phi} - \frac{r'}{r^2}e^{-\ii\phi}\Big)\right).
	\]
	
	\noindent Hence,
	\[
		\theta' = r^{-2}
		\quad\Leftrightarrow\quad
		\theta(\phi) = \int \frac{d\phi}{r^2(\phi)}.
	\]
	
	\noindent The values $\cos_\Omega\theta$ and $\sin_\Omega\theta$ are equal to $r^{-1}\cos\phi$ and $r^{-1}\sin\phi$ correspondingly. The functions $\cos_{\Omega^\polar}$, $\sin_{\Omega^\polar}$ and the corresponding angle $\theta^\polar\leftrightarrow\theta$ can be found explicitly using differentiation formulae:
	\[
		\begin{array}{lcl}
			\cos_{\Omega^\polar}\theta^\polar = (\sin_\Omega\theta)'_\theta = y'_\phi/\theta'_\phi = r^2 y'_\phi,
			&\Rightarrow&
			\cos_{\Omega^\polar}\theta^\polar = r\cos\phi - r'\sin\phi;\\
			\sin_{\Omega^\polar}\theta^\polar = -(\cos_\Omega\theta)'_\theta = -x'_\phi/\theta'_\phi = -r^2 x'_\phi
			&\Rightarrow&
			\sin_{\Omega^\polar}\theta^\polar = r\sin\phi + r'\cos\phi.
		\end{array}
	\]
	\noindent So, the curve $p+\ii q = (r+\ii r')e^{\ii\phi}$ describes the boundary of the polar set $\Omega^\polar$. Similarly,
	\[
		{\theta^\polar}' = \Re \big((r+\ii r')e^{\ii\phi}\ii((r-\ii r')e^{-\ii\phi})'\big) =
		\Re((r+\ii r')e^{\ii\phi}\ii e^{-\ii\phi}(-\ii (r+r''))
	\]
	
	\noindent Hence,
	\[
		{\theta^\polar}' = r^2+rr''
		\quad\Leftrightarrow\quad
		\theta^\polar(\phi) = \int (r^2(\phi)+r(\phi)r''(\phi))\,d\phi.
	\]
	
	
\end{enumerate}

\begin{remark}
	Formulae for extremals in the present paper are usually obtained in terms of convex trigonometry. According to item~\ref{item:omega_paramtric}, if the boundary of~$\Omega$ is given parametrically, then all these formulae can be easily rewritten to ones that does not use convex trigonometry at all.
\end{remark}

\begin{remark}
	If the boundary of $\Omega$ consists of finite number of parts of boundaries of sets $\Omega_1$, $\ldots$, $\Omega_N$, then graphs of functions $\cos_\Omega$ and $\sin_\Omega$ are glued of corresponding parts of graphs of functions $\cos_{\Omega_k}$, $\sin_{\Omega_k}$, $1\le k\le\Omega_N$. For example, a cut disc is considered in detail in Sec.~\ref{sec:yachts}
\end{remark}

\subsection{Smoothness of the functions $\cos_\Omega$ and $\sin_\Omega$}

Using the formulae of convex trigonometry obtained in the previous subsection, we are able to monitor smoothness of the convex trigonometric functions.

\begin{prop} 
\label{prop1}
	The following statements are equivalent: 
	\begin{itemize}
		\item[$(i)$]   the boundary of $\Omega$ is a regular $C^k$-curve; 
		\item[$(ii)$]  $\cos_\Omega$ and $\sin_\Omega$ are $C^k$ functions; 
		\item[$(iii)$] $s_{\Omega^\polar}$ is $C^k$ outside the origin. 
	\end{itemize}
	Here $k\ge 1$ is integer, $k=\infty$, or $k=\omega$.
\end{prop}

\begin{proof}
	$(i)\Rightarrow(ii)$. Let the boundary $\partial\Omega$ be given as $(x(s), y(s))$, where $x,y\in C^k$, and ${x'_s}^2+{y'_s}^2\ne 0$. Then $\theta'_s(s)\in C^{k-1}$ by~\eqref{eq:theta_by_s}. So $\theta(s)\in C^k$. Moreover, $\theta'_s\ne 0$, since ${x'_s}^2 + {y'_s}^2\ne 0$ and $0\in\mathrm{int}\,\Omega$. So the inverse function $s(\theta)$ has the same smoothness, $s(\theta)\in C^k$ by the inverse function theorem.  Thus $\cos_\Omega\theta=x(s(\theta))$ and $\sin_\Omega(\theta)=y(s(\theta))$ are $C^k$.
	
	$(ii)\Rightarrow(i)$. The functions $\cos_\Omega\theta$ and $\sin_\Omega\theta$ give $C^k$ parametrization of the boundary $\partial\Omega$, and ${\cos'_\Omega}^2\theta+{\sin'_\Omega}^2\theta = \cos_{\Omega^\polar}^2\theta^\polar+\sin_{\Omega^\polar}^2\theta^\polar\ne0$, since $0\in\mathrm{int}\,\Omega^\polar$.
	
	$(iii)\Rightarrow(i)$. Consider the support function $s_{\Omega^\polar}$ in the classic polar coordinates $x=\rho\cos\phi$, $y=\rho\sin\phi$. Since $s_{\Omega^\polar}$ is positively homogeneous, we have $s_{\Omega^\polar}=\rho g(\phi)$, where $g(\phi)>0$ as $0\in\mathrm{int}\,\Omega$. So $s_{\Omega^\polar}$ is $C^k$ iff $g$ is $C^k$. Since $g\ne 0$, this is equivalent to the fact that $1/g$ is $C^k$. So if $s_{\Omega^\polar}$ is $C^k$, then $x(\phi)=\cos\phi/g(\phi)$, $y(\phi)=\sin\phi/g(\phi)$ is a $C^k$-parametrization of the boundary $\partial\Omega$, and ${x'_\phi}^2+{y'_\phi}^2=(g^2 + g'^2)/g^4>0$. 
	
	$(i)\Rightarrow(iii)$. Consider a $C^k$-parametrization $(x(s),y(s))$ of $\partial\Omega$, ${x'_s}^2 + {y'_s}^2\ne 0$. In the (classic) polar coordinates $(\rho,\phi)$ we have the parametrization $(\rho(s),\phi(s))$, which is again $C^k$, since $0\in\mathrm{int}\,\Omega$. We know that $s_{\Omega^\polar}(x(s),y(s))\equiv 1$, so $g(\phi(s))\equiv 1/\rho(s)$. Since $\phi'_s = (xy'_s-x'_sy)/(x^2+y^2)$, we have $\phi'_s\in C^{k-1}$, $\phi\in C^k$, and $\phi'_s\ne 0$ as $0\in\mathrm{int}\,\Omega$. So the function $g(\phi)$ must be $C^k$ by the inverse function theorem.
\end{proof}

\begin{prop} 
\label{prop:w_k_p}
	For any $1\le p\le\infty$ and integer $k\ge 1$, the following statements are equivalent: 
	\begin{itemize}
		\item[$(i)$]   the boundary of $\Omega$ is a regular $W^k_p$-curve\footnote{The term ``regular $W^k_p$-curve'' means that there exists a $W^k_p$ parametrization~$(x(s),y(s))$ of the curve such that ${x'_s}^2+{y'_s}^2$ is separated from~0 for a.e.~$s$.}; 
		\item[$(ii)$]  $\cos_\Omega$ and $\sin_\Omega$ are $W^k_p$ functions (on $\R/2\mathbb{S}\Z$);
		\item[$(iii)$] $s_{\Omega^\polar}(\cos\phi,\sin\phi)$ is $W^k_p$ function (on $\R/2\pi\Z$).
\end{itemize}
\end{prop}

\begin{proof}
	$(i)\Rightarrow(ii)$. Let the boundary $\partial\Omega$ be given as $(x(s),y(s))$, where $x,y\in W^k_p[s_0;s_1]$, and ${x'_s}^2+{y'_s}^2\ge C>0$ for some constant $C>0$. Consider first the case $k=1$. In this case, $s'_\theta(\theta)=1/\theta'_s(s(\theta))\le 1/C$, so $s'(\theta)$ is a bounded function and $s(\theta)$ is Lipschitz continuous. Moreover, $s'$ vanishes on a set of zero measure. Hence the derivative $\cos'_\Omega\theta=x'_s(s(\theta))s'(\theta)$ is defined for a.e.\ $\theta$, since $\cos_\Omega \theta = x(s(\theta))$. Therefore, if $p<\infty$ we have
$$
		\int_{\theta(s_0)}^{\theta(s_1)} |\cos'_\Omega\theta|^p\,d\theta = 
		\int_{\theta(s_0)}^{\theta(s_1)} |x'_s(\theta(s))|(s'(\theta))^p\,d\theta 
		\le\frac{1}{C^{p-1}}\int_{\theta(s_0)}^{\theta(s_1)} |x'_s(\theta(s))|s'(\theta)\,d\theta =
		\frac{1}{C^{p-1}}\|x'_s\|_p.
$$
	\noindent In the case $p=\infty$, we have
	\[
		|\cos'_\Omega\theta| = |x'_s(\theta(s))|s'(\theta) \le \frac{1}{C}|x'_s(\theta(s))|.
	\]
	\noindent Hence, in the both cases, $\cos_\Omega\in W^1_p[\theta(s_0);\theta(s_1)]$. Similarly, $\sin_\Omega\in W^1_p[\theta(s_0);\theta(s_1)]$.
	
	Now, we consider the case $k\ge 2$. In this case, $\partial\Omega\in C^{k-1}$. Therefore, $\theta(s)$ is a $C^{k-1}$ function and $s(\theta)$ is $C^{k-1}$ too. Moreover, $\theta'_s$ is a continuous function separated from~$0$, so $s(\theta)$ is a bi-Lipschitzian function. Hence, $\theta'_s(s(\theta))$ is $W^{k-1}_p$,  since $\theta'_s(s)$ is $W^{k-1}_p$. Moreover, $0<C\le \theta'_s\le C'$ for some constants $C'\ge C>0$, and the function $a\mapsto 1/a$ is smooth on $[C;C']$. Hence, $s'(\theta)=1/\theta'(s(\theta))\in W^{k-1}_p$. Therefore, $s(\theta)$ is $W^{k}_p$. It remains to compute
	\[
		\frac{d^k}{d\theta^k}\cos_\Omega\theta = x^{(k)}(s(\theta))(s'(\theta))^k + x'_s(s(\theta))s^{(k)}(\theta) + G(\theta)
	\]
	\noindent where $G$ is a continuous function. So $\cos_\Omega\in W^k_p$ and similarly $\sin_\Omega\in W^k_p$.
	
	The other proofs $(ii)\Rightarrow(i)$, 	$(i)\Rightarrow(iii)$, and $(iii)\Rightarrow(i)$ are similar to the corresponding proofs for Proposition~\ref{prop1}.
\end{proof}

\begin{remark}
\label{rm:smoothness_on_boundary_interval}
	Suppose that an open interval $l$ on $\partial\Omega$ has smoothness $W^k_p$ (or $C^k$). Then the functions $\cos_\Omega$ and $\sin_\Omega$ have the same smoothness while the point $P_\theta$ moves along $l$.
\end{remark}

We already know, that if $\partial\Omega$ is $C^1$ then the function $\theta^\polar(\theta)$ is continuous. Let us extend this property:

\begin{corollary}
\label{cor:curvature}
	If $\partial\Omega$ is a regular $C^k$ (or $W^k_p$) curve for $k\ge 2$ and $1\le p\le+\infty$, then the function $\theta^\polar(\theta)$ is $C^{k-1}$ (or $W^{k-1}_p$) and
\begin{equation}\label{eq:dtheta0}
		\frac{d\theta^\polar}{d\theta} = 
				\cos'_\Omega\theta\sin''_\Omega\theta - \sin'_\Omega\theta\cos''_\Omega\theta.
\end{equation}
	\noindent The curvature $\kappa$ of $\partial\Omega$ at the point $(\cos_\Omega\theta,\sin_\Omega\theta)$ is given by the formula
	\[
		\kappa(\theta) = (\cos_{\Omega^\polar}^2\theta^\polar + \sin_{\Omega^\polar}^2\theta^\polar)^{-3/2} \frac{d\theta^\polar}{d\theta}.
	\]
\end{corollary}

\begin{proof}
	The pair $(\cos'_\Omega\theta,\sin'_\Omega\theta)$ is a $C^{k-1}$ (or $W^{k-1}_p$) parametrization of $\partial\Omega^\polar$. Using the inverse polar change of coordinates (see Section~\ref{subsec:CT_properties}), we obtain formula~\eqref{eq:dtheta0}. Using it, we see that ${\theta^\polar}'(\theta)$ is $C^{k-2}$ (or $W^{k-2}_p$) and $\theta^\polar(\theta)$ is $C^{k-1}$ (or~$W^{k-1}_p$).
	
	The curvature can be computed in a  standard way: put $\gamma(\theta)=(\cos_\Omega\theta,\sin_\Omega\theta)$, then $\gamma'=(-\sin_{\Omega^\polar}\theta^\polar,\cos_{\Omega^\polar}\theta^\polar)$ and $\gamma''=-\gamma\,d\theta^\polar/d\theta$. It remains to compute $\kappa = (\gamma'\times\gamma'')/|\gamma'|^3$.
\end{proof}

\section{Finsler geometry on the Lobachevsky hyperbolic plane}
\label{sec:lobachavsky}

Consider the following control system on the upper half-plane $L^2=\{(x,y)\in \R^2\mid y>0\}$:
\begin{equation}
\label{eq:lobachevsky:control_system}
	\begin{cases}
		\dot x = y u_1;\\
		\dot y = y u_2;
	\end{cases}
\end{equation}
\noindent where the 2-dimensional control $u=(u_1,u_2)$ belongs to a convex compact set $\Omega$ with $0\in\mathrm{int}\,\Omega$. Solutions of the time minimizing problem $T\to\inf$ for this system are Finsler geodesics on $L^2$, where the length of a tangent vector $(\xi,\eta)\in T_{(x,y)}L^2$ is given by\footnote{Here $\mu_\Omega=s_{\Omega^\polar}$ denotes the Minkowski function of $\Omega$.} $y\mu_{\Omega}(\xi,\eta)$. Alternatively, we can consider this problem as a left-invariant Finsler geometry on the Lie group $\mathrm{Aff}(\R)$ of proper affine transformations of the line $\R$. When $\Omega$ is a unit circle, we have the classical hyperbolic plane with Lobachevskian geometry.

In this section we  show an easy way of constructing Finsler geodesics on $L^2$ by the convex trigonometry. Using the Pontryagin maximum principle (PMP), we get the following generalized Hamiltonian:
\[
	\mathcal{H} = pyu_1 + qyu_2,
\]
\noindent where $p$ and $q$ are adjoint variables. Maximum of $\mathcal{H}$ in $u\in\Omega$ can be written very conveniently in terms of the support function $s_\Omega$ of the set~$\Omega$:
\[
	H=\max_{u\in\Omega}\mathcal{H} = s_\Omega(py,qy).
\]
\noindent  Denote for short the arguments of $s_\Omega$ by $h_1=py$ and $h_2=qy$. Obviously $\frac{d}{dt}H=0$ and $H=\const$. We claim that the case $H=0$ is not possible. Indeed, if $H=0$, then $h_1=h_2=0$. Since $y>0$, we get $p=q=0$, which contradicts to the Pontryagin maximum principle.

So $H>0$. Consequently the point $(h_1,h_2)$ moves along the boundary of the polar set $\Omega^\polar$ stretched by $H$ times:
\[
	(h_1,h_2)\in H\partial\Omega^\polar.
\]

Let us now find the ``velocity'' of the point $(h_1,h_2)$. Substituting $\dot p=-\mathcal{H}'_x=0$ and $\dot q=-\mathcal{H}'_y=-pu_1-qu_2$, we get
\[
	\dot h_1 = h_1u_2,\qquad\dot h_2 = -h_1u_1,
\]
and control can be found from the Pontryagin maximum condition $h_1u_1+h_2u_2\to\max_{u\in\Omega}$. Equations on $\dot h_1$ and $\dot h_2$ do not depend on the shape of $\Omega$, but solutions to the Pontryagin maximum condition highly depend on the shape of $\Omega$.

Nonetheless, extremals can be found simultaneously for all $\Omega$ by machinery of the convex trigonometry. Let us use the generalized polar change of coordinates (see Subsec.~\ref{subsec:CT_properties}):
\[
	h_1 = H\cos_{\Omega^\polar}\theta^\polar\qquad\mbox{and}\qquad h_2=H\sin_{\Omega^\polar}\theta^\polar,
\]
\noindent since $H=s_{\Omega}(h_1,h_2)$. Moreover, if $(u_1,u_2)$ is an optimal control, then $h_1u_1+h_2u_2=H$ and $u\in\partial\Omega$. Hence, using the generalized Pythagorean identity (see Subsec.~\ref{subsec:CT_properties}), we get
\[
	u_1=\cos_\Omega\theta\qquad\mbox{and}\qquad u_2=\sin_\Omega\theta
\]
\noindent for an angle $\theta\leftrightarrow\theta^\polar$. 

Let us emphasize that the optimal control $u=(\cos_\Omega\theta,\sin_\Omega\theta)$  can be easily recovered from the angle $\theta^\polar(t)$ by the relation $\theta\leftrightarrow\theta^\polar$. Alternatively, we may use the differentiation formulae (see Subsec.~\ref{subsec:CT_properties}):
\[
	\dot{\theta^\polar}u_1=\dot{\theta^\polar}\sin'_{\Omega^\polar}\theta^\polar = \frac{d}{dt}\sin_{\Omega^\polar}\theta^\polar
	\quad\mbox{and}\quad
	\dot{\theta^\polar}u_2=-\dot{\theta^\polar}\cos'_{\Omega^\polar}\theta^\polar = -\frac{d}{dt}\cos_{\Omega^\polar}\theta^\polar.
\]
\noindent Thus we find $u_1$ and $u_2$ dividing by $\dot{\theta^\polar}$. 

So the last thing we need is to find $\theta^\polar$. Let us compute the derivative of $\theta^\polar$. According to the inverse polar change of coordinates (see Subsec.~\ref{subsec:CT_properties}) for a.e.\ $t$ we have
\[
	\dot\theta^\polar = \frac{h_1\dot h_2-h_2\dot h_1}{s_\Omega^2(h_1,h_2)} = -\frac{H^2\cos^2_{\Omega^\polar}\theta\cos_\Omega\theta + H^2\cos_{\Omega^\polar}\theta^\polar\sin_{\Omega^\polar}\theta^\polar\sin_{\Omega}\theta}{H^2} = -\cos_{\Omega^\polar}\theta^\polar.
\]

First, consider the case $p\ne 0$. We have 
\[
	y = \frac{h_1}{p} = \frac{H}{p}\cos_{\Omega^\polar}\theta^\polar.
\]
\noindent Let us find $x$. Since $\dot\theta^\polar=-\cos_{\Omega^\polar}\theta^\polar$, the derivative $\dot\theta^\polar$ is continuous. Hence $\theta^\polar\in C^2$. Thus, the functions $\cos_{\Omega^\polar}\theta^\polar(t)$ and $\sin_{\Omega^\polar}\theta^\polar(t)$ are Lipschitz continuous (see Subsec.~\ref{subsec:CT_properties}), and using \eqref{eq:lobachevsky:control_system} we get
\[
		\dot x = yu_1=\frac{H}{p}\cos_{\Omega^\polar}\theta^\polar\cos_\Omega\theta = 
			-\frac{H}{p}\cos_\Omega\theta\dot{\theta^\polar} = -\frac{H}{p}\frac{d}{dt}\sin_{\Omega^\polar}\theta^\polar;
\]
\noindent since $(\cos_{\Omega^\polar}\theta^\polar)'=-\sin_\Omega\theta$ and $(\sin_{\Omega^\polar}\theta^\polar)'=\cos_\Omega\theta$ by the differentiation formulae (see Subsec.~\ref{subsec:CT_properties}). Recall that $H=\const$ and $p=\const$, so
\[
	x=x_0 - \frac{H}{p}\sin_{\Omega^\polar}\theta^\polar.
\]

Thus, for the case $p\ne 0$ we have proved that Finsler geodesics on the Lobachevsky plane coincide with parts of boundary of the polar set $\Omega^\polar$, which is stretched $\frac{H}{|p|}$ times, rotated $\pm90^\polar$ (depending on the sign of $p$) and moved horizontally to the distance $x_0$. These geodesics are naturally called \textit{horizontal}.

A natural parametrization of geodesics is given by solutions of the equation
\begin{equation}
\label{eq:lobachevsky:natural_param}
	\dot{\theta^\polar} = -\cos_{\Omega^\polar}\theta^\polar. 
\end{equation}
\noindent The topological structure of solutions to this equation is very simple. The point $Q_{\theta^\polar}=(\cos_{\Omega^\polar}\theta^\polar,\sin_{\Omega^\polar}\theta^\polar)$ moves along $\partial\Omega^\polar$. The intersection points of $\partial\Omega^\polar$ with the vertical axis are fixed points. In the right half-plane, $Q_{\theta^\polar}$ moves clockwise, and in the left half-plane, $Q_{\theta^\polar}$ moves counterclockwise (see Fig. \ref{fig:lobachevsky}). Note that if the set $\Omega$ is not symmetric w.r.t. the origin, then the distance function on $L^2$ is not symmetric too, $\mathrm{dist}(A,B)\ne\mathrm{dist}(B,A)$. So, if $p\ne 0$, then geodesics going to the left and to the right are different. In the case, when the set $\Omega$ is symmetric, $L^2$ becomes a metric space and geodesics going to the left and to the right are represented by the same curves.

\begin{figure}[ht]
	\centering
	\begin{subfigure}{0.18\textwidth}
		\includegraphics[width=\textwidth]{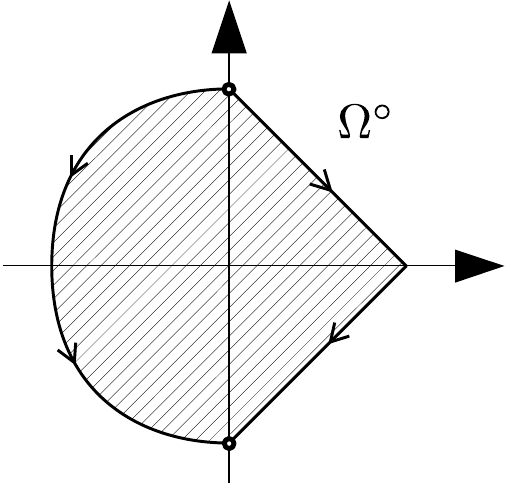}
	\end{subfigure}
	\hspace{1cm}
	\begin{subfigure}{0.45\textwidth}
		\includegraphics[width=\textwidth]{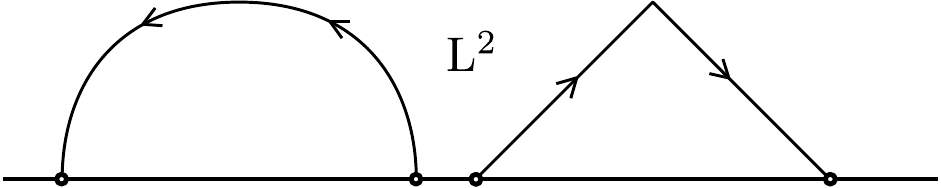}
	\end{subfigure}
	\caption{Examples of the polar set $\Omega^\polar$ and Finsler geodesics on $L^2$.}
	\label{fig:lobachevsky}
\end{figure}

Equation \eqref{eq:lobachevsky:natural_param} can be explicitly integrated in many cases. For example, if $\Omega$ is a convex polygon, then $\Omega^\polar$ is a polygon too and the functions $\cos_{\Omega^\polar}\theta^\polar$ and $\sin_{\Omega^\polar}\theta^\polar$ are linear in $\theta^\polar$ while the point $Q_{\theta^\polar}$ moves along an edge of $\Omega^\polar$. So in this case the equation $\dot{\theta^\polar} = -\cos_{\Omega^\polar}\theta^\polar$ on an edge takes the form $\dot\theta^\polar=a\theta^\polar + b$ and can be integrated easily. If $\Omega$ is an ellipse (see Subsec.~\ref{subsec:explicit_sets}), then equation~\eqref{eq:lobachevsky:natural_param} takes the form
\[
	\dot s = \frac{-a\cos s - x_0}{ab - y_0a\sin s+x_0b\cos s},
\]
\noindent which can be integrated explicitly by elementary functions.

Now consider the last case $p=0$. We have $h_1\equiv 0$ and $h_2=\const\ne 0$. If $h_2>0$, then the optimal control $u=(u_1,u_2)\in\partial \Omega$ has maximal second coordinate $u_2$, and $u$ is the upper point of $\Omega$ (if it is unique) or moves arbitrarily along the upper edge of $\Omega$. The case $h_2<0$ is similar. It is easy to see that, in both cases, the described control is optimal, since the coordinate $y$ takes maximal (or minimal) possible value at any given time~$T$. These geodesics are naturally called \textit{vertical}.

All vertical geodesics moving up from a given point $(x_0,y_0)\in L^2$ form a domain in~$L^2$ bounded by parts of two lines $(y-y_0)/(x-x_0)=u_2^0/u_1^0$, $y\ge y_0$, which appear when the control $u$ is constant, $u(t)=(u^0_1,u^0_2)$, and sits at the left (or right) end of the upper edge of $\Omega$. The similar fact holds for vertical geodesics moving down.

Summarizing, we get the following full description of geodesics in the Finsler problem on the Lobachevsky hyperbolic plane:

\begin{thm}
\label{thm:FinslerLobachevsky}
	There are two types of geodesics in the sub-Finsler problem on the Lobachevsky hyperbolic plane~\eqref{eq:lobachevsky:control_system}:
	
	\begin{enumerate}
	\item[$(h)$] Horizontal geodesics come from the polar set~$\Omega^\polar$. To obtain a horizontal geodesic one should $(i)$ rotate~$\Omega^\polar$ by $\pm 90^\circ$, $(ii)$ stretch it $\lambda>0$ times, $(iii)$ take the upper part of the boundary, and $(iv)$ move this part horizontally to any distance. If $\Omega^\polar$ was rotated clockwise then the motion on the geodesic is counterclockwise and vice versa (see Fig.~$\ref{fig:lobachevsky}$).
	
	\item[$(v)$] To obtain a vertical geodesic moving up one should choose an arbitrary measurable control $u(t)\in\Omega$ having maximal possible $u_2$-coordinate in~$\Omega$ for a.e.~$t$. Then the corresponding vertical geodesic starting at $(x_0,y_0)$ has the following form:
	\[
		x(t)=x_0+y_0\int_0^t e^{u_2s}u_1(s)\,ds;\qquad y(t)=y_0 e^{u_2t}.
	\]
	Vertical geodesics moving down are constructed in the same way by taking $u(t)\in\Omega$ having minimal possible $u_2$-coordinate in~$\Omega$ for all $t$.
	\end{enumerate}
\end{thm}

\section{Second order ODEs containing functions of convex trigonometry} 
\label{sec:second-order}

A lot of problems considered below can be reduced to solving a differential inclusion containing functions $\cos_\Omega$ and $\sin_\Omega$ of the following form\footnote{Results similar to those obtained in the present section are also fulfilled for the inclusion $\ddot \theta(t) \in -\mathcal{U}'(\theta(t))$ by the bipolar theorem.}:
\begin{equation}
\label{eq:general_ct_ode}
	\ddot \theta^\polar(t) \in -\mathcal{U}'(\theta^\polar(t))\mbox{ for a.e. }t,
\end{equation}
\noindent where $\theta^\polar(t)\in W^2_\infty$ (i.e.,\ $\dot\theta^\polar$ is a Lipschitz continuous functions), $\mathcal{U}(\theta^\polar)=f(\cos_{\Omega^\polar}\theta^\polar,\sin_{\Omega^\polar}\theta^\polar)$, $f(p,q)$ is a sufficiently smooth function, and the symbol $\mathcal{U}'$ denotes
\[
	\mathcal{U}'(\theta^\polar)=\Big\{-f_p\sin_\Omega\theta + f_q\cos_\Omega\theta\quad\mbox{for all }\theta\leftrightarrow\theta^\polar\Big\}.
\]
\noindent We use the previous notation, since $\mathcal{U}(\theta^\polar)$ has derivative for a.e.\ $\theta^\polar$ which is equal to $-f_p\sin_\Omega\theta + f_q\sin_\Omega\theta$ by the differentiation formulae (see Subsec.~\ref{subsec:CT_properties}). The derivative $\mathcal{U}'(\theta^\polar)$ certainly exists if there exists a unique $\theta\leftrightarrow\theta^\polar$ (up to the period $2\mathbb{S}^\polar=2\mathbb{S}(\Omega^\polar)$). In the opposite case, $\mathcal{U}(\theta^\polar)$ has left and right derivatives, and the interval between them we denoted by $\mathcal{U}'(\theta^\polar)$. In the last case, the left and right derivatives coincide iff $df(\cos_{\Omega^\polar}\theta^\polar,\sin_{\Omega^\polar}\theta^\polar)\perp(\cos_{\Omega^\polar}\theta^\polar,\sin_{\Omega^\polar}\theta^\polar)$. 

Usually, if $\mathcal{U}'(\theta^\polar)$ consists of  a unique element for some $\theta^\polar$, $\mathcal{U}'(\theta^\polar)=\{a\}$, we will simply write $\mathcal{U}'(\theta^\polar)=a$ for short.

\begin{prop}
\label{prop:existance_for_general_ct_ode}
	For any given initial data $\theta^\polar_0,\theta^\polar_1\in\R$, there exists a solution $\hat\theta^\polar(t)$ to~\eqref{eq:general_ct_ode} satisfying $\hat\theta^\polar(0)=\theta^\polar_0$ and $\dot{\hat\theta}^\polar(0)=\theta^\polar_1$ and defined for all $t\in\R$.
\end{prop}

\begin{proof}
	The multivalued function $\mathcal{U}'$ is upper semicontinuous at any point $\theta^\polar$ in the following sense: for any $\theta^\polar\in\R$ and $\varepsilon>0$ there exists $\delta>0$ such that $\mathcal{U}'(\tilde\theta^\polar)$ belongs to the $\varepsilon$-neighborhood of $\mathcal{U}'(\theta^\polar)$ for any $\tilde\theta^\polar\in\R$ in the $\delta$-neighborhood of $\theta^\polar$ (this fact follows immediately from the differentiation formulae in Subsec.~\ref{subsec:CT_properties}). Moreover, the set $\mathcal{U}'(\theta^\polar)$ is a non empty convex compactum for any $\theta^\polar$. Hence, for any initial data $\hat\theta^\polar(0)=\theta^\polar_0$, $\dot{\hat\theta}^\polar(0)=\theta^\polar_1$, there exists a solution to~\eqref{eq:general_ct_ode} in a neighborhood of $t=0$ by Filippov's theorem (see \cite{Filippov}). Moreover, any solution can be extended to the whole line $t\in\R$, since the function $\mathcal{U}'$ is bounded.
\end{proof}

Nonetheless, uniqueness of solution to~\eqref{eq:general_ct_ode} (with given initial data) may fail. Despite the fact that $\mathcal{U}(\theta^\polar)$ has derivative w.r.t.\ $\theta^\polar$ for a.e.\ $\theta^\polar$, existence of the derivative of function $\mathcal{U}(\hat\theta^\polar(t))$ w.r.t.\ $t$ needs additional verification. 

In the present section we will prove some general facts about differential inclusion~\eqref{eq:general_ct_ode}. It appears that uniqueness may fail only in very specific situations. Moreover, these situations correspond to singular extremals in optimal control problems considered below. 

\subsection{First integral of energy}

\begin{thm}
\label{thm:general_energy_integral}
The function
	$\mathbb{E}=\frac12(\dot\theta^\polar)^2+\mathcal{U}(\theta^\polar)$ is a first integral for~\eqref{eq:general_ct_ode}, i.e.,\ $\mathbb{E}$ is constant along any solution $\hat\theta^\polar(t)$ to \eqref{eq:general_ct_ode}. Moreover, if $\tilde\theta^\polar(t)$ is an arbitrary $C^1$-function (on an interval), $\dot{\tilde\theta}^\polar(t)\ne 0$ for all $t$, and~$\mathbb{E}$ is constant along $\tilde\theta^\polar$, then $\tilde\theta^\polar$ belongs to $W^2_\infty$ on the interval and it is a solution to~\eqref{eq:general_ct_ode}.
\end{thm}

\begin{proof}
	We start with the first part of the theorem. Denote $Z=\{t:{\dot{\hat\theta}^\polar}(t)=0\}$. For a.e.~$\theta^\polar$ there exists a unique $\theta\leftrightarrow\theta^\polar$, since $\partial\Omega^\polar$ has a countable number of corners at most. Hence, for a.e.\ $t\not\in Z$ there exists a unique $\hat\theta(t)\leftrightarrow\hat\theta^\polar(t)$ and the set $\mathcal{U}'(\hat\theta^\polar(t))$ contains a unique element, which must coincide with $-\ddot{\hat\theta}^\polar$. Therefore, the function $\mathcal{U}(\hat\theta^\polar(t))$ has derivative for a.e.\ $t\not\in Z$, and
	\[
		\dot{\mathbb E} = \ddot{\hat\theta}^\polar(t)\dot{\hat\theta}^\polar(t) + \dot{\hat\theta}^\polar(t)\mathcal{U}'(\hat\theta^\polar(t))=0\quad\mbox{for a.e.\ }t\not\in Z.
	\]
	
	It remains to consider the case $t\in Z$. Since $\hat\theta^\polar\in W^2_\infty$, there exists $\ddot{\hat\theta}^\polar(t)$ for a.e.\ $t\in Z$. Moreover, for all $t\in Z$, $\hat\theta^\polar(t+\tau)-\hat\theta^\polar(t)=o(\tau)$. Hence, $\frac{d}{dt}\mathcal{U}(\hat\theta^\polar(t))=0$ for all $t\in Z$, since the functions $\cos_{\Omega^\polar}$ and $\sin_{\Omega^\polar}$ are Lipschitz continuous. Therefore, we have 
	\[
		\dot{\mathbb E}=\ddot{\hat\theta}^\polar(t)\dot{\hat\theta}^\polar(t) + \frac{d}{d\tau}\Big|_{\tau=0} \mathcal{U}(\hat\theta^\polar(t+\tau))=0
		\quad\mbox{for a.e.\ }t\in Z.
	\]
	
	\medskip
	
	Let us prove the second part. Consider a $C^1$-function $\tilde\theta^\polar$ such that $\dot{\tilde\theta}^\polar(t)\ne 0$  
	for all~$t$ on an interval. Assume that $\dot{\tilde\theta}^\polar(t)>0$ for all $t$ (the case $\dot{\tilde\theta}^\polar(t)<0$ for all $t$ is similar). In this case, $\dot{\tilde\theta}^\polar = \sqrt{2}(\mathbb{E}-\mathcal{U}(\tilde\theta^\polar))^{1/2}$. Therefore, $\dot{\tilde\theta}^\polar$ is locally Lipschitz continuous and $\tilde\theta^\polar\in W^2_{\infty,\mathrm{loc}}$. Hence $0=\dot{\mathbb{E}}=\dot{\tilde\theta}^\polar\ddot{\tilde\theta}^\polar + \mathcal{U}'(\tilde\theta^\polar)\dot{\tilde\theta}^\polar$. So $\ddot{\tilde\theta}^\polar = -\mathcal{U}'(\tilde\theta^\polar)$ for a.e.\ $t$, since $\dot{\tilde\theta}^\polar(t)\ne 0$. It remains to say that $\mathcal{U}'$ is a bounded function, therefore $\tilde\theta^\polar\in W^2_\infty$.
	
\end{proof}

\subsection{Uniqueness of solution}

Let us now fix some initial data $\theta^\polar_0,\theta^\polar_1\in\R$:
\begin{equation}
\label{eq:general_ct_ode_initial_data}
	\theta^\polar(0)=\theta^\polar_0;\qquad \dot\theta^\polar(0)=\theta^\polar_1.
\end{equation}
If $\partial\Omega^\polar$ is sufficiently smooth, then uniqueness of solution with the given initial data must hold due to the classical Picard theorem.

\begin{prop}
	\label{prop:uniqueness_of_solution_when_Omega_smooth}
	Let $l$ be an open segment on $\partial\Omega^\polar$ and $(\cos_{\Omega^\polar}\theta^\polar_0,\sin_{\Omega^\polar}\theta^\polar_0)\in l$. If $l$ is $ W^2_\infty$, then there exists a unique solution $\hat\theta^\polar(t)$ to~\eqref{eq:general_ct_ode} in a neighborhood of $t=0$ satisfying initial data~\eqref{eq:general_ct_ode_initial_data}.
\end{prop}

\begin{proof}
	Indeed, in this case, $\cos_{\Omega^\polar}$ and $\sin_{\Omega^\polar}$ are $W^2_\infty$-functions by Proposition~\ref{prop:w_k_p} (see Remark~\ref{rm:smoothness_on_boundary_interval}). Hence $\mathcal{U}\in W^2_\infty$ and $\mathcal{U}'\in W^1_\infty$, i.e.,\ $\mathcal{U}'$ is a (single-valued) Lipschitz continuous function. Therefore there exists a unique solution by the Picard theorem.
\end{proof}

Consequently, absence of uniqueness may be present only at points, where the boundary $\partial\Omega^\polar$ does not have $W^2_\infty$-smoothness. For example, if $\Omega^\polar$ is a polygon. Another example is the following one. If $1<p<2$, then the boundary of the set $\{x^p+y^p\le 1\}\subset\R^2$ is $W^2_q$-smooth for $q<1/(2-p)$, but is not $W^2_q$-smooth for $q>1/(2-p)$. So the previous proposition does not work here, and we will show absence of uniqueness for solution of Pontryagin maximum principle in these cases for all left-invariant sub-Finsler problems on all unimodular 3D Lie groups in Section~\ref{sec:3DLie}.

Let us now investigate the non-uniqueness phenomenon.

\begin{prop}
\label{prop:unique_solution_if_dot_theta_ne_zero}
	Suppose that $\theta^\polar_1\ne 0$. Then, there exists a unique solution~$\hat\theta^\polar(t)$ to~\eqref{eq:general_ct_ode} in a neighborhood of $t=0$ satisfying initial data~\eqref{eq:general_ct_ode_initial_data}. Moreover,
	\begin{equation}
	\label{eq:integrating_energy_theta_polar_derivative_not_null}
		\sqrt{2}t = \mathrm{sgn}\,{\theta^\polar_1}\int_{\theta^\polar_0}^{\hat\theta^\polar(t)}\frac{d\theta^\polar}{\sqrt{\mathbb{E} - \mathcal{U}(\theta^\polar)}},
	\end{equation}
	where $\mathbb{E}=\frac12(\theta^\polar_1)^2 + \mathcal{U}(\theta^\polar_0)$.
\end{prop}

\begin{proof}
	Obviously, $|\dot{\hat\theta}^\polar(t)|=\sqrt{2}(\mathbb{E}-\mathcal{U}(\hat\theta^\polar(t)))^{1/2}$ by Theorem~\ref{thm:general_energy_integral}. Since $\mathcal{U}(\theta^\polar)<\mathbb{E}$ for $\theta^\polar$ in a neighborhood of $\theta^\polar_0$, we have
	\[
		\frac{\dot{\hat\theta}^\polar(t)}{\sqrt{\mathbb{E}-\mathcal{U}(\hat\theta^\polar(t))}} = \mathrm{sgn}\,{\theta^\polar_1} \sqrt{2}
	\]
	for $t$ in a neighborhood of $0$. Hence, \eqref{eq:integrating_energy_theta_polar_derivative_not_null} must hold. By the inverse function theorem, equation~\eqref{eq:integrating_energy_theta_polar_derivative_not_null} has a unique solution, since derivative (w.r.t. $\theta^\polar$) of the right-hand side is equal to $(\mathbb{E}-\mathcal{U}(\theta^\polar))^{-1/2}$, which is continuous and does not vanish in a neighborhood of~$\theta^\polar_0$.
\end{proof}

Hence, uniqueness may fail only at points, where $\dot\theta^\polar=0$. Moreover, even if $\dot\theta^\polar=0$ this does not guarantee absence of uniqueness (as for example in the case $\partial\Omega^\polar\in W^2_\infty$).

\begin{prop}
\label{prop:uniqueness_if_zero_not_in_U_derivative}
	Let $\theta^\polar_0\in\R$ and $\theta^\polar_1=0$. If $0\not\in\mathcal{U}'(\theta^\polar_0)$, then there exists a unique solution~$\hat\theta^\polar(t)$ to~\eqref{eq:general_ct_ode} in a neighborhood of $t=0$ satisfying initial data~\eqref{eq:general_ct_ode_initial_data}. This solution is even $\hat\theta^\polar(t)=\hat\theta^\polar(-t)$ for $t$ lying in the neighborhood and satisfies
	\begin{equation}
	\label{eq:integrating_energy_theta_polar_derivative_null_U_not_null}
		\sqrt{2}|t| \,\mathrm{sgn}\,\mathcal{U'}(\theta^\polar_0) = \int_{\theta^\polar_0}^{\hat\theta^\polar(t)}\frac{d\theta^\polar}{\sqrt{\mathbb{E} - \mathcal{U}(\theta^\polar)}}.
	\end{equation}
\end{prop}

\begin{proof}
	Since $\theta^\polar_1=0$, we have $\mathbb{E}=\mathcal{U}(\theta^\polar_0)$. Let $\mathcal{U}'(\theta^\polar_0)<0$ (the case $\mathcal{U}'(\theta^\polar_0)>0$ is similar). Then $\mathcal{U}'(\theta^\polar)\subset[-a;-b]$ in a neighborhood of $\theta^\polar_0$ for some $a>b>0$. Hence, using~\eqref{eq:general_ct_ode}, we obtain $\ddot{\hat\theta}^\polar(t)>0$, $\dot{\hat\theta}^\polar(t)>0$, and $\hat\theta^\polar(t)>\theta^\polar$ for small $t>0$. Thereby
	\[
		\sqrt{2}t = \int_{\theta^\polar_0}^{\hat\theta^\polar(t)}\frac{d\theta^\polar}{\sqrt{\mathbb{E} - \mathcal{U}(\theta^\polar)}}
		\quad\mbox{for}\quad t\ge 0.
	\]
	This integral has singularity at the left end, but it is finite, since $\mathbb{E}-\mathcal{U}(\theta^\polar)>b\theta^\polar$. Put $\hat\theta^\polar(t)=\theta^\polar_0+\varphi^2$, $\varphi\ge0$:
	\[
		\sqrt{2}t = \int_{\theta^\polar_0}^{\theta^\polar_0+\varphi^2}\frac{d\theta^\polar}{\sqrt{\mathbb{E} - \mathcal{U}(\theta^\polar)}}\stackrel{\mathrm{def}}{=} F(\varphi).
	\]
	The function $F$ is absolutely continuous, $F(0)=0$, and 
	\[
		\frac{2}{\sqrt{a}}\le F'(\varphi) = \frac{2\varphi}{\sqrt{\mathbb{E} - \mathcal{U}(\theta^\polar_0+\varphi^2)}} \le \frac{2}{\sqrt{b}}.
	\]
	Hence, $F$ is a strictly monotone increasing function, and there exists a unique solution $\varphi=\varphi(t)\ge 0$ to the equation $F(\varphi)=\sqrt{2}t$ for small $t\ge 0$. Moreover, $\varphi(t)$ is a bi-Lipschitz continuous function, and $\hat\theta^\polar(t)=\theta^\polar_0+\varphi^2(t)$.
	
	For $t\le 0$, we have
	\[
		-\sqrt{2}t = \int_{\theta^\polar_0}^{\hat\theta^\polar(t)}\frac{d\theta^\polar}{\sqrt{\mathbb{E} - \mathcal{U}(\theta^\polar)}}\mbox{ for }t\ge 0.
	\]
	Hence $\hat\theta^\polar(-t)=\hat\theta^\polar(t)$.
\end{proof}

The condition $0\not\in\mathcal{U}'(\theta^\polar_0)$ is equivalent to the following one: the vector $df(\cos_{\Omega^\polar}\theta^\polar_0,\sin_{\Omega^\polar}\theta^\polar_0)$ is not null and defines a non-supporting line to $\Omega^\polar$ at $(\cos_{\Omega^\polar}\theta^\polar_0,\sin_{\Omega^\polar}\theta^\polar_0)$.

\subsection{Absence of uniqueness}
\label{subsec:nonuniqueness_general_ct_ode}

So uniqueness may fail only if $\theta^\polar_1=0$ and $0\in\mathcal{U}'(\theta^\polar_0)$. Let us now investigate this situation. Obviously, if $\mathcal{U}(\theta^\polar)\ge\mathcal{U}(\theta^\polar_0)$ in a neighborhood of $\theta^\polar_0$, then a unique solution with $\theta^\polar_1=0$ is $\hat\theta^\polar(t)\equiv\theta^\polar_0$. Indeed, $(\dot{\hat\theta}^\polar(t))^2\le 0$ by Theorem~\ref{thm:general_energy_integral}. It remains to describe solutions for the cases, when $\mathcal{U}(\theta^\polar)<\mathcal{U}(\theta^\polar_0)$ in a left or in a right neighborhood\footnote{We omit here some rare cases, when $\mathcal{U}(\theta^\polar)\le 0$ in right (or left) neighborhood of $\theta^\polar_0$, but inequality $\mathcal{U}(\theta^\polar)<0$ does not hold in any left (of right) neighborhood of $\theta^\polar_0$.} of $\theta^\polar_0$.

The following theorem describes behaviour of solutions  only for $t\ge 0$. The case $t\le 0$ is similar, since if $\hat\theta^\polar(t)$ is a solution to~\eqref{eq:general_ct_ode}, \eqref{eq:general_ct_ode_initial_data} and $\theta^\polar_1=0$, then $\hat\theta^\polar(-t)$ is a solution to~\eqref{eq:general_ct_ode}, \eqref{eq:general_ct_ode_initial_data} as well.

\begin{thm}
\label{thm:absense_of_uniqueness_for_general_ct_ode}
	Let $\theta^\polar_0\in\R$, $\theta^\polar_1=0$, and $0\in\mathcal{U}'(\theta^\polar_0)$. Put $\mathbb{E}=\mathcal{U}(\theta^\polar_0)$. Suppose that $\mathcal{U}(\theta^\polar)<\mathbb{E}$ either $(r)$ in a right, or $(l)$ in a left, or $(b)$ in both left and right neighborhoods of $\theta^\polar_0$. Then there exists $\tau>0$ with the following properties.
	
	\begin{itemize}
		\item[$(1)$] Let $\hat\theta^\polar(t)$ be a solution to~\eqref{eq:general_ct_ode}, \eqref{eq:general_ct_ode_initial_data} for $t\ge 0$. Put
		\[
			t_+=\sup \{t\ge 0 : \hat\theta^\polar(s)=\theta^\polar_0\mbox{ for all }s\in[0;t]\}.
		\]
			If there exists a solution $\hat\theta^\polar$ with $t_+<\infty$, then $\dot{\hat\theta}^\polar(t)$ for $t\in(t_+;t_++\tau)$ is not null\footnote{Hence, it must be positive in the case $(r)$ and negative in the case $(l)$.}, and $(\mathbb{E} - \mathcal{U}(\theta^\polar))^{-1/2}\in L_1(\theta^\polar_0,\theta^\polar_0+\varepsilon)$ for some $\varepsilon\ne 0$, $\mathrm{sgn}\,\varepsilon=\mathrm{sgn}\,\dot{\hat\theta}^\polar(t_++0)$. Moreover, for any solution $\hat\theta^\polar$ with $t_+<\infty$, we have
		\begin{equation}
		\label{eq:integrating_energy_theta_polar_derivative_null_U_null}
			\sqrt{2}(t-t_+)\mathrm{sgn}\,\dot{\hat\theta}^\polar(t_++0) = \int_{\theta^\polar_0}^{\hat\theta^\polar(t)}\frac{d\theta^\polar}{\sqrt{\mathbb{E} - \mathcal{U}(\theta^\polar)}}\quad\mbox{for}\quad t\in(t_+;t_++\tau].
		\end{equation}

		\item[$(2)$] If $(r)$ (or $(b)$) and $(\mathbb{E} - \mathcal{U}(\theta^\polar))^{-1/2}\in L_1(\theta^\polar_0,\theta^\polar_0+\varepsilon)$ for some $\varepsilon>0$, then for any $t_+\ge 0$ there exists a solution $\hat\theta^\polar(t)$ to~\eqref{eq:general_ct_ode}, \eqref{eq:general_ct_ode_initial_data} for $t\ge 0$ such that $\hat\theta^\polar(t)=\theta^\polar_0$ for $t\in[0;t_+]$, $\dot{\hat\theta}^\polar(t)$ is positive for $t\in(t_+;t_++\tau)$, and~\eqref{eq:integrating_energy_theta_polar_derivative_null_U_null} holds. The remaining case  ($(l)$ or $(b)$) is similar with $\varepsilon<0$.
	\end{itemize}
\end{thm}

The geometric meaning of the case $0\in\mathcal{U}'(\theta^\polar_0)$ is the following. The constant $\hat\theta^\polar(t)\equiv\theta^\polar_0$ is always a solution. But if, for example, $\mathcal{U}(\theta^\polar)<\mathcal{U}(\theta^\polar_0)$ in a right neighborhood of $\theta^\polar_0$ and the following improper integral
\[
	\int_{\theta^\polar_0}^{\theta^\polar_0+\varepsilon}\frac{d\theta^\polar}{\sqrt{\mathbb{E} - \mathcal{U}(\theta^\polar)}}
\]
is finite for some $\varepsilon>0$, then uniqueness fails. A solution stays at $\theta^\polar_0$ for an arbitrary time $t_+\ge 0$ and then goes to the right for at least $\tau>0$ time. So, there is a 1-parameter family (determined by the parameter $t_+\ge 0$) of solutions for $t\ge 0$ defined on $[0;t_++\tau]$. Each solution can then be  extended for $t\in[0;+\infty)$ by Filippov's theorem. If $\mathcal{U}(\theta^\polar)<\mathcal{U}(\theta^\polar_0)$ in a left neighborhood of $\theta^\polar_0$ and the improper integral is finite for some $\varepsilon<0$, then a solution can go also to the left (and there is another 1-parameter family of solutions for $t\in[0;t_++\tau]$, $t_+\ge 0$). A similar situation happens for $t\le 0$.

\begin{proof}[Proof of Theorem~$\ref{thm:absense_of_uniqueness_for_general_ct_ode}$]
	Denote $\Theta=\{\theta^\polar\in\R:\mathcal{U}(\theta^\polar)=\mathbb{E}\}$. Obviously, if $\dot{\hat\theta}^\polar(t)=0$, then $\hat\theta^\polar(t)\in\Theta$ by Theorem~\ref{thm:general_energy_integral}. Put
	\[
		\theta^\polar_r=\inf(\Theta\cap\{\theta^\polar>\theta^\polar_0\});
		\qquad
		\theta^\polar_l=\sup(\Theta\cap\{\theta^\polar<\theta^\polar_0\}).
	\]
	If $(r)$ or $(b)$, then $\theta^\polar_r>\theta^\polar_0$, and if $(l)$ or $(b)$, then $\theta^\polar_l<\theta^\polar_0$ by the hypothesis of the  theorem.

	Since $|\dot{\hat\theta}^\polar(t)|=\sqrt{2}(\mathbb{E}-\mathcal{U}(\hat\theta^\polar(t)))^{1/2}$ for any solution $\hat\theta^\polar$ and $\mathcal{U}$ is a bounded function, the derivative $\dot{\hat\theta}^\polar(t)$ is bounded by the constant $M=\sqrt{2}(\mathbb{E}-\min_{\theta^\polar}\mathcal{U}(\theta^\polar))^{1/2}$, i.e.,\ $|\dot{\hat\theta}^\polar(t)|\le M$ for all $t$. Denote
	\[
		\tau_l=(\theta^\polar_0-\theta^\polar_l)/M,\qquad \tau_r=(\theta^\polar_r-\theta^\polar_0)/M,
	\]
	and put $(l)$  $\tau=\tau_l$, $(r)$  $\tau=\tau_r$, or $(b)$  $\tau=\min\{\tau_l,\tau_r\}$. Obviously $\tau>0$.
	
	The proof of item 1 is based on the following key property of $\tau$: if $\hat\theta^\polar(t_1)=\hat\theta^\polar(t_2)=\theta^\polar_0$ and $|t_2-t_1|<\tau$, then $\hat\theta^\polar(t)=\theta^\polar_0$ for all $t$ between $t_1$ and $t_2$. Let us prove this property by contradiction. Let $t_1<t_2$. Suppose that there exists $s\in(t_1,t_2)$ such that $\hat\theta^\polar(s)\ne \theta^\polar_0$. Let for example $\hat\theta^\polar(s)>\theta^\polar_0$ (this is possible only if $(r)$ or $(b)$ is fulfilled). Put $T=\argmax_{t\in[t_1;t_2]}\hat\theta^\polar(t)$. Then $\hat\theta^\polar(T)>\theta^\polar_0$, $\dot{\hat\theta}^\polar(T)=0$, and $\hat\theta^\polar(T)\in\Theta$. Therefore, $\hat\theta^\polar(T)\ge \theta^\polar_r$ and we have the following contradiction:
	\[
		\theta^\polar_r-\theta^\polar_0 \le \hat\theta^\polar(T)-\hat\theta^\polar(t_1)\le M(T-t_1) < M\tau\le \theta^\polar_r-\theta^\polar_0.
	\]

	Now we are ready to prove item 1 of the theorem. Suppose that there exists a solution~$\hat\theta^\polar$ with $t_+<\infty$. By the definition of $t_+$, there exists $s\in(t_+;t_++\tau)$ such that $\hat\theta^\polar(s)\ne \theta^\polar_0$. We claim that if $\hat\theta^\polar(s)>\theta^\polar_0$, then $\dot{\hat\theta}^\polar(t)>0$ for all $t\in(t_+;t_++\tau)$. Let us prove this claim again by contradiction. Since $\hat\theta^\polar(s)>\theta^\polar_0$, $(r)$ or $(b)$ is fulfilled. Moreover, if the claim does not hold, then there exists $T\in(t_+;t_++\tau)$ such that $\dot{\hat\theta}^\polar(T)=0$. Therefore, $\hat\theta^\polar(T)\in\Theta$. Since $T-t_+<\tau$, we have $\hat\theta^\polar(T)=\theta^\polar_0$. Hence, $\hat\theta^\polar(t)=\theta^\polar_0$ for $t\in[t_+;T]$ by the key property of $\tau$, and we obtain a contradiction with the definition of~$t_+$. The case $\hat\theta^\polar(s)<\theta^\polar_0$ is similar. 
	
	So, $\dot{\hat\theta}^\polar(t)\ne 0$ and it does not change sign for all $t\in(t_+;t_++\tau)$. Let $\dot{\hat\theta}^\polar(t)>0$ for all $t\in(t_+;t_++\tau)$ (the opposite case is similar). Using Theorem~\ref{thm:general_energy_integral}, we obtain $\dot{\hat\theta}^\polar(t)=\sqrt{2}(\mathbb{E}-\mathcal{U}(\hat\theta^\polar(t)))^{1/2}>0$ for $t\in(t_+;t_++\tau)$. Hence 
	\[
		\frac{\dot{\hat\theta}^\polar(t)}{\sqrt{\mathbb{E}-\mathcal{U}(\hat\theta^\polar(t))}}=\sqrt{2}
		\qquad\mbox{for}\qquad t\in(t_+;t_++\tau)
	\]
	and
	\[
		\int_{t_+}^t \frac{\dot{\hat\theta}^\polar(s)\,ds}{\sqrt{\mathbb{E}-\mathcal{U}(\hat\theta^\polar(s))}} = \sqrt{2}(t-t_+)
		\qquad\mbox{for}\qquad t\in(t_+;t_++\tau].
	\]
	The change of variables $d\theta^\polar=\dot{\hat\theta}^\polar(s)\,ds$, $\theta^\polar=\hat\theta^\polar(s)$ is admissible, since the function $\hat\theta^\polar\in W^2_\infty$ is strictly monotone increasing for $t\in[t_+;t_++\tau]$. This change proves both $(\mathbb{E}-\mathcal{U}(\theta^\polar))^{-1/2}\in L_1(\theta^\polar_0,\theta^\polar_0+\varepsilon)$ for some $\varepsilon>0$ and~\eqref{eq:integrating_energy_theta_polar_derivative_null_U_null} for any solution $\hat\theta^\polar$ with~$t_+<\infty$.

	Let us now prove item 2 of the theorem. Suppose that $(\mathbb{E}-\mathcal{U}(\theta^\polar))^{-1/2}\in L_1(\theta^\polar_0,\theta^\polar_0+\varepsilon)$ for some $\varepsilon>0$ (the case $\varepsilon<0$ is similar). Then $(r)$ or $(b)$ is fulfilled, and the function
	\[
		G(\theta^\polar) \stackrel{\mathrm{def}}{=} \int_{\theta^\polar_0}^{\theta^\polar} \frac{d\psi}{\sqrt{\mathbb{E}-\mathcal{U}(\psi)}}
	\]
	is strictly increasing for $\theta^\polar\in[\theta^\polar_0;\theta^\polar_r]$. Moreover, $G(\theta^\polar_0)=0$ and $G(\theta^\polar_r)\ge\sqrt{2}\tau$, since $\mathbb{E}-\mathcal{U}(\theta^\polar)\le\frac12M^2$. Hence for any $t\in[0;\tau]$, there exists $\tilde\theta^\polar(t)$ such that $G(\tilde\theta^\polar(t))=\sqrt{2}t$. Moreover, since $G\in C^1(\theta^\polar_0,\theta^\polar_r)$ and $G'(\theta^\polar)\ne0$ for $\theta^\polar\in(\theta^\polar_0,\theta^\polar_r)$, we obtain $\tilde\theta^\polar\in C^1(0;\tau)$ and $\dot{\tilde\theta}^\polar=\sqrt{2}(\mathbb{E}-\mathcal{U}(\tilde\theta^\polar))^{1/2}>0$  by the inverse function theorem. Hence, $\tilde\theta^\polar$ is a solution to~\eqref{eq:general_ct_ode} by Theorem~\ref{thm:general_energy_integral}.	
	
	Since $\tilde\theta^\polar(0)=\theta^\polar_0$ and $\dot{\tilde\theta}^\polar(0)=0$, it remains to put
	\[
		\hat\theta^\polar(t)=
		\left\{\begin{array}{lll}
			\theta^\polar_0&\mbox{for}&t\in[0;t_+];\\
			\tilde\theta(t-t_+)&\mbox{for}&t\in(t_+;t_++\tau].
		\end{array}\right.
	\]
	Indeed, the function $\hat\theta$ is $C^1$ and its second derivative is $L_\infty$ and satisfies~\eqref{eq:general_ct_ode} for a.e.~$t$.
\end{proof}

So, uniqueness may fail only if a solution has zero speed $\theta^\polar_1=0$ at a point $\theta^\polar_0$ and $0\in\mathcal{U}'(\theta^\polar_0)$. Moreover, if $\partial\Omega^\polar$ is $W^2_\infty$ in a neighborhood of the point $(\cos_{\Omega^\polar}\theta^\polar_0,\sin_{\Omega^\polar}\theta^\polar_0)$ then uniqueness persists, since improper integral~\eqref{eq:integrating_energy_theta_polar_derivative_null_U_null} is not finite (or by Proposition~\ref{prop:uniqueness_of_solution_when_Omega_smooth}). Let us develop some easy to check condition that guarantees absence of uniqueness.

\begin{prop}
\label{prop:corner_guarantees_non_uniqueness}
	If $\mathcal{U}'(\theta^\polar_0+0)<0$, then $(\mathbb{E}-\mathcal{U}(\theta^\polar))^{-1/2}\in L_1(\theta^\polar_0,\theta^\polar_0+\varepsilon)$ for $\mathbb{E}=\mathcal{U}(\theta^\polar_0)$ and some $\varepsilon>0$. A similar result holds for the case $\mathcal{U}'(\theta^\polar_0-0)>0$.
\end{prop}

\begin{proof}
	Indeed, in this case, $\mathcal{U}(\theta^\polar)\le \mathbb{E} - \frac12\delta(\theta^\polar-\theta^\polar_0)$ for $\theta^\polar$ in a right neighborhood of~$\theta^\polar_0$.
\end{proof}

Let us give a geometric explanation of the pair of conditions $0\in\mathcal{U}'(\theta^\polar_0)$ and $\mathcal{U}'(\theta^\polar_0+0)<0$ (which together guarantee absence of uniqueness for solutions with $\theta^\polar_1=0$). Denote $Q=(\cos_{\Omega^\polar}\theta^\polar,\sin_{\Omega^\polar}\theta^\polar)\in\partial\Omega^\polar$. Obviously, if $df(Q)=0$, then $\mathcal{U}'(\theta^\polar_0)=\{0\}$. So we assume that $df(Q)\ne 0$. In this case, $df(Q)$ defines an (affine) hyperplane $\Pi=\{(p,q)\in\R^{2*}: \langle(p,q),df(Q)\rangle \le \langle Q,df(Q)\rangle\}$. The line $\partial\Pi$ is tangent to $\partial\Omega^\polar$ at $Q$ iff $0\in \mathcal{U}'(\theta^\polar_0)$. The set of all $\theta\leftrightarrow\theta^\polar_0$ has the form $[\theta_-;\theta_+] + 2\mathbb{S}\Z$, and the counterclockwise tangent ray at $Q$ to $\partial\Omega^\polar$ is defined by $\theta_+$: this ray is co-directed with $(-\sin_\Omega\theta_+,\cos_\Omega\theta_+)$. In these notations, $\mathcal{U}'(\theta^\polar_0+0)=\langle df(Q), (-\sin_\Omega\theta_+,\cos_\Omega\theta_+)\rangle <0$. Therefore, the ray belongs to the interior of the supporting hyperplane $\Pi$ at $Q$ to $\Omega^\polar$, defined by $df(Q)\ne0$.

\subsection{Admissible controls}
\label{subsec:theta_in_general_ct_ode}

In what follows, solutions to~\eqref{eq:general_ct_ode} often determine a lift of extremals into the cotangent bundle, and $\hat\theta(t)\leftrightarrow\hat\theta^\polar(t)$ determines control on extremals. So the last thing we want to develop in this section is to find all functions $\hat\theta(t)$ such that, for a.e. $t$, $\hat\theta(t)\leftrightarrow\hat\theta^\polar(t)$ and 
\[
	\ddot{\hat\theta}^\polar = df_p\sin_\Omega\hat\theta(t) - df_q\cos_\Omega\hat\theta(t)
\]
(measurable functions $\hat\theta(t)$ satisfying these two properties for a.e.\ $t$ will be called admissible).

First, suppose that $\theta^\polar_1\ne 0$. In this case, there exists a unique solution $\hat\theta^\polar(t)$ to~\eqref{eq:general_ct_ode}, \eqref{eq:general_ct_ode_initial_data} for $t$ in a neighborhood of $t=0$ by Proposition~\ref{prop:unique_solution_if_dot_theta_ne_zero}. Then $\dot{\hat\theta}^\polar(t)\ne 0$ in a neighborhood of $t=0$. Since for a.e.\ $\theta^\polar$ there exists a unique $\theta\leftrightarrow\theta^\polar$, we obtain that there exists a unique $\hat\theta(t)\leftrightarrow\hat\theta^\polar(t)$ for a.e.~$t$ in a neighborhood of~$t=0$. Hence there exists a unique (up to a set of zero measure) admissible $\hat\theta(t)$.

Now suppose that $\theta^\polar_1=0$. In this case, we want to examine, whether $\hat\theta^\polar(t)\equiv\theta^\polar_0$ is a solution, and if it is, we want to find an admissible $\hat\theta(t)$. Obviously,
\[
 	-df_p(Q_0)\sin_\Omega\hat\theta(t)+df_q(Q_0)\cos_\Omega\hat\theta(t)=-\ddot{\hat\theta}^\polar(t)=0,
\]
where we denote $Q_0=(\cos_{\Omega^\polar}\theta^\polar_0,\sin_{\Omega^\polar}\theta^\polar_0)$ for short. Since $\hat\theta(t)\leftrightarrow\theta^\polar_0$,
\[
	\cos_{\Omega^\polar}\theta^\polar_0\cos_\Omega\hat\theta(t) + \sin_{\Omega^\polar}\theta^\polar_0\sin_\Omega\hat\theta(t) = 1.
\]

Hence, $x=\cos_\Omega\hat\theta(t)$ and $y=\sin_\Omega\hat\theta(t)$ satisfy the following system of two linear equations:
\[
	\left\{
		\begin{array}{l}
			x \cos_{\Omega^\polar}\theta^\polar + y \sin_{\Omega^\polar}\theta^\polar = 1;\\
			x f_q (Q(t)) - y f_p(Q(t)) = 0;
		\end{array}
	\right.
	\quad\Leftrightarrow\quad
	\begin{pmatrix}
		\cos_{\Omega^\polar}\theta^\polar & \sin_{\Omega^\polar}\theta^\polar\\
		f_q (Q(t)) & -f_p(Q(t))
	\end{pmatrix}
	\begin{pmatrix}
		x\\y
	\end{pmatrix}
	=
	\begin{pmatrix}
		1\\0
	\end{pmatrix}.
\]
\begin{enumerate}
	\item If $df(Q_0)\not\perp Q_0$, then the system has a unique solution $(x_0,y_0)$, since the matrix in the left-hand side has non-zero determinant. Moreover, the point $(x_0,y_0)$ belongs to the supporting line to $\Omega$ determined by the first equation of the system. So, if $df(Q_0)\not\perp Q_0$, then there are two possibilities. 
	\begin{enumerate}
		\item\label{ct_ode_general_singular} $(x_0,y_0)\in\partial\Omega$ iff $0\in\mathcal{U}'(\theta^\polar_0)$. In this case, $\hat\theta^\polar(t)\equiv\theta^\polar_0$ is a solution to~\eqref{eq:general_ct_ode}. Moreover, there exists a unique (up to a set of zero measure) admissible $\hat\theta(t)\equiv\theta_0$ where $x_0=\cos_\Omega\theta_0$ and $y_0=\sin_\Omega\theta_0$. This case usually corresponds to singular extremals.
	
		\item $(x_0,y_0)\not\in\partial\Omega$ iff $0\not\in\mathcal{U}'(\theta^\polar_0)$. In this case, $\hat\theta^\polar(t)\equiv\theta^\polar_0$ is not a solution to~\eqref{eq:general_ct_ode}. By Proposition~\ref{prop:uniqueness_if_zero_not_in_U_derivative} there exists a unique solution~$\hat\theta^\polar(t)$ to~\eqref{eq:general_ct_ode}, \eqref{eq:general_ct_ode_initial_data}, and $\dot{\hat\theta}^\polar(t)\ne 0$ in a punctured neighborhood of $t=0$. Hence the admissible function $\hat\theta(t)$ is unique (again up to a zero measure) in this neighborhood.
	\end{enumerate}

	\item If $df(Q_0)\perp Q_0$, then $\mathcal{U}'(\theta^\polar_0)$ is a singleton. Indeed, $(f_p(Q_0),f_q(Q_0))=\lambda (-\sin_{\Omega^\polar}\theta^\polar_0,\cos_{\Omega^\polar}\theta^\polar_0)$ for some $\lambda\in\R$ and $\mathcal{U}'(Q_0)=\{-f_p(Q_0)\sin_\Omega\theta + f_q(Q_0)\cos_\Omega\theta\text{ for all }\theta\leftrightarrow\theta^\polar_0\}$. Hence by the generalized Pythagorean identity (see Subsec.~\ref{subsec:CT_properties}), we have $\mathcal{U}'(\theta^\polar_0)=\{\lambda\}$. There are two possibilities.

	\begin{enumerate}
		\item If $df(Q_0) \ne 0$, then $\mathcal{U}'(\theta^\polar_0)\ne 0$ and $\hat\theta^\polar(t)\equiv\theta^\polar_0$ is not a solution to~\eqref{eq:general_ct_ode}.
	
		\item\label{ct_ode_special_singular} If $df(Q_0) = 0$, then $\mathcal{U}'(\theta^\polar_0)=0$ and $\hat\theta^\polar(t)\equiv\theta^\polar_0$ is a solution to~\eqref{eq:general_ct_ode}. In this case, any measurable function $\hat\theta(t)$ satisfying $\hat\theta(t)\leftrightarrow\theta^\polar_0$ is admissible. If $\Omega^\polar$ has a corner at $Q_0$, then $\hat\theta(t)$ is not unique and vice versa. This case also corresponds to singular extremals.
	\end{enumerate}
\end{enumerate}

\begin{remark}
\label{remark:two_different_singular_controls}
	To distinguish cases~(\ref{ct_ode_general_singular}) and~(\ref{ct_ode_special_singular}) we use the following terms. In case (\ref{ct_ode_general_singular}), we say that $\hat\theta(t)$ is \textit{general singular}; and in case (\ref{ct_ode_special_singular}), we say that $\hat\theta(t)$ is \textit{special singular}.
\end{remark}

\section[Left-invariant  sub-Finsler problems on 3D unimodular Lie groups]{Left-invariant  sub-Finsler problems \\on 3D unimodular Lie groups}\label{sec:3DLie}
\subsection{Problem statement and normalization}
Let $G$ be a 3-dimensional unimodular Lie group, and $L = T_{\Id}G$ its Lie algebra, where $\Id$ is the identity element of $G$. Thus the following cases are possible~\cite{jacobson}:
\begin{itemize} 
\item
$L = \h_3$ is the Heisenberg algebra,
\item
$L = \su_2$ is the Lie algebra of the group of $2 \times 2$ unitary matrices,
\item
$L = \sl_2$ is the Lie algebra of the group of $2\times 2$ real unimodular matrices,
\item
$L = \se_2$ ($L = \sh_2$) is the Lie algebra of the group of Euclidean (resp. hyperbolic) motions of the 2-dimensional plane.
\end{itemize} 

Let $\Delta \subset L$ be a 2-dimensional subspace that is not a subalgebra, and let $U \subset \Delta$ be a convex compact set such that $0 \in \intt_{\Delta}U$. We consider the following left-invariant time-optimal problem:
\begin{align}
&\dot q \in q U, \qquad q \in G,  \label{pr1}\\
&q(0) = q_0, \quad q(T) = q_1, \label{pr2}\\
&T \to \min. \label{pr3}
\end{align}

\begin{remark}
If $U$ is an ellipse centered at the origin, then problem~\eq{pr1}--\eq{pr3} is sub-Riemannian. Some cases of this problem were studied in papers~\cite{boscain_rossi, berestovsky, max_sre, cut_sre1, cut_sre2, sh21, sh22, sh23, sr_so3}.
\end{remark}

\begin{remark}
Problem~\eq{pr1}--\eq{pr3} defines a distance function on the group in a classical way.
If $U = - U$, then 
the distance is symmetric. In this case 
problem~\eq{pr1}--\eq{pr3} is 
usually called
sub-Finsler
in literature. If $U \neq - U$ then the distance is not symmetric, but it defines a regular Hausdorff topology on the group. So we will call such a problem as well sub-Finsler for short.
\end{remark}

Since problem~\eq{pr1}--\eq{pr3} is left-invariant, we can assume that $q_0 = \Id$.

By the theorem of Agrachev-Barilari~\cite{agr_bar} on classification of contact left-invariant sub-Riemannian structures on 3D unimodular Lie groups, there exists a basis $L = \spann(X_1, X_2, X_3)$ such that
\begin{align}
&\Delta = \spann(X_1, X_2), \nonumber\\
&[X_1, X_2] = X_3, \quad [X_3, X_1] = (\chi + \kappa)X_2, \quad [X_3, X_2] = (\chi - \kappa) X_1, \label{product1}\\
&\chi = \kappa = 0 \text{\quad or \quad} (\chi^2 + \kappa^2= 1, \ \chi \geq 0). \nonumber
\end{align}

The case $\chi = \kappa = 0$ corresponds to the Heisenberg algebra $L = \h_3$, and the case when $G$ is the Heisenberg group was first studied in~\cite{buseman}, where curves having closed projection on the plane $(x_1,x_2)$ were found without using Pontryagin maximum principle. A full description of all geodesics in this problems was obtained in \cite{berestovskiiHeisenberg} with the help of Pontryagin maximum principle, and in~\cite{CT1} with the help of convex trigonometry. Thus we assume in the sequel that $\chi^2 + \kappa^2= 1$, $\chi \geq 0$. 

The two-parameter group of transformations
$
(X_1, X_2, X_3) \mapsto (\alpha X_1, \beta X_2, \alpha \beta X_3)$, $\alpha, \beta > 0$,
reduces product table~\eq{product1} to the following one:
\begin{align}
&[X_1, X_2] = X_3, \quad [X_3, X_1] = a X_2, \quad [X_3, X_2] = b X_1, \label{product21}\\
&a, \ b \in \{ \pm 1, \ 0\}, \qquad a+b\geq 0, \qquad (a,b) \neq (0,0), \label{product22}
\end{align}
where $a = \sgn(\chi + \kappa)$,  $b = \sgn(\chi  - \kappa)$.
So we get the following table of reduction from parameters $(\chi, \kappa)$ to parameters $(a,b)$:
\begin{enumerate}
\item
$\chi + \kappa < 0$, $\chi - \kappa > 0$ ($L = \sl_2$) \quad $\then$ \quad $a = -1$, $b = 1$,
\item
$\chi + \kappa = 0$, $\chi - \kappa > 0$ ($L = \sh_2$) \quad $\then$ \quad $a = 0$, $b = 1$,
\item
$\chi + \kappa > 0$, $\chi - \kappa > 0$ ($L = \sl_2$) \quad $\then$ \quad $a = 1$, $b = 1$,
\item
$\chi + \kappa > 0$, $\chi - \kappa = 0$ ($L = \se_2$) \quad $\then$ \quad $a = 1$, $b = 0$,
\item
$\chi + \kappa > 0$, $\chi - \kappa < 0$ ($L = \su_2$) \quad $\then$ \quad $a = 1$, $b = -1$.
\end{enumerate}

Parametrize the set $U \subset \Delta$ by a control parameter $u = (u_1,u_2) \in \Omega \subset \R^2 $, so that
$
U = \{u_1X_1+u_2X_2 \mid u = (u_1, u_2) \in \Omega\}$.
Then problem~\eq{pr1}--\eq{pr3}  reads as follows:
\begin{align}
&\dot q = u_1 X_1(q) + u_2 X_2(q), \qquad q \in G, \quad u = (u_1, u_2) \in \Omega, \label{pr21}\\
&q(0) = q_0 = \Id, \quad q(T) = q_1, \label{pr22}\\
&T \to \min, \label{pr23}
\end{align}
where left-invariant vector fields $X_1$, $X_2$ on the Lie group $G$ satisfy conditions~\eq{product21}, \eq{product22}.

\subsection{Pontryagin maximum principle}
We describe extremals of problem~\eq{pr21}--\eq{pr23}.
Notice that optimal controls in problem~\eq{pr21}--\eq{pr23} exist by the Rashevsky-Chow and Filippov theorems~\cite{notes}.
Denote the cotangent bundle of the Lie group $G$ as $T^*G$, and its points as $\lambda \in T^*G$. Introduce linear on fibers of $T^*G$ Hamiltonians corresponding to the basis vector fields:
$
h_i(\lambda) = \langle \lambda, X_i(q)\rangle$, $q = \pi(\lambda)$, $i = 1, 2, 3$,
where $\pi \ : \ T^*G \to G$ is the canonical projection.

Apply the Pontryagin maximum principle~\cite{PBGM, notes} to problem~\eq{pr21}--\eq{pr23}.
The Hamiltonian function of PMP is $\mathcal{H}(u, \lambda) = u_1 h_1(\lambda) + u_2 h_2(\lambda)$. The Pontryagin maximum principle states that if a curve $q(t)$ and a control $u(t)$, $t \in [0, T]$, are optimal, then there exists a Lipschitzian curve $\lambda_t \in T_{q(t)}^*G$, $\lambda_t \neq 0$, $t \in [0, T]$, that satisfies the following conditions:
\begin{itemize}
\item
the maximality condition
\be{max}
u_1h_1 + u_2 h_2 = \max_{w \in \Omega} (w_1 h_1 + w_2h_2) = s_{\Omega}(h_1, h_2) =:H,
\ee
\item
the Hamiltonian system
\be{Ham}
\begin{cases}
\dot h_1 = - u_2 h_3, \\
\dot h_2 = u_1 h_3, \\
\dot h_3 = - a u_1 h_2 - b u_2 h_1, \\
\dot q = u_1 X_1 + u_2 X_2,
\end{cases}
\ee
\item
and the identity
 $H \equiv \const \geq 0$ along any trajectory.
\end{itemize}

\medskip
\textbf{Abnormal case.} Let $H \equiv 0$. Then $h_1 = h_2 \equiv 0$, and since $\lambda_t \neq 0$ then $h_3 \neq 0$. Then the first two equations of system~\eq{Ham} yield $u_2 h_3 = u_1 h_3 \equiv 0$, thus $u_1 = u_2 \equiv 0$. So abnormal trajectories are constant, $q \equiv\const$.

\medskip
\textbf{Cauchy nonuniqueness of extremals.}
Extremals are trajectories of a nonautonomous ODE --- the Hamiltonian system of PMP~\eq{Ham}. So there can be two different extremals $\lam^1_t \not\equiv \lam^2_t$ that intersect one another:  $\lam^1_{t_0} = \lam^2_{t_0}$. We call such phenomenon Cauchy nonuniqueness of extremals. 
There can be two reasons for Cauchy nonuniqueness:
\begin{enumerate}
\item
Cauchy nonuniqueness of extremals due to different controls: the point $\lam_{t_0}$ of an extremal $\lam_{t}$ determines a nonunique control $u(t_{0})$ via the maximality condition of PMP~\eq{max}. Then we have different controls $u_1(t) \not\equiv u_2(t)$ and the corresponding different extremals $\lam^1_t \not\equiv \lam^2_t$ with $\lam^1_{t_0} = \lam^2_{t_0} = \lam_{t_0}$.
\item
Cauchy nonuniqueness due to nonsmoothness: the Hamiltonian vector field $\vec H$ corresponding to maximized Hamiltonian $H$ of PMP  is not $C^1$-smooth. 
Then, even if maximality condition~\eq{max} of PMP  defines uniquely control as a function of covector $\lambda \mapsto u(\lambda)$, the resulting vector field $\vec{H}(\lambda)$ may have different trajectories with the same initial point (like the ODE $\dot x = x^{1/3}$). 
\end{enumerate}

In order to distinguish different reasons for Cauchy nonuniqueness of extremals, introduce the following sets:
$$
\Sing_k = \{ (h_1, h_2) \in \R^2 \mid \text{$H \notin C^k(h_1, h_2)$} \}, \qquad k = 1, 2.
$$
It is obvious that $\{0 \} \subset \Sing_1 \subset \Sing_2$ and that $\Sing_1$ consists of a finite or countable set of rays beginning at the origin.  Moreover, the set of points where $\partial\Omega^\circ$ is not $C^2$ has zero measure by Alexandrov theorem~\cite{Alexandrov}. Hence the set $\Sing_2$ consists of rays beginning at the origin, and intersection of $\Sing_2$ with the unit circle has zero measure on it.

If $(h_1,h_2) \in \Sing_1$, then maximality condition~\eq{max} determines a nonunique control, and we may have classical singular trajectories. Moreover, bang-bang extremals can join singular extremals, so mixed extremals (concatenations of bang-bang and singular extremals) may appear because of Cauchy nonuniqueness due to different controls.

If $(h_1, h_2) \in \Sing_2 \setminus \Sing_{1}$, then $H \in C^1$ and maximality condition~\eq{max} determines a unique control $u = \nabla H$. Although, the Hamiltonian vector field $\vec{H} \in C^{0} \setminus C^{1}$, thus we may have Cauchy nonuniqueness of extremals due to nonsmoothness, and this case needs additional accuracy.

Finally, if $(h_1, h_2) \in \R^2 \setminus \Sing_2$, then $H \in C^2$, thus $\vec{H} \in C^1$, and there is no Cauchy nonuniqueness of extremals.

\medskip \textbf{Normal case.} Let $H > 0$. 
The maximized Hamiltonian $H(h_1, h_2)$ is a convex positively homogeneous of order one function in the plane $\R^2_{h_1,h_2}$. 
Maximality condition~\eq{max} yields $u = (u_1, u_2) \in \partial \Omega$, thus
\be{u12theta}
u_1 = \co \theta, \qquad u_2 = \so \theta
\ee
for an angle $\theta \in \R /(2\Sn \Z)$.
Moreover, condition~\eq{max} yields $(h_1, h_2) \in H \partial \Omega^{\polar}$, thus
\be{h12theta}
h_1 = H \coo \theta^{\polar}, \qquad h_2 = H \soo \theta^{\polar}
\ee
for an angle $\theta^{\polar} \in \R /(2 \So \Z)$, $\So = \Sn(\Omega^{\polar})$.
Since
$$
\co \theta \coo \theta^{\polar} + \so \theta \soo \theta^{\polar} = \frac{u_1 h_1+ u_2 h_2 }{H}= 1,
$$
then $\theta \leftrightarrow \theta^{\polar}$ for a.e. $t$.
Further, we have
$$
\dot \tho = \frac{h_1 \dot h_2 - h_2 \dot h_1}{H^2} = (H \coo \tho \co \theta h_3 + H \soo \tho \so \theta h_3) /H^2 = h_3/H.
$$
So the vertical subsystem of Hamiltonian system~\eq{Ham} reduces to the system
\be{Ham2}
\begin{cases}
\dot \tho = h_3/H, \\
\dot h_3 = -H(a \co \theta \soo \tho + b \so \theta \coo \tho).
\end{cases}
\ee

Using Theorem~\ref{thm:general_energy_integral}, we see that $\frac12\dot{\theta^\polar}^2 + \frac12(a \soo^2\tho - b\coo^2\tho)$ is a first integral of the system. This fact can also be obtained via the left-invariant structure of the system. Indeed, the function
$
\Cn = \frac 12 (h_3^2 - b h_1^2 + a h_2^2)
$
is a Casimir on the Lie coalgebra $L^*$ for Lie--Poisson bracket, thus the ratio $\En = \Cn/H$ is a first integral of system~\eq{Ham2}. Decompose
\begin{align*}
&\En = \frac{1}{2H} h_3^2 + \Un(\tho),\\
&\Un(\tho) = \frac{1}{2H}(ah_2^2 - b h_1^2) = \frac{H}{2} (a \sin^2_{\Omega^{\polar}} \tho - b \cos^2_{\Omega^{\polar}} \tho).
\end{align*}
We call $\frac{1}{2H} h_3^2$ the kinetic energy, $\Un(\tho)$ the potential energy, and $\En$ the full energy of system~\eq{Ham2}. This system turns out to be a Hamiltonian system with the Hamiltonian $\En$:
\be{Ham3}
\begin{cases}
\dot \tho = h_3/H, & \tho \in \R/(2 \So \Z), 
\\
\dot h_3 \in -\Un'(\tho), & h_3 \in \R.
\end{cases}
\ee 

If $\Omega^\polar$ has $C^2$-smooth boundary, then $\cos_{\Omega^\polar} \theta^{\polar}$ and $\sin_{\Omega^\polar} \theta^{\polar}$ are $C^2$-smooth functions by Proposition~\ref{prop1}. In this case we have a classical smooth Hamiltonian system. If $\partial \Omega^\polar$ has corners then $\Un'(\theta^\polar)$ is an interval between left and right derivatives at a corner $\theta^\polar$. In this case we read Hamiltonian system~(\ref{Ham3}) in the Filippov sense~\cite{Filippov}, and it may have multiple solutions with the same initial data.

\begin{remark}
	The function $\Un(\tho)$ is Lipschitzian in general, since $\partial\Omega^\circ$ is Lipschitzian, and system~(\ref{Ham3}) may have non-unique solution. But if an open interval on $\partial\Omega^\circ$ is $C^2$, then $\Un\in C^2$ for corresponding angles $\theta^\circ$, and system~(\ref{Ham3}) has a unique solution for any given initial data. In other words, uniqueness may fail only on the set $\Sing_2$.
\end{remark}

The phase portrait of system~\eq{Ham3} is completely determined by the potential energy~$\Un(\tho)$. 

We call an extremal arc $\lambda_t$, $t \in [\a, \b]$, $\a < \b$:
\begin{itemize}
\item
a bang arc if
$(h_1, h_2)(\lambda_t) \notin \Sing_{{2}}$ for all $t \in (\a, \b)$,
\item
a bang-bang arc if 
$(h_1, h_2)(\lambda_t) \notin \Sing_{{2}}$ for a.e. $t\in(\a,\b)$,
\item
a singular arc if $(h_1, h_2)(\lambda_t) \in \Sing_{{2}}$ for all $t \in [\a, \b]$,
\item
mixed if it is a concatenation of a finite number of bang-bang and singular arcs. 
\end{itemize}

\begin{remark}
	Usually, the term ``singular extremal'' is used for extremals where PMP defines a non-unique control. On these extremals uniqueness of solutions to PMP is lost. Moreover, singular extremals were studied mostly for problems with one-dimensional control, where analogues of the sets $\Sing_1$ and $\Sing_2$ coincide one with another. We believe that the term ``non-singular extremals'' must be applied to extremals where everything is regular, i.e., \ PMP has the property of uniqueness of solutions. Therefore we decided to use term ``singular'' in this paper for extremals in $\Sing_2$.
\end{remark}

In examples considered below we prove that Fuller's phenomenon is not present here, i.e.,
along any extremal arc duration of all maximal bang arcs is separated from zero,  thus any extremal arc is either bang-bang, singular, or mixed.

\subsection{Bang-bang extremals}

System~\eq{Ham3} has fixed points $(\tho, h_3) = (\tho_*, 0)$, where $\Un'(\tho_*) \ni 0$.

If $\Un(\tho)$ has a local minimum at $\tho_*$, decreases at an interval $(\tho_* - \varepsilon, \tho_*]$ and increases at an interval $[\tho_*, \tho_* + \varepsilon)$, then the phase portrait of system~\eq{Ham3} has a fixed point $(\tho, h_3) = (\tho_*, 0)$ of the center type.  

If $\Un(\tho)$ has a local maximum at $\tho_*$, increases at an interval $(\tho_* - \varepsilon, \tho_*]$ and decreases at an interval $[\tho_*, \tho_* + \varepsilon)$, then the phase portrait of system~\eq{Ham3} has a fixed point $(\tho, h_3) = (\tho_*, 0)$ of the saddle type. The separatrix entering the   fixed point $(\tho_*, 0)$ in the strip $\tho \in (\tho_*-\varepsilon, \tho_*)$ comes to the fixed point for the time evaluated by the improper integral (see Section~\ref{subsec:nonuniqueness_general_ct_ode})
\be{I}
I = H \int_{\tho_* - \varepsilon}^{\tho_*} \frac{d \tho}{\sqrt{2 (\En - \Un(\tho))}}.
\ee
This time is finite or infinite depending on whether the integral $I$ converges or not. Similarly for the separatrix in the strip $\tho \in (\tho_*, \tho_* + \varepsilon)$. 
So, unlike the classic case, separatrix can enter the fixed point of saddle type in finite time. In this case, we have Cauchy nonuniqueness of extremals.

\subsection{Singular extremals}
\label{sec:singular_lie}
{Since the singular set $\Sing_2$ of Hamiltonian $H$ consists of rays starting at the origin forming set of zero measure}, and $s_\Omega(h_1,h_2)=H=\const$, any singular arc $\lambda_t$ satisfies the identities
\be{sing1}
h_1(\lambda_t) \equiv \const, \qquad h_2(\lambda_t) \equiv \const.
\ee
If the corresponding extremal trajectory $q(t)$ is not constant, then the Hamiltonian system~\eq{Ham} implies additionally that
\begin{align}
&h_3(\lambda_t) \equiv 0, \label{sing2}\\
&a u_1(t) h_2(\lambda_t) + b u_2(t) h_1(\lambda_t) \equiv 0. \label{sing3}
\end{align}
Taking into account \eq{u12theta}, \eq{h12theta} and the maximality condition \eq{max},
we conclude that a singular extremal satisfies the following conditions:
\begin{align}
&\tho \equiv \const, \label{singc1}\\
&h_3 \equiv 0, \label{singc2}\\
&a \sin_{\Omega^{\polar}} \theta^{\polar} \cos_{\Omega} \theta + b \cos_{\Omega^{\polar}} \theta^{\polar} \sin_{\Omega} \theta \equiv 0, \label{singc3}\\
&\cos_{\Omega^{\polar}} \theta^{\polar} \cos_{\Omega} \theta + \sin_{\Omega^{\polar}} \theta^{\polar} \sin_{\Omega} \theta \equiv 1, \label{singc4}\\
&(h_1(\lambda_t), h_2(\lambda_t)) \in \Sing_{{2}} \label{singc5}.
\end{align}
Conversely, any arc $\lam_t$ and control $u(t)$ that satisfy conditions \eq{singc1}--\eq{singc5} are singular.

Hence $\theta^\circ=\mathrm{const}$ is a clue parameter determining the singular extremal.

Let us investigate solutions of system \eq{singc1}--\eq{singc5} {(the investigation is similar to Section~\ref{subsec:theta_in_general_ct_ode})}. We start with equations \eq{singc3}, \eq{singc4}. They may be written in the following form:
\be{Au}
	A
	\left(
		\begin{array}{cc}
			u_1\\u_2
		\end{array}		
	\right)=
	\left(
		\begin{array}{cc}
			1\\0
		\end{array}		
	\right),
\ee
where $u=(u_1,u_2)=(\cos_\Omega\theta,\sin_\Omega\theta)$ and
\[
	A=\left(
		\begin{array}{cc}
			\cos_{\Omega^\circ}\theta^\circ & \sin_{\Omega^\circ}\theta^\circ \\
			a\sin_{\Omega^\circ}\theta^\circ& b\cos_{\Omega^\circ}\theta^\circ \\
		\end{array}		
	\right).
\]
Any solution to linear system \eq{Au} must belong to the support line of $\Omega$ defined by the first equation 
$
	\cos_{\Omega^\circ}\theta^\circ u_1 + \sin_{\Omega^\circ}\theta^\circ u_2=1$,
but it may happen that there exist solutions to linear system \eq{Au} that do not belong to~$\partial\Omega$. 

Let $u^s(t)\in\partial\Omega$ for $t\in[t_1,t_2]$ be a measurable function such that $Au^s(t)=(1,\ 0)^T$ for a.e.\ $t\in[t_1,t_2]$. Then for the initial point $q_0=\mathrm{Id}$ there exists a unique corresponding singular extremal of the form $h_1=\cos_{\Omega^\circ}\theta^\circ=\mathrm{const}$, $h_2=\sin_{\Omega^\circ}\theta^\circ=\mathrm{const}$, $h_3=0$ and $u(t)=u^s(t)$, where $q(t)$ is found as a unique solution to \eq{Ham} with $u(t)=u^s(t)$. Moreover, any singular extremal with given $\theta^\circ=\mathrm{const}$ has the described form. Hence, solutions of linear system \eq{Au} play a key role in investigating singular extremals.

Let us now investigate linear system \eq{Au}. Obviously, $\mathrm{rk}\,A\ge 1$, since $(\cos_{\Omega^\circ}\theta^\circ,\sin_{\Omega^\circ}\theta^\circ)\ne 0$. If $\mathrm{rk}\,A=2$, then linear system \eq{Au} has a unique solution~$u^s$. If $u^s\not\in \partial\Omega$ then there are no singular extremals with given~$\theta^\circ$. If $u^s\in\partial\Omega$, then for the initial point $q_0=\mathrm{Id}$, there exists a unique corresponding singular extremal with given $\theta^\polar$, and it has the constant singular control $u(t)=u^s=\mathrm{const}$.
	
Let us now consider the special case $\mathrm{rk}\,A=1$. In this case, system \eq{Au} has a solution iff the second row of $A$ is $(0,\ 0)$, i.e.,
$
	a\sin_{\Omega^\circ}\theta^\circ = b\cos_{\Omega^\circ}\theta^\circ=0$.

\begin{enumerate}
	\item If $a^2+b^2=2$ (cases $\mathfrak{sl}_2$ and $\mathfrak{su}_2$), then there are no singular extremals in the special case.
	
	\item If $a=0$ and $b=1$ (case $L=\mathfrak{sh}_2$), then there are singular extremals in the special case if $\cos_{\Omega^\circ}\theta^\circ=0$ and $(0,\sin_{\Omega^\circ}\theta^\circ)\in\mathfrak{S}_2$. Indeed, if the last two conditions are fulfilled, then linear system \eq{Au} has infinite number of solutions, which form a horizontal supporting line to $\Omega$. Intersection of this line with $\partial\Omega$ is a segment $[u^s_l,u^s_r]$ (or a point $u^s_l=u^s_r$ if $(h_1,h_2)\in\mathfrak{S}_2\setminus\mathfrak{S}_1$). Hence, for any measurable function $u^s(t)\in[u^s_l,u^s_r]$ for a.e.\ $t\in[t_1,t_2]$ (i.e.,\ $u^s(t)=(\cos_\Omega\theta(t),\sin_\Omega\theta(t))$ where $\theta(t)\leftrightarrow\theta^\circ$ for a.e. \ $t\in[t_1,t_2]$), there exists a unique corresponding singular extremal { with $h_1=H\cos_{\Omega^\circ}\theta^\circ=\const$, $h_2=H\sin_{\Omega^\circ}\theta^\circ=\const$, and $h_3=0$}.
	
	\item If $a=1$ and $b=0$ (case $L=\mathfrak{se}_2$), then there are singular extremals in the special case if $\sin_{\Omega^\circ}\theta^\circ=0$ and $(\cos_{\Omega^\circ}\theta^\circ,0)\in\mathfrak{S}_2$. If so, then the situation is similar to the previous one, but for vertical supporting lines to $\Omega$.
	
	\item If $a=b=0$ (case $L=\mathfrak{h}_3$), then there are singular extremals in special case for arbitrary $\theta^\circ$ if $(\cos_{\Omega^\circ}\theta^\circ,\sin_{\Omega^\circ}\theta^\circ)\in\mathfrak{S}_2$ (see \cite{CT1}). The corresponding singular control has the form $u^s(t)=(\cos_\Omega\theta(t),\sin_\Omega\theta(t))$ where $\theta(t)\leftrightarrow\theta^\circ$ for a.e.\ $t\in[t_1,t_2]$.
\end{enumerate}

In the sequel, we distinguish two cases of singular extremals described above:
\begin{itemize}
\item
general singular extremals in the case $\mathrm{rk}\,A=2$,
\item
special singular extremals in the case $\mathrm{rk}\,A=1$.
\end{itemize}


\subsection{Mixed extremals}
A mixed extremal appears when a bang arc $(\tho(t), h_3(t))$ enters for a finite time a point $(\tho_*, 0)$ that corresponds to a singular extremal.

\subsection{Special case: $\Omega$ a polygon}
 Let $\Omega \subset \R^2$ be a convex polygon containing the origin in its interior.
The polar set to~$\Omega$ is a convex polygon $\Omega^{\polar} = \conv(\om^{\polar}_1, \dots,  \om^{\polar}_k) \subset \R^2$  with the origin in its interior.  The generalized trigonometric functions $\coo \tho$ and $\soo \tho$ are piecewise linear  on $\R / (2 \So \Z)$ with possible corner points $\tho_1$, \dots, $\tho_k$, where $(\coo \tho_i, \soo \tho_i) = \om^{\polar}_i$, $i = 1, \dots, k$. The potential energy $\Un(\tho) = \frac 12 (a \soo^2 \tho - b \coo^2 \tho)$ is piecewise quadratic, thus $\Un'(\tho)$ is piecewise linear with the same possible corner points. So at each segment $\tho \in (\tho_i, \tho_{i+1})$, $i = 1, \dots, k$, $\tho_{k+1} = \tho_1$, system~\eq{Ham3} takes the form
$$
\left( \begin{array}{cc}
\dot \tho \\ \dot h_3 
\end{array} \right)
= 
\left( \begin{array}{cc}
0 & 1 \\
\alpha_i & 0 
\end{array} \right)
\left( \begin{array}{cc}
\tho \\ h_3 
\end{array} \right)
+
\left( \begin{array}{cc}
0 \\   \beta_i 
\end{array} \right), \qquad \alpha_i, \ \beta_i \in \R,
$$
where $\alpha_i = a \cos_{\Omega}^2 \theta_i + b \sin_{\Omega}^2 \theta_i$.
This system in the strip $  (\tho_i, \tho_{i+1}) \times \R$ has the following phase portraits (we mention in parentheses the elementary functions in which the system is integrated):

\begin{itemize}
\item[$\alpha_i>0$:] hyperbolic (hyperbolic functions),
\item[$\alpha_i<0$:] elliptic (trigonometric functions),
\item[$\alpha_i=0$:] parabolic (quadratic functions).
\end{itemize}
The phase portrait of system~\eq{Ham3} in the whole cylinder $  (\R / (2 \So \Z)) \times \R $ is glued from the phase portraits in the strips $(\tho_i, \tho_{i+1})\times\R $. Let us examine behaviour of system~\eq{Ham3} on lines $\tho=\tho_i$. If $\tho(t_0)=\tho_i$ and $\dot\tho(t_0)\ne 0$ on a trajectory, then it intersects the line $\tho=\tho_i$ at an isolated point $t=t_0$ and has a corner type singularity at it. But if $\tho(t_0)=\tho_i$ and $\dot\tho(t_0)=0$, then a singular arc $\tho(t)\equiv\tho_i$ with Cauchy nonuniqueness may appear.

Singular arcs are curves $\lambda_t \in T^* M $  that satisfy conditions~\eq{singc1}--\eq{singc5}.
Switching times correspond to motion of Hamiltonian system~\eq{Ham3} between successive vertices of the polygon~$\Omega^{\polar}$. {Note that condition of Theorem~\ref{thm:absense_of_uniqueness_for_general_ct_ode} is automatically fulfilled for each corner~$\theta^\polar_0$ with $0\in\mathcal{U}'(\theta^\polar_0)$, since the function $f$ is analytic. So, using Propositions~\ref{prop:uniqueness_of_solution_when_Omega_smooth}, \ref{prop:uniqueness_if_zero_not_in_U_derivative} and Theorem~\ref{thm:absense_of_uniqueness_for_general_ct_ode}, we obtain}

\begin{thm}\label{th:Lie-polygon}
If $\Omega$ is a convex polygon containing the origin in its interior, then
all extremals in problem~\eq{pr21}--\eq{pr23} are bang-bang, singular or mixed. {Hence, if an extremal does not contain special singular control, then the control is piecewise constant and the extremal is piecewise smooth.}
\end{thm} 

In Subsections~\ref{subsec:li} and~\ref{subsec:l1} we consider in full detail the cases of the squares $\Omega = \{u \in \R^2 \mid \|u\|_{\infty} \leq 1\}$ and $\Omega = \{u \in \R^2 \mid \|u\|_{1} \leq 1\}$. 

\subsection{Special case: $\Omega$ strictly convex}
Let $\Omega \subset \R^2$ be a strictly convex compact set containing the origin in its interior.
Then the support function~$H$ is $C^1$ out of the origin, i.e., $\Sing_1 = \{ 0 \}$. 
If in addition $\partial\Omega^{\polar}$ is $C^2$-smooth, then $\vec{H} \in C^1$ and there are neither singular nor mixed trajectories.
Moreover,   if $\partial\Omega^{\polar}$ is $W^2_\infty$, then the solution is unique by Proposition~\ref{prop:uniqueness_of_solution_when_Omega_smooth}. So the following theorem holds.

\begin{thm}\label{th:Lie-strict}
If $\Omega$ is a strictly convex compact set containing the origin in its interior and $\partial\Omega^{\polar}$ is  {$C^k$-}smooth, $k\ge 2$, then all extremals in problem~\eq{pr21}--\eq{pr23} are bang and {$C^k$-}smooth.
Moreover, there is no Cauchy nonuniqueness of extremals.
\end{thm}

\begin{proof}
	Indeed, $\cos_{\Omega^\polar}$ and $\sin_{\Omega^\polar}$ are $C^k$ functions by Proposition~\ref{prop:w_k_p}. Hence, $\mathcal{U}$ is $C^k$, and $\theta^\polar(t)$ is $C^{k+1}$. Moreover $\theta(\theta^\polar)$ is $C^{k-1}$ by Corollary~\ref{cor:curvature}. Thus $u_1(t)=\cos_\Omega\theta(\theta^\polar(t))$ and $u_2(t)=\sin_\Omega\theta(\theta^\polar(t))$ are $C^{k-1}$. Hence extremals are $C^k$ by~\eqref{pr21}.
\end{proof}

We consider the case $\Omega = \{u \in \R^2 \mid \|u\|_{p} \leq 1\}$, $1 < p < \infty$, in Subsec.~\ref{subsec:lp}.

\subsection{Example: $\Omega = \{ \|u\|_{\infty} \leq 1\}$}\label{subsec:li}
Let $\Omega = \{(u_1, u_2) \in \R^2 \mid |u_1|, \ |u_2| \leq 1\}$. 
Then $\Omega^{\polar} = \{(h_1, h_2) \in \R^2 \mid |h_1| + |h_2| \leq 1\}$ and $\mathbb \So = 2$. 
The functions $\soo \tho$ and $\coo \tho$ are piecewise linear and 4-periodic, they are plotted at Figs.~\ref{fig:sinli} and~\ref{fig:cosli} (these plots were first given in~\cite{CT1}). 

\twofiglabels{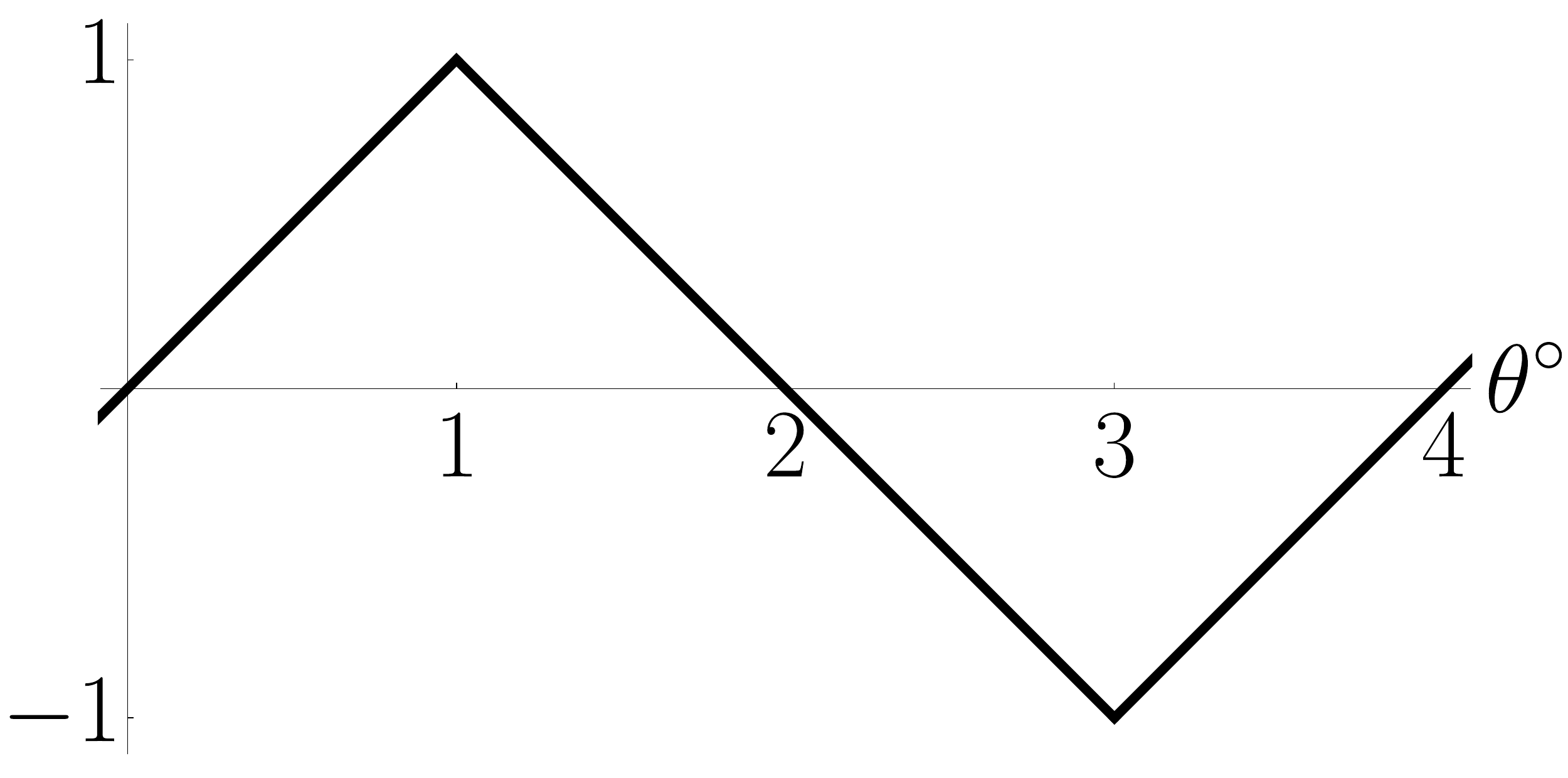}{Plot of $\soo  \tho$: $\ell_\infty$}{fig:sinli}{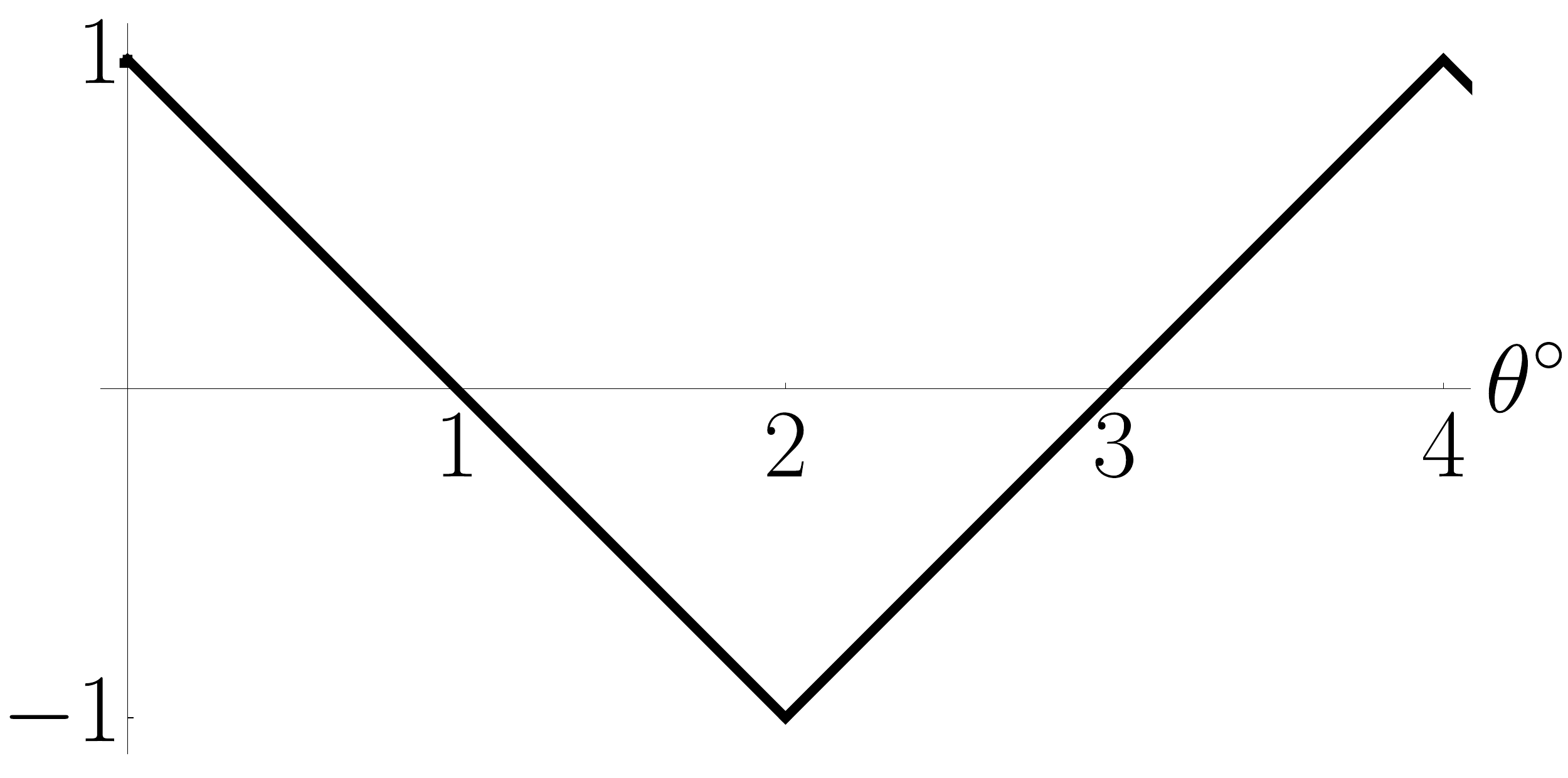}{Plot of $\coo \tho$: $\ell_\infty$}{fig:cosli}{0.35}{0.35}

We have $\soo^2 \tho = (1 - |(\tho{\mod 2})-1|)^2$, see Fig.~\ref{fig:sin2li}, and $\coo^2 \tho=(1-(\tho {\mod 2}))^2$, see Fig.~\ref{fig:cos2li}. 

\twofiglabel{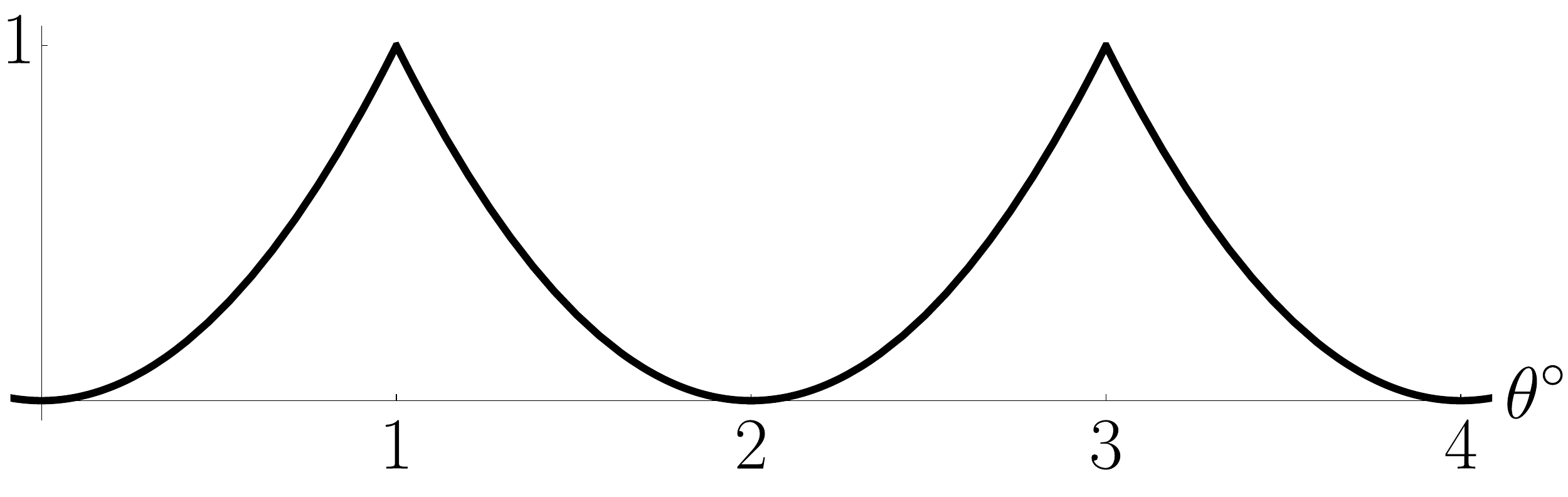}{Plot of $\soo^2 \tho$: $\ell_\infty$}{fig:sin2li}{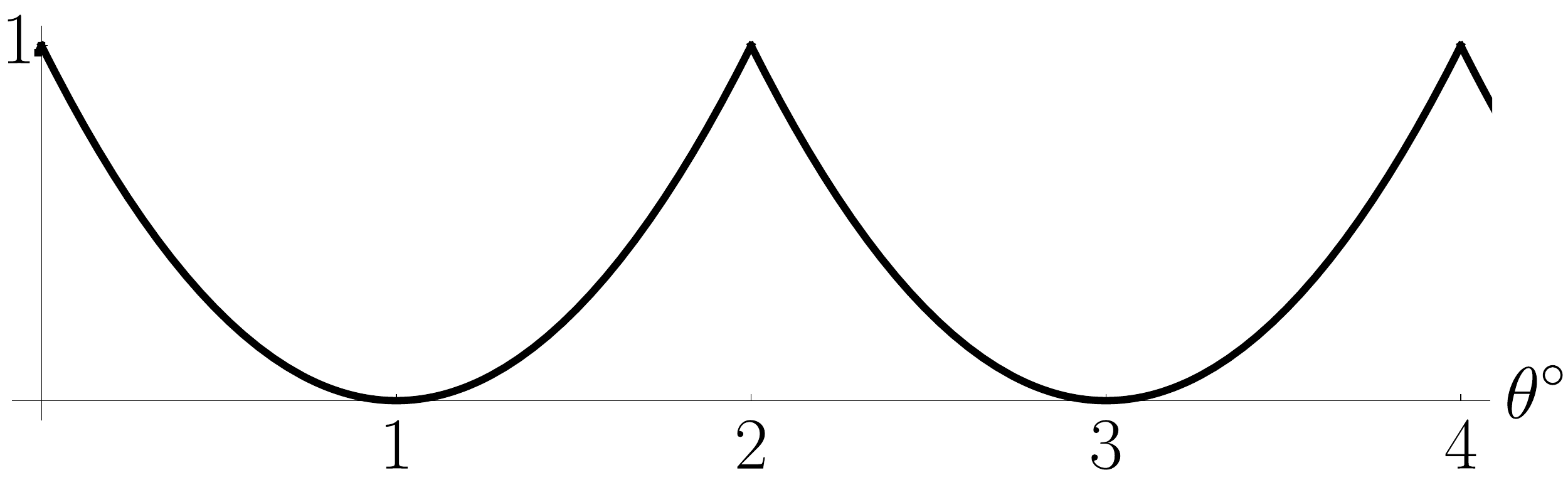}{Plot of $\coo^2 \tho$: $\ell_\infty$}{fig:cos2li}

\medskip \textbf{Case 1: $a = -1$, $b = 1$.} We have $\Un(\tho {+k}) = -(2 (\tho)^2 - 2 \tho +1)/2$ {for $\tho\in[0;1]$, $k\in \Z$}, see Fig.~\ref{fig:U1li}. In the strip $(0, 1) \times \R$ system~\eq{Ham3} has the form
$
\dot \tho = h_3$, $\dot h_3 = 2 \tho - 1$,
it has the saddle phase portrait (see Fig.~\ref{fig:phase1li}) and the bang trajectories
\begin{align*}
&\tho = C_1 \cosh (\sqrt 2 t) + C_2 \sinh (\sqrt 2 t) + 1/2,\\
&h_3 = C_1 \sqrt 2 \sinh (\sqrt 2 t) + C_2 \sqrt 2 \cosh(\sqrt 2 t), \\
&C_1 = \tho_0 - 1/2, \qquad C_2 = h_3^0/\sqrt 2.
\end{align*} 
\twofiglabel
{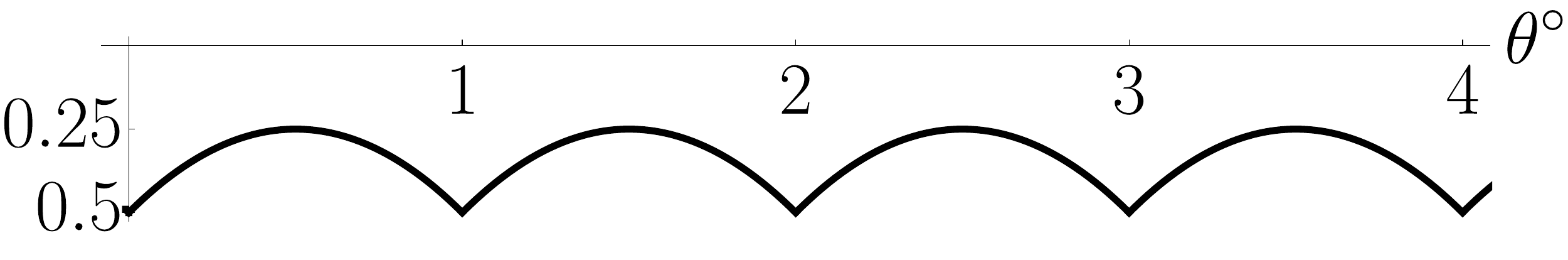}{Plot of $\Un(\tho)$: $\ell_\infty$, Case 1}{fig:U1li}
{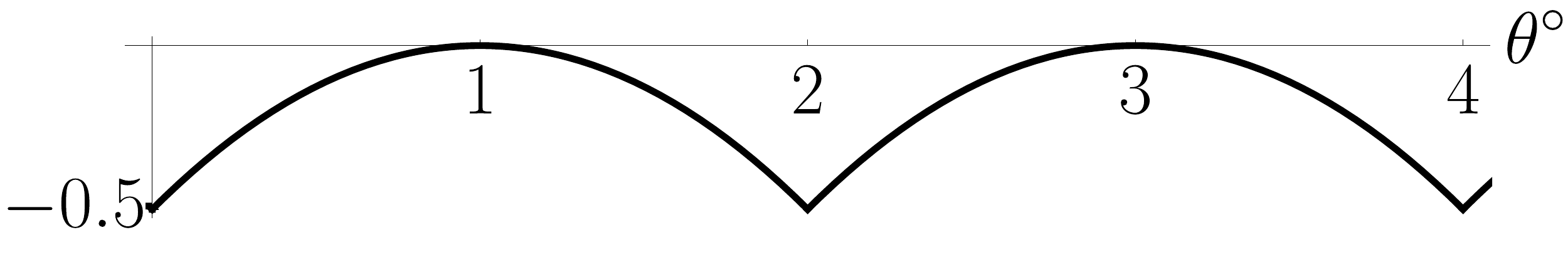}{Plot of $\Un(\tho)$: $\ell_\infty$, Case 2}{fig:U2li}

General singular extremals are $(\tho, h_3) = (n, 0)$, $n \in \Z$. There are {neither}  mixed  {nor special singular} extremals.  

\medskip \textbf{Case 2: $a = 0$, $b = 1$.} We have $\Un(\tho {+2k}) = -(\tho -1)^2/2$ {for $\tho\in[0;2]$, $k\in\Z$}, see Fig.~\ref{fig:U2li}. In the strip $(0, 2) \times \R$ system~\eq{Ham3} has the form
$
\dot \tho = h_3$, $\dot h_3 = \tho - 1$,
it has the saddle phase portrait (see Fig.~\ref{fig:phase2li}) and the bang trajectories
\begin{align*}
&\tho = C_1 \cosh t + C_2 \sinh t + 1,\\
&h_3 = C_1 \sinh t + C_2 \sinh t, \\
&C_1 = \tho_0 - 1, \qquad C_2 = h_3^0.
\end{align*} 

\twofiglabels
{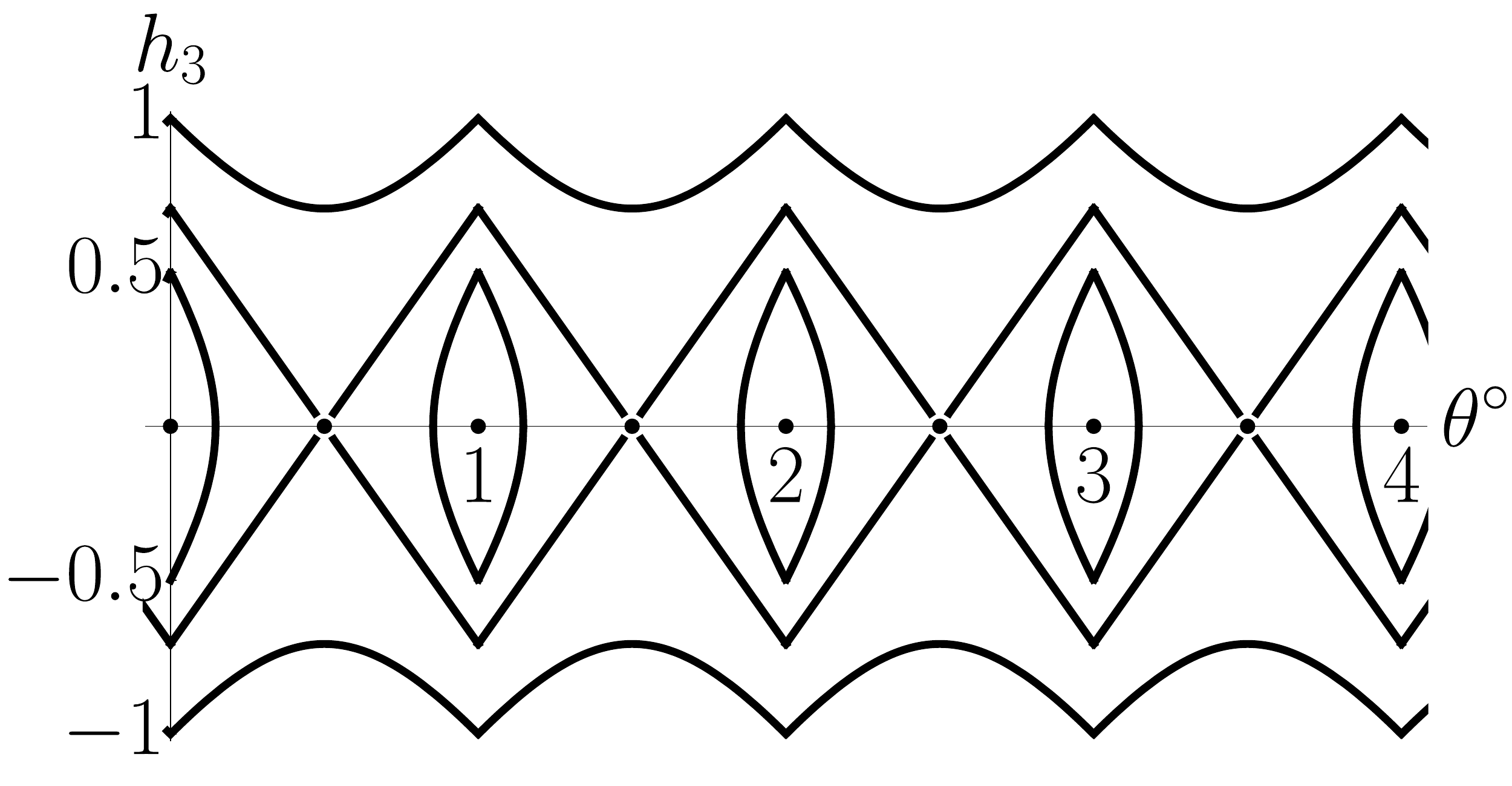}{Phase portrait of \eq{Ham3}: $\ell_\infty$, Case 1}{fig:phase1li}
{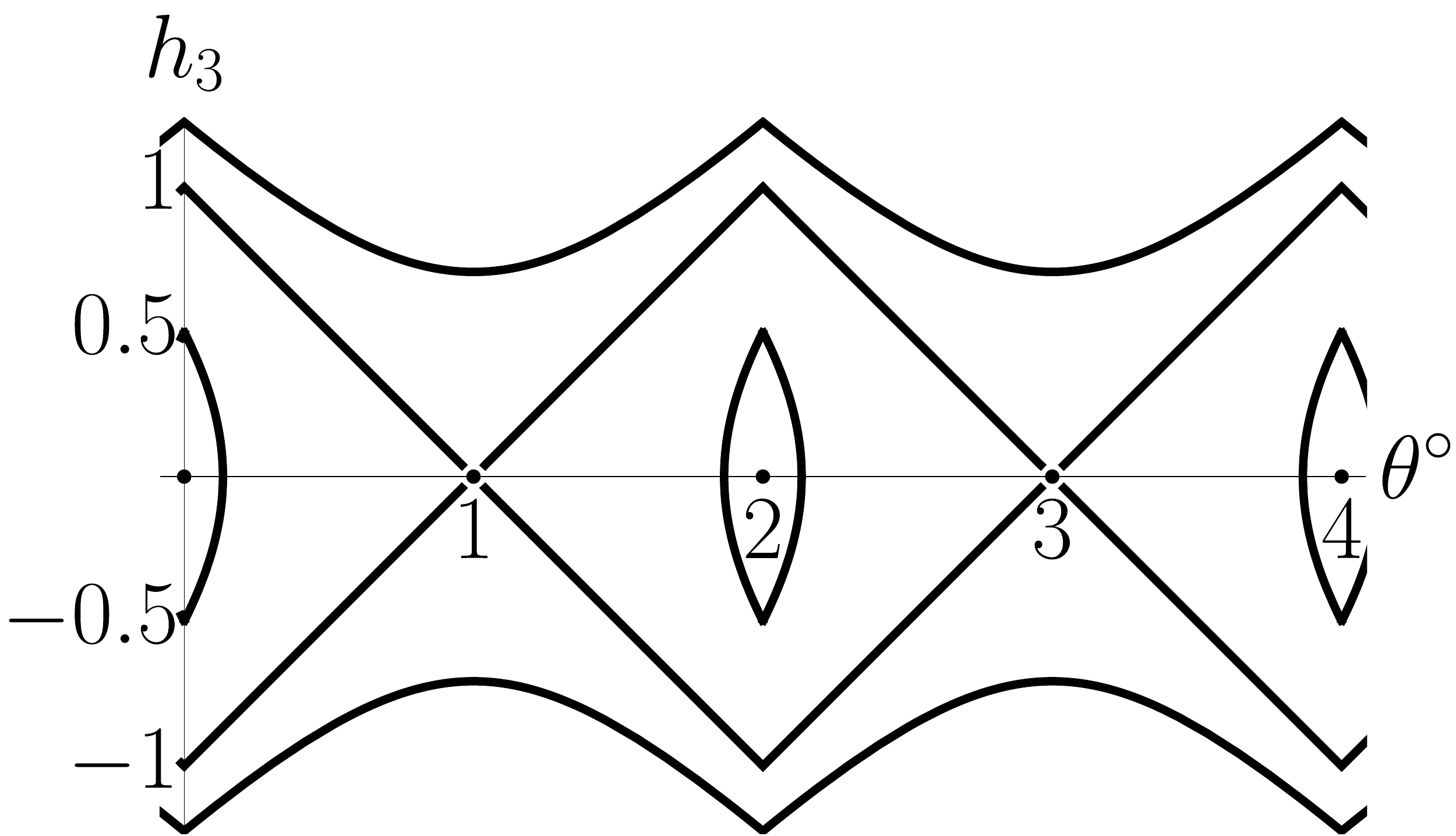}{Phase portrait of \eq{Ham3}: $\ell_\infty$, Case 2}{fig:phase2li}{0.39}{0.35}
General singular extremals are $(\tho, h_3) = (2n, 0)$, $n \in \Z$, and special singular extremals are $(\tho, h_3) = (2n+1, 0)$, $n \in \Z$. There are no mixed extremals. 

\medskip \textbf{Case 3: $a = 1$, $b = 1$.} We have $\Un(\tho {+2k}) = 1/2 -|\tho -1|$ {for $\tho\in[0;2]$, $k\in\Z$}, see Fig.~\ref{fig:U3li}. In the strip $(0, 1) \times \R$ system~\eq{Ham3} has the form
$
\dot \tho = h_3$, $\dot h_3 = - 1$,
it has the parabolic phase portrait (see Fig.~\ref{fig:phase3li}) and the bang trajectories
\begin{align*}
&\tho = \tho_0 + h_3^0 t - t^2/2,\\
&h_3 = h_3^0 - t.
\end{align*} 
\twofiglabels
{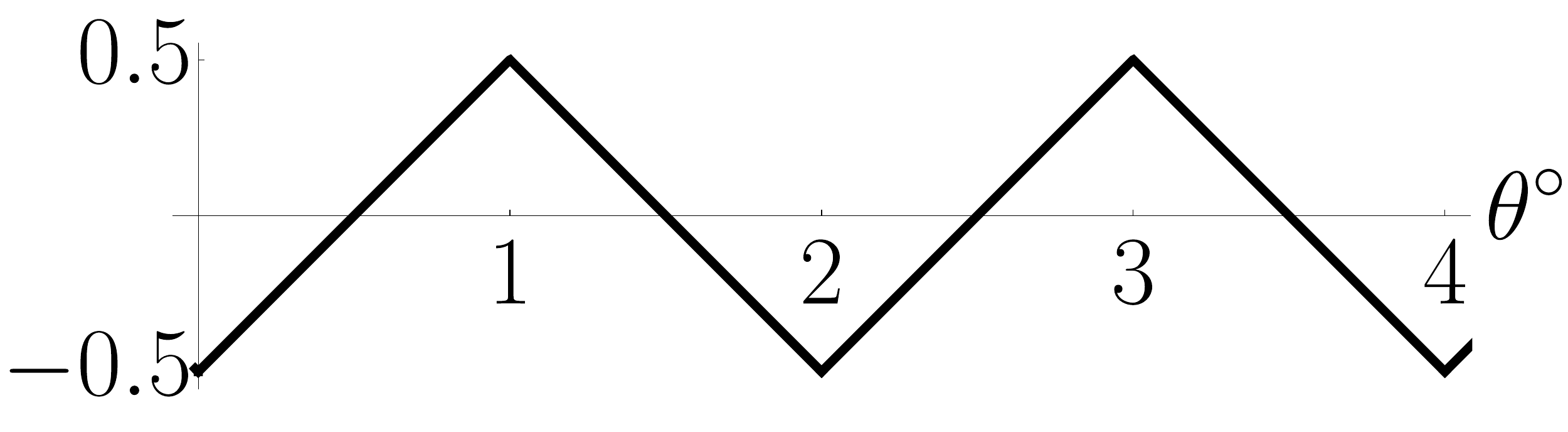}{Plot of $\Un(\tho)$: $\ell_\infty$, Case 3}{fig:U3li}
{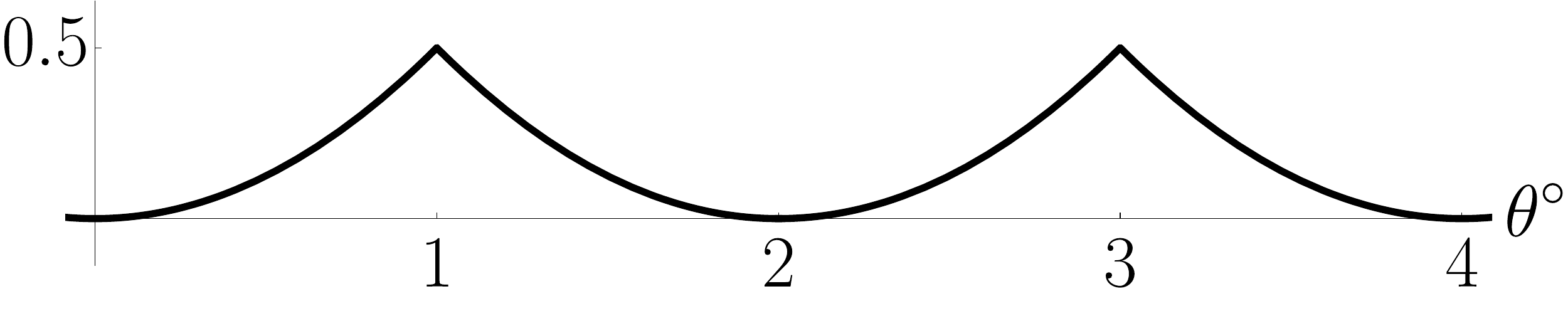}{Plot of $\Un(\tho)$: $\ell_\infty$, Case 4}{fig:U4li}{0.35}{0.47}

General singular extremals are $(\tho, h_3) = (n, 0)$, $n \in \Z$.  {There are no special singular extremals.}
The parabolas enter the points $(\tho, h_3) = (1, 0)$ and $(3,0)$ for finite time. Thus
at these points  bang extremals join singular extremals, there appear mixed extremals.

\medskip \textbf{Case 4: $a = 1$, $b = 0$.} We have $\Un(\tho {+2k}) = (1 - |\tho-1|)^2/2$ {for $\tho\in[0;2]$, $k\in\Z$}, see Fig.~\ref{fig:U4li}. In the strip $(0, 1) \times \R$ system~\eq{Ham3} has the form
$
\dot \tho = h_3$, $\dot h_3 = -\tho$,
it has the center phase portrait (see Fig.~\ref{fig:phase4li}) and the bang trajectories --- arcs of circles
\begin{align*}
&\tho = C_1 \cos t + C_2 \sin t,\\
&h_3 = -C_1 \sin t + C_2 \cos t, \\
&C_1 = \tho_0, \qquad C_2 = h_3^0.
\end{align*} 
\twofiglabels
{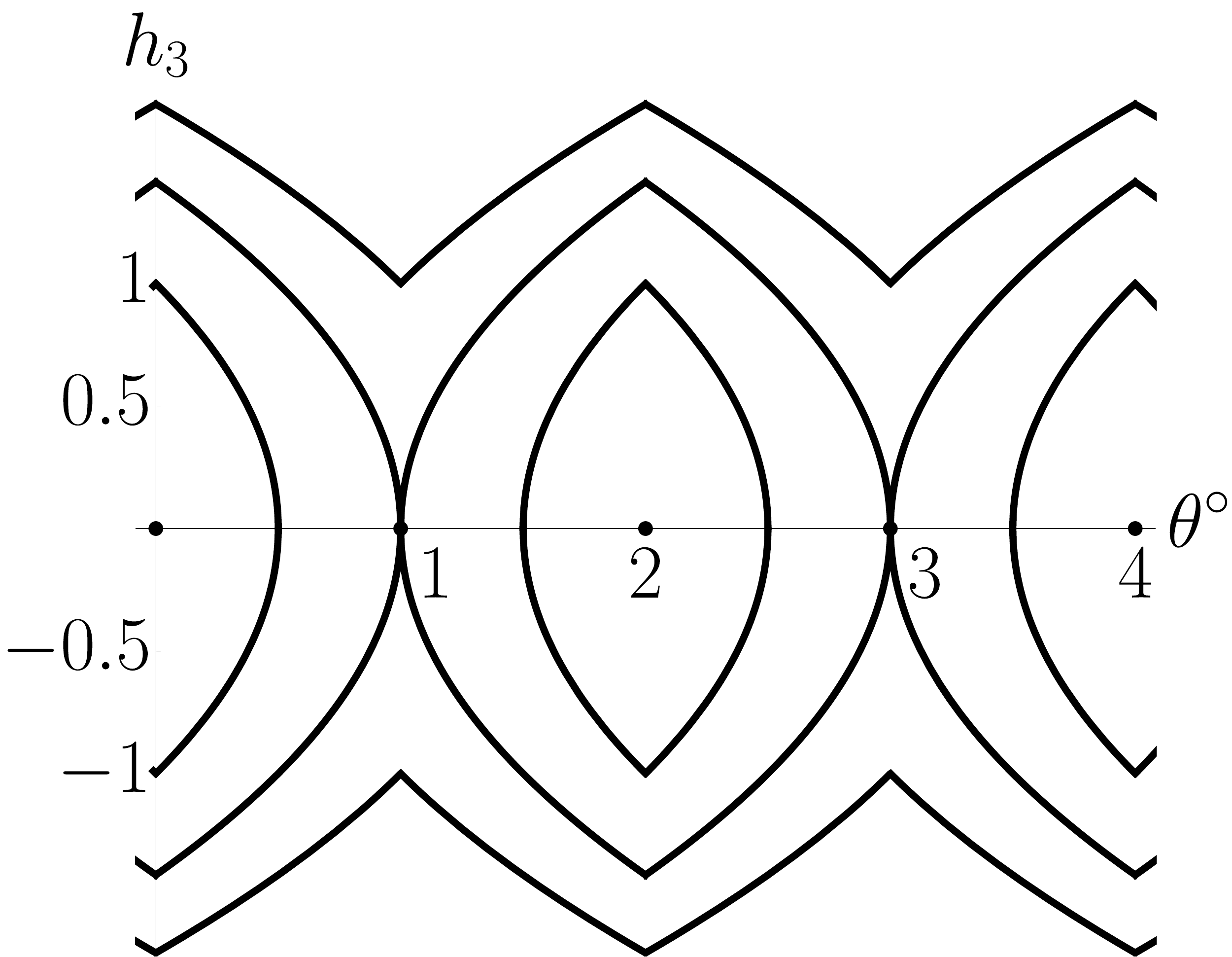}{Phase portrait of \eq{Ham3}: $\ell_\infty$, Case 3}{fig:phase3li}
{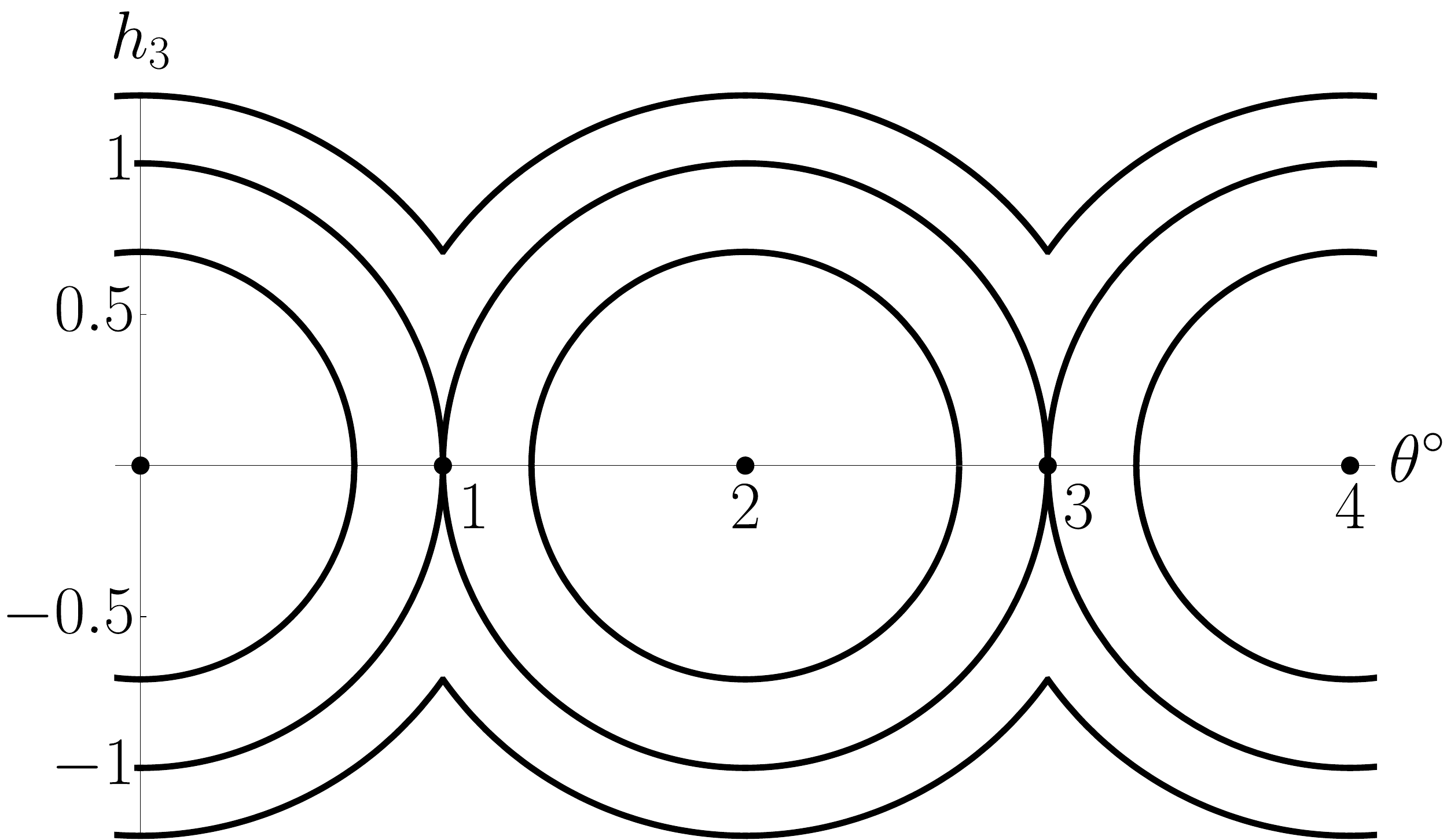}{Phase portrait of \eq{Ham3}: $\ell_\infty$, Case 4}{fig:phase4li}{0.35}{0.47}

General singular extremals are $(\tho, h_3) = (2n+1, 0)$, $n \in \Z$, and special singular extremals are $(\tho, h_3) = (2n, 0)$, $n \in \Z$. 
The circles enter the points $(\tho, h_3) = (1, 0)$ and $(3,0)$ for finite time. Thus
at these points  bang extremals join {(general)} singular extremals, there appear mixed extremals.

\medskip \textbf{Case 5: $a = 1$, $b = -1$.} We have $\Un(\tho {+k}) = (\tho)^2 - \tho + 1/2$ {for $\tho\in[0;1]$, $k\in\Z$}, see Fig.~\ref{fig:U5li}. In the strip $(0, 1) \times \R$ system~\eq{Ham3} has the form
$
\dot \tho = h_3$, $\dot h_3 = 1-2\tho$,
it has the center phase portrait (see Fig.~\ref{fig:phase5li}) and the bang trajectories --- arcs of ellipses
\begin{align*}
&\tho = C_1 \cos (\sqrt 2 t) + C_2 \sin (\sqrt 2 t) + 1/2,\\
&h_3 = -C_1 \sqrt 2 \sin (\sqrt 2 t) + C_2 \sqrt 2 \cos(\sqrt 2 t), \\
&C_1 = \tho_0-1/2, \qquad C_2 = h_3^0/\sqrt 2.
\end{align*} 

\twofiglabels
{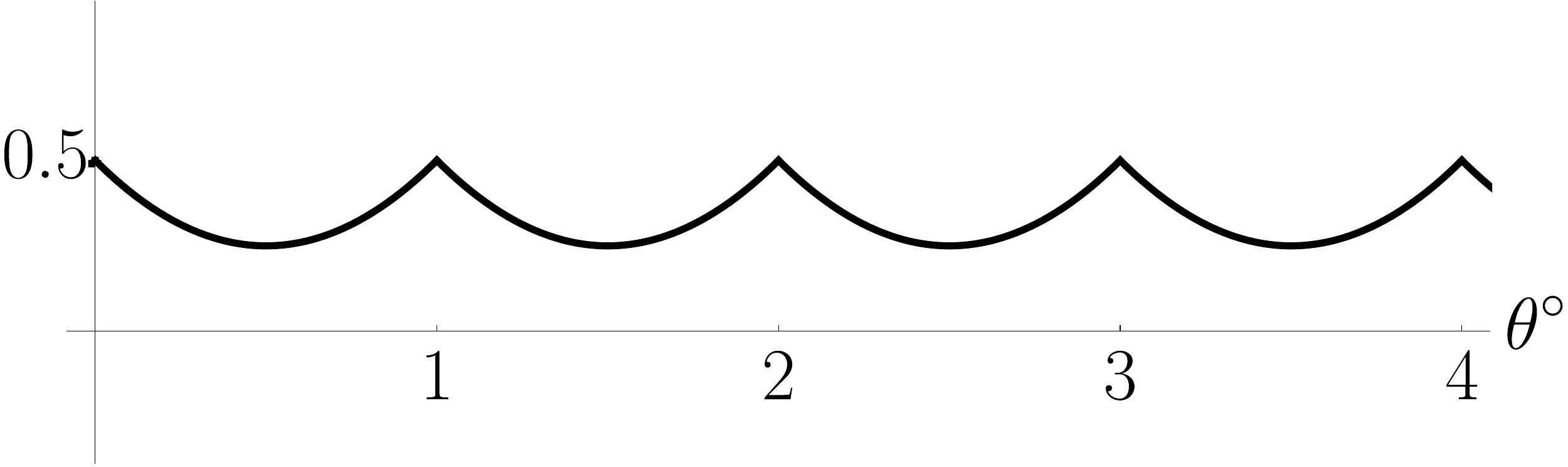}{Plot of $\Un(\tho)$: $\ell_\infty$, Case 5}{fig:U5li}
{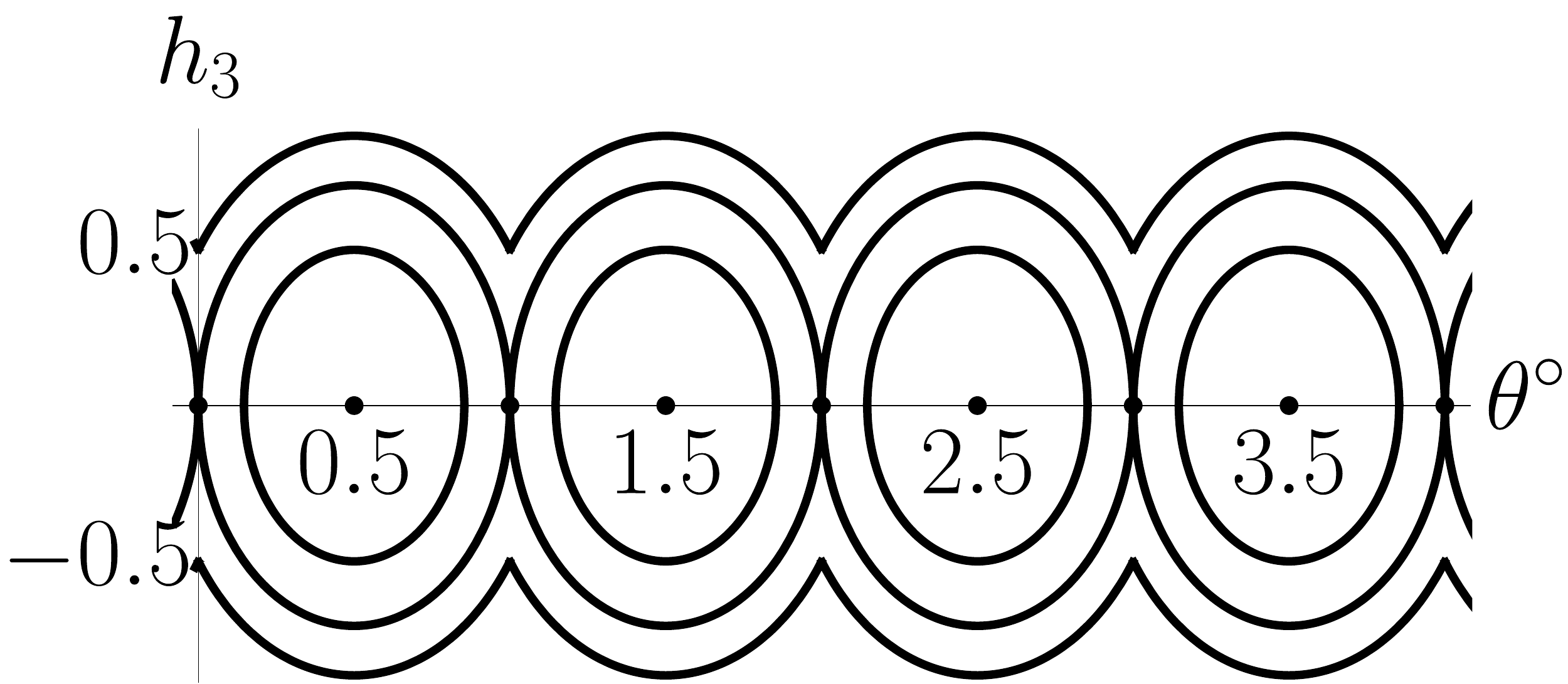}{Phase portrait of \eq{Ham3}: $\ell_\infty$, Case 5}{fig:phase5li}{0.47}{0.35}

General singular extremals are $(\tho, h_3) = (n, 0)$, $n \in \Z$. {There are no special singular extremals.}
The ellipses enter the equilibria $(\tho, h_3) = (n,0)$ for finite time. Thus
at these points  bang extremals join singular extremals, there appear mixed extremals.

\subsection{Example: $\Omega = \{ \|u\|_{1} \leq 1\}$}\label{subsec:l1} 
Let $\Omega = \{(u_1, u_2) \in \R^2 \mid |u_1|+  |u_2| \leq 1\}$. 
Then $\Omega^{\polar} = \{(h_1, h_2) \in \R^2 \mid |h_1|,\ |h_2| \leq 1\}$ and $\mathbb S(\Omega^{\polar}) = 4$. 
The functions $\soo \tho$ and $\coo \tho$ are piecewise linear and 8-periodic, they are plotted at Figs.~\ref{fig:sinl1} and~\ref{fig:cosl1}  (these plots were first given in~\cite{CT1}). 

\twofiglabel{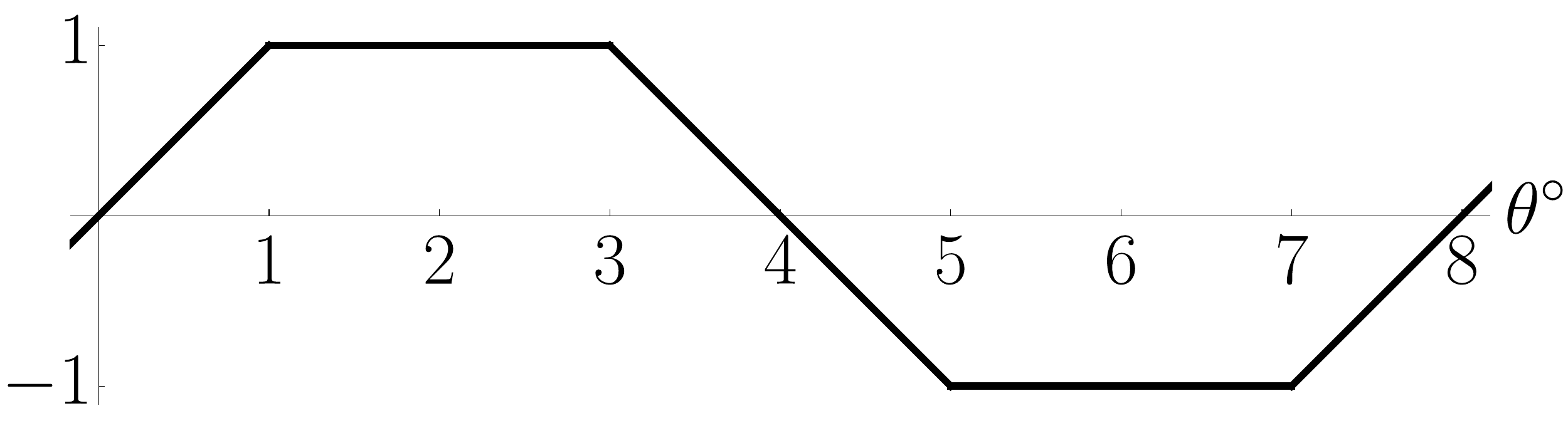}{Plot of $\soo \tho$: $\ell_1$}{fig:sinl1}{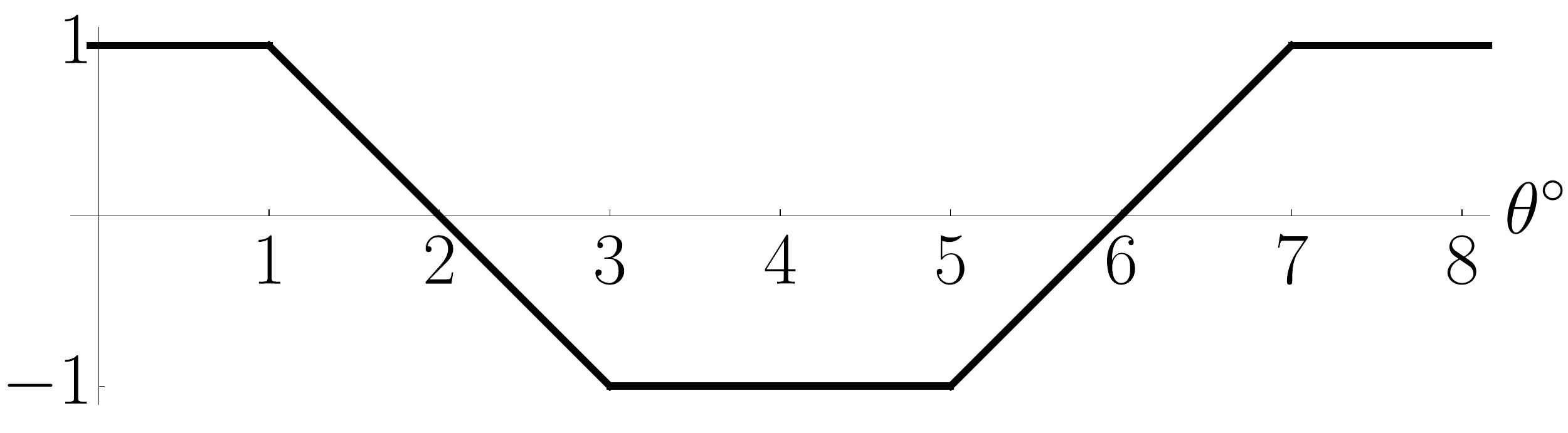}{Plot of $\coo \tho$: $\ell_1$}{fig:cosl1}

The functions $\soo^2 \tho$, $\coo^2 \tho$ are 4-periodic, even with respect to $\tho = 2$, with
$$
\soo^2 \tho = 
\begin{cases}
(\tho)^2, & \tho \in [0, 1], \\
1, & \tho \in [1, 2], 
\end{cases}
\qquad
\coo^2 \tho = 
\begin{cases}
1, & \tho \in [0, 1], \\
 (\tho-2)^2, & \tho \in [1, 2],
\end{cases}
$$
see Fig.~\ref{fig:sin2l1} and   Fig.~\ref{fig:cos2l1}.

\twofiglabel{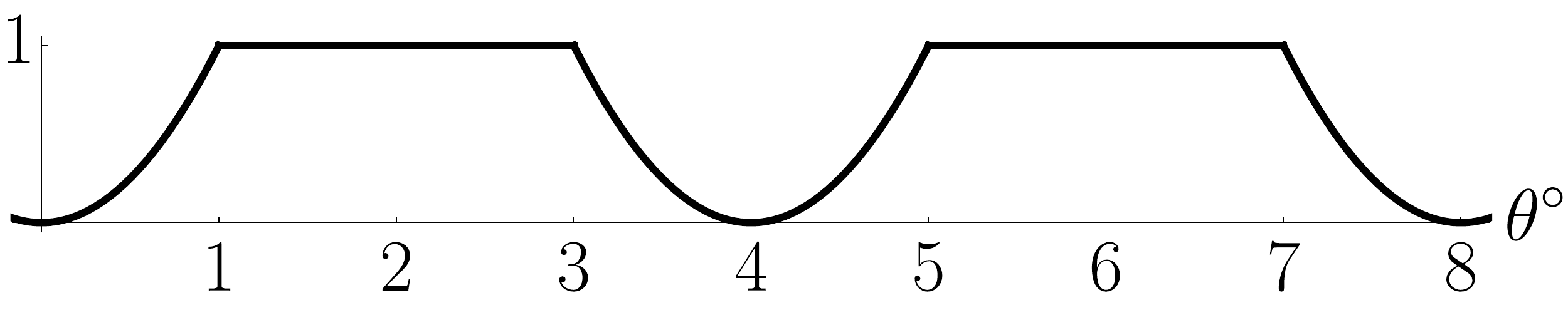}{Plot of $\soo^2 \tho$: $\ell_1$}{fig:sin2l1}{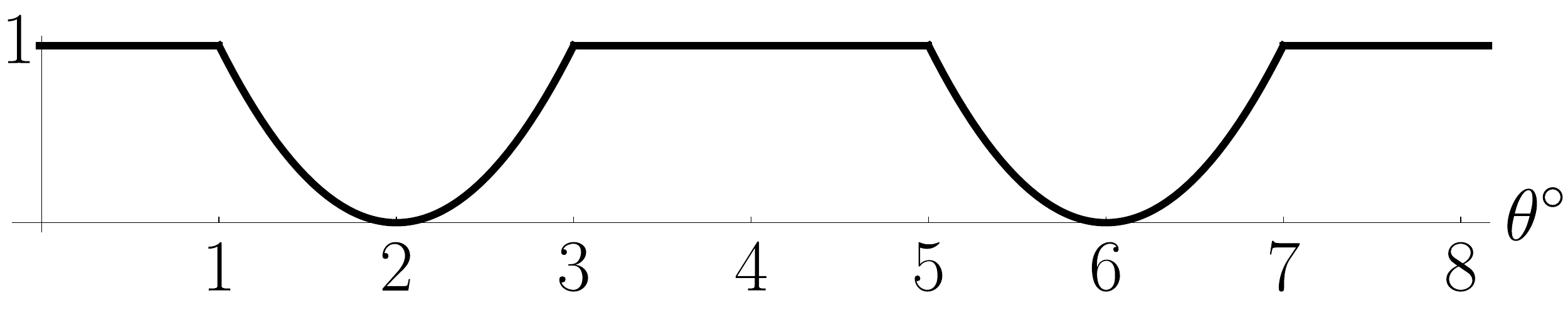}{Plot of $\coo^2 \tho$: $\ell_1$}{fig:cos2l1}  

{Since the set $\Omega$ has neither horizontal nor vertical edges, all singular extremals are of general type (see Subsec.~\ref{sec:singular_lie}).}

\medskip \textbf{Case 1: $a = -1$, $b = 1$.} The function $\Un(\tho)$ is 2-periodic  { and even}, with
$\Un(\tho) = -(1+(\tho)^2)/2$ for $\tho \in [{-1},1]$,
see Fig.~\ref{fig:U1l1}.  In the strip $(0, 1) \times \R$ system~\eq{Ham3} has the form
$
\dot \tho = h_3$, $\dot h_3 = \tho$,
it has the saddle phase portrait (see Fig.~\ref{fig:phase1l1}) and the bang trajectories
\begin{align*}
&\tho = C_1 \cosh t + C_2 \sinh  t,\\
&h_3 = C_1 \sinh t + C_2 \cosh  t, \\
&C_1 = \tho_0, \qquad C_2 = h_3^0.
\end{align*}

\twofiglabel
{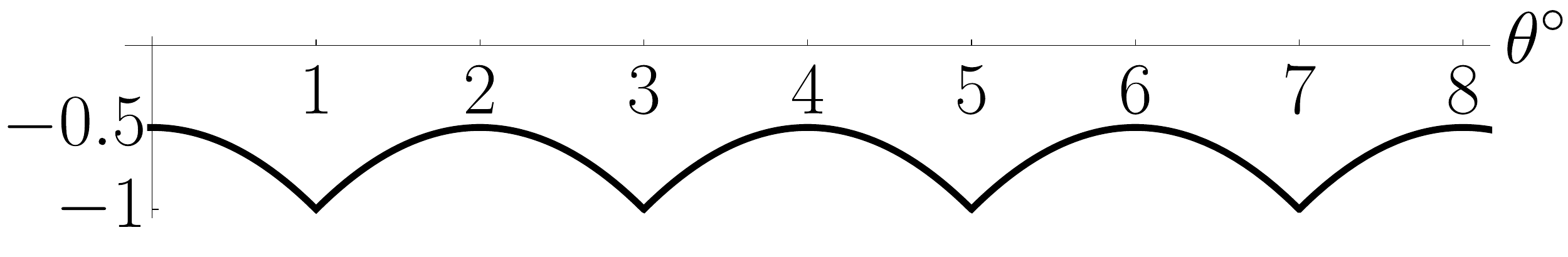}{Plot of $\Un(\tho)$: $\ell_1$, Case 1}{fig:U1l1}
{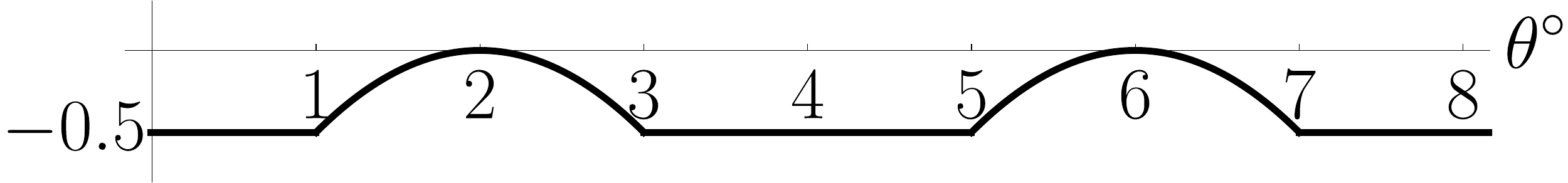}{Plot of $\Un(\tho)$: $\ell_1$, Case 2}{fig:U2l1}

General singular extremals are $(\tho, h_3) = (2n+1, 0)$, $n \in \Z$. 
There are no mixed extremals.  

\medskip \textbf{Case 2: $a = 0$, $b = 1$.} The function $\Un(\tho)$ is 4-periodic {and even}, with
$$
\Un(\tho) =  
\begin{cases}
-1/2, & \tho \in [0, 1], \\
 -(\tho-2)^2/2, & \tho \in [1, 2],
\end{cases}
$$
see Fig.~\ref{fig:U2l1}.  
In the strip $(0, 1) \times \R$ the phase portrait of system~\eq{Ham3} consists of bang trajectories --- horizontal segments
$
\tho = \tho_0 + h_3^0t$, $h_3 \equiv h_3^0 \neq 0
$
and equilibria
$
\tho \equiv \const \in (0, 1)$, $h_3 = 0$.
In the strip $(1, 2) \times \R$ system~\eq{Ham3} has the form
$
\dot \tho = h_3$, $\dot h_3 = \tho-2$,
it has the saddle phase portrait (see Fig.~\ref{fig:phase2l1}) and the bang trajectories
\begin{align*}
&\tho = C_1 \cosh t + C_2 \sinh  t + 2,\\
&h_3 = C_1 \sinh t + C_2 \cosh  t, \\
&C_1 = \tho_0 - 2, \qquad C_2 = h_3^0.
\end{align*} 

\twofiglabels
{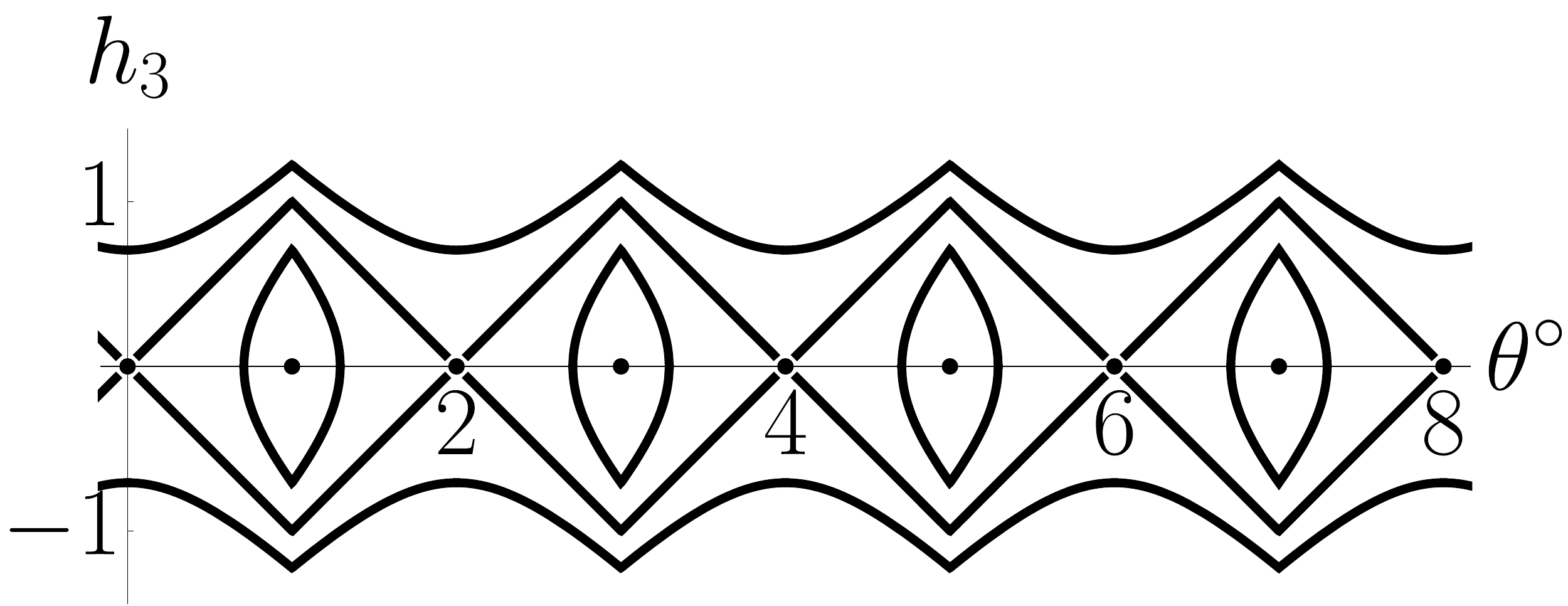}{Phase portrait of \eq{Ham3}: $\ell_1$, Case 1}{fig:phase1l1}
{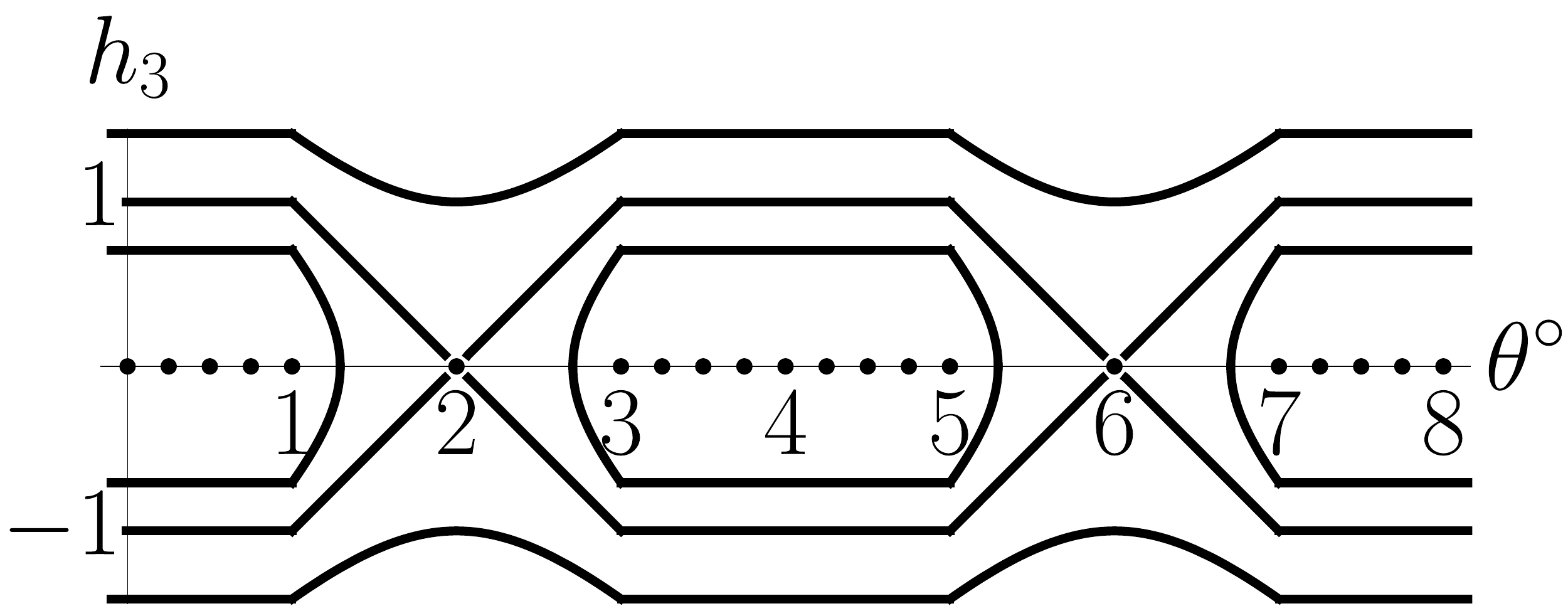}{Phase portrait of \eq{Ham3}: $\ell_1$, Case 2}{fig:phase2l1}{0.35}{0.35}

General singular extremals are $(\tho, h_3) = (2n+1, 0)$, $n \in \Z$. 
There are no mixed extremals. 

\medskip \textbf{Case 3: $a = 1$, $b = 1$.} The function $\Un(\tho)$ is 4-periodic, even  and odd w.r.t. $\tho = 1$, with $\Un(\tho) = ((\tho)^2 -1)/2$ for $\tho \in [0, 1]$.
see Fig.~\ref{fig:U3l1}.  
In the strip $(0, 1) \times \R$ system~\eq{Ham3} has the form
$
\dot \tho = h_3$, $\dot h_3 = -\tho$,
it has the center phase portrait  and the bang trajectories --- arcs of circles
\begin{align*}
&\tho = C_1 \cos t + C_2 \sin  t,\\
&h_3 = -C_1 \sin t + C_2 \cos  t, \\
&C_1 = \tho_0, \qquad C_2 = h_3^0.
\end{align*}
The circles enter the point $(\tho, h_3) = (1, 0)$ for finite time. 
In the strip $(1, 2) \times \R$ system~\eq{Ham3} has the form
$
\dot \tho = h_3$, $\dot h_3 = \tho-2$,
it has the saddle phase portrait  (see Fig.~\ref{fig:phase3l1}) and the bang trajectories
\begin{align*}
&\tho = C_1 \cosh t + C_2 \sinh  t + 2,\\
&h_3 = C_1 \sinh t + C_2 \cosh  t, \\
&C_1 = \tho_0 - 2, \qquad C_2 = h_3^0.
\end{align*}

\twofiglabel
{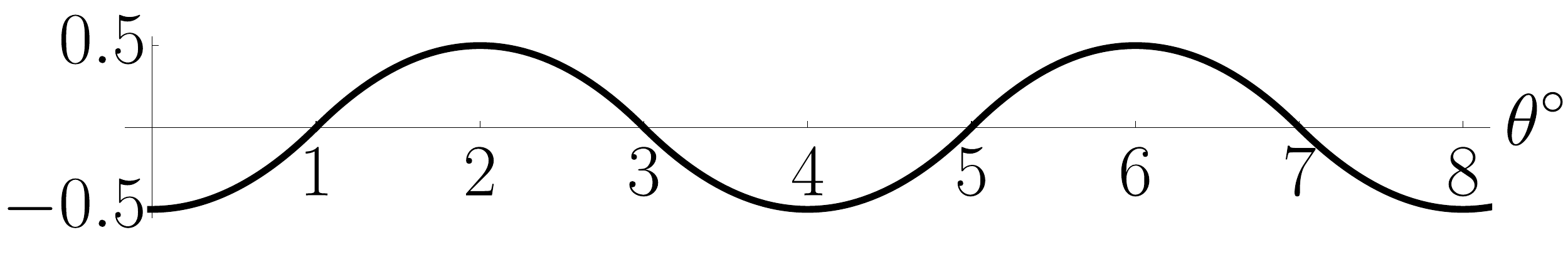}{Plot of $\Un(\tho)$: $\ell_1$, Case 3}{fig:U3l1}
{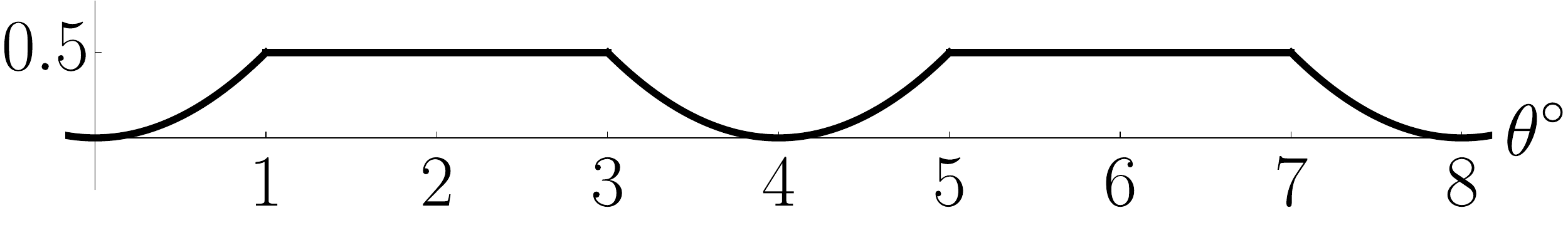}{Plot of $\Un(\tho)$: $\ell_1$, Case 4}{fig:U4l1}

There are neither singular nor mixed extremals. 

\medskip \textbf{Case 4: $a = 1$, $b = 0$.} The function $\Un(\tho)$ is 4-periodic {and even}, with
$$
\Un(\tho) =  
\begin{cases}
 (\tho)^2/2, & \tho \in [0, 1], \\
 1/2, & \tho \in [1, 2],
\end{cases}
$$
see Fig.~\ref{fig:U4l1}.  
In the strip $(0, 1) \times \R$ system~\eq{Ham3} has the form
$
\dot \tho = h_3$, $\dot h_3 = -\tho$,
it has the center phase portrait  and the bang trajectories --- arcs of circles
\begin{align*}
&\tho = C_1 \cos t + C_2 \sin  t,\\
&h_3 = -C_1 \sin t + C_2 \cos  t, \\
&C_1 = \tho_0, \qquad C_2 = h_3^0.
\end{align*}
In the strip $(1, 2) \times \R$ the phase portrait of system~\eq{Ham3} consists of horizontal segments
$
\tho = \tho_0 + h_3^0t$, $h_3 \equiv h_3^0\neq 0$
and equilibria
$
\tho = \const \in (1, 2)$, $h_3 \equiv 0$,
see Fig.~\ref{fig:phase4l1}.

\twofiglabels
{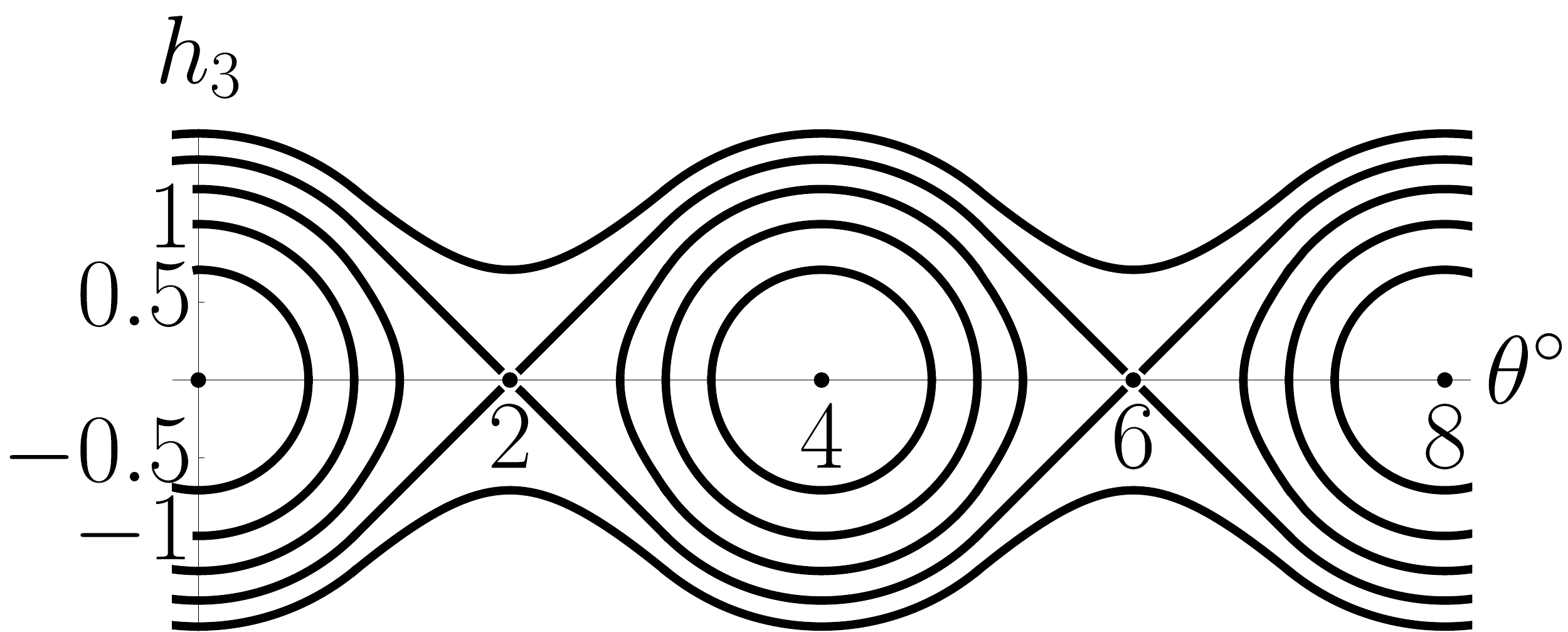}{Phase portrait of \eq{Ham3}: $\ell_1$, Case 3}{fig:phase3l1}
{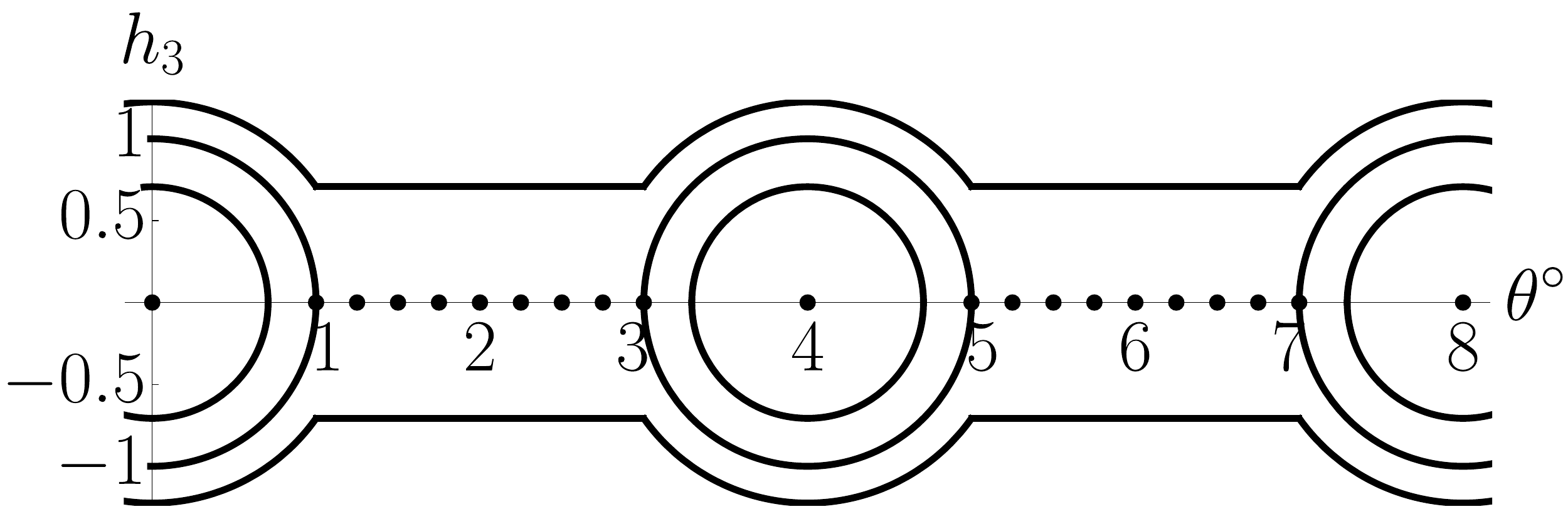}{Phase portrait of \eq{Ham3}: $\ell_1$, Case 4}{fig:phase4l1}{0.35}{0.47}

General singular extremals are $(\tho, h_3) = (2n+1, 0)$, $n \in \Z$. 
The circles enter the points $(\tho, h_3) = (2n+1, 0)$ for finite time. Thus
at these points  bang extremals join singular extremals, there appear mixed extremals.

\medskip \textbf{Case 5: $a = 1$, $b = -1$.} The function $\Un(\tho)$ is 2-periodic { and even}, with
$\Un(\tho) = (1+(\tho)^2)/2$ for $\tho \in [0,1]$,
see Fig.~\ref{fig:U5l1}.  In the strip $(0, 1) \times \R$ system~\eq{Ham3} has the form
$
\dot \tho = h_3$, $\dot h_3 = \tho$,
it has the center phase portrait (see Fig.~\ref{fig:phase5l1}) and the bang trajectories --- arcs of circles
\begin{align*}
&\tho = C_1 \cos t + C_2 \sin  t,\\
&h_3 = -C_1 \sin t + C_2 \cos  t, \\
&C_1 = \tho_0, \qquad C_2 = h_3^0.
\end{align*}

\twofiglabels
{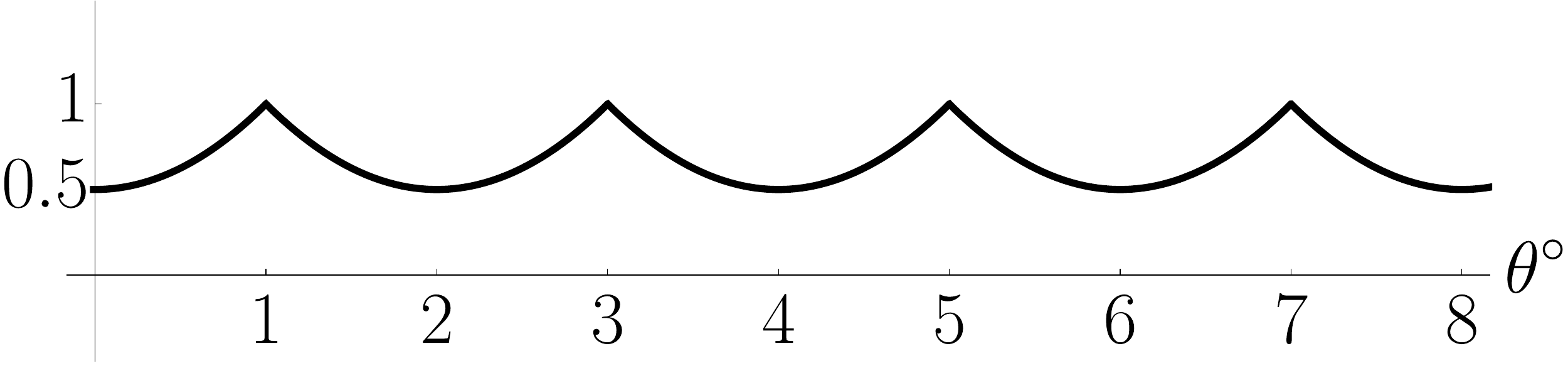}{Plot of $\Un(\tho)$: $\ell_1$, Case 5}{fig:U5l1}
{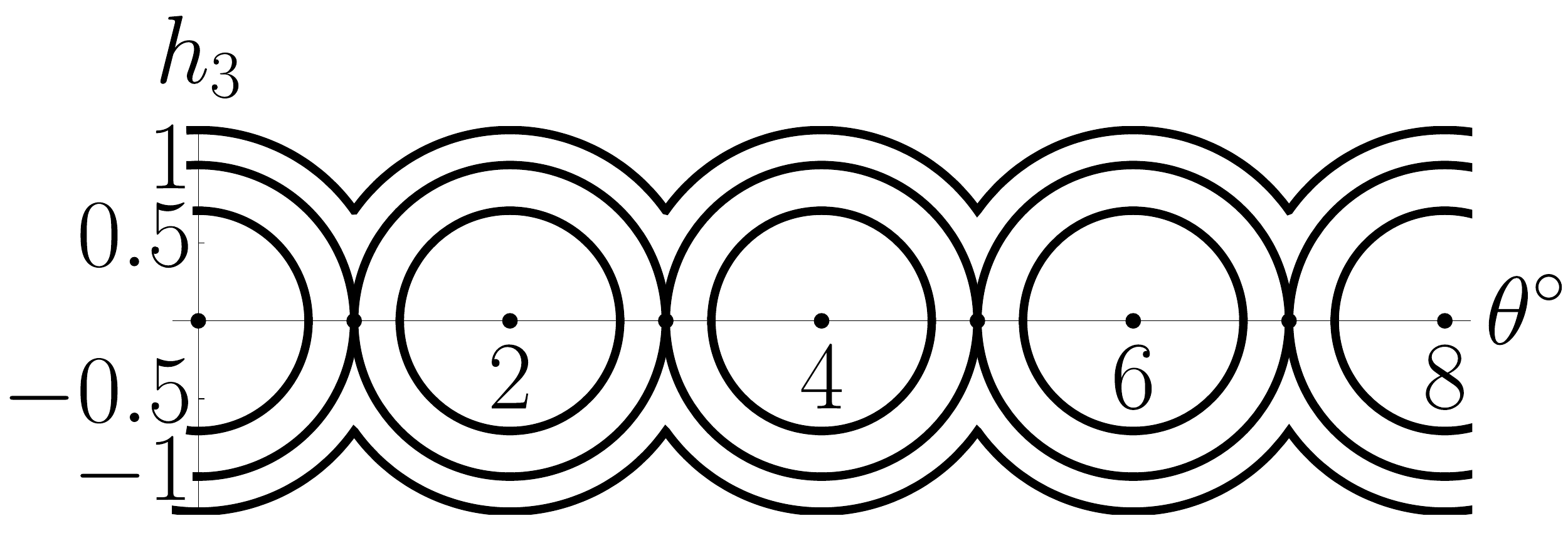}{Phase portrait of \eq{Ham3}: $\ell_1$, Case 5}{fig:phase5l1}{0.47}{0.35}

General singular extremals are $(\tho, h_3) = (2n+1, 0)$, $n \in \Z$. 
The circles enter the points $(\tho, h_3) = (2n+1, 0)$ for finite time. Thus
at these points  bang extremals join singular extremals, there appear mixed extremals.

\subsection{Example: $\Omega = \{ \|u\|_{p} \leq 1\}$, $1 < p < \infty$}\label{subsec:lp} 
Let $\Omega = \{(u_1, u_2) \in \R^2 \mid |u_1|^p+  |u_2|^p \leq 1\}$, and let  $1 < p < \infty$, then $\Omega$ is a strictly convex domain. 
Then $\Omega^{\polar} = \{(h_1, h_2) \in \R^2 \mid |h_1|^q+ |h_2|^q \leq 1\}$, $q = p/(1-p)\in (1, \infty)$, and $\mathbb S^{\polar} = \mathbb S(\Omega^{\polar}) = 4 \Gamma^2(1+1/q)/\Gamma(1+2/q)$. 
The function $\soo \tho$ has period $2 \mathbb S^{\polar}$, is odd w.r.t. $\tho =  \mathbb S^{\polar}$ and even w.r.t. $\tho =  \mathbb S^{\polar}/2$. 
The function $\coo \tho$ has period $2 \mathbb S^{\polar}$, is even w.r.t. $\tho =  \mathbb S^{\polar}$ and odd w.r.t. $\tho =  \mathbb S^{\polar}/2$. 
If $1< q < 2$, then $\soo(\mathbb S^{\polar}/2) < 1/\sqrt{2}$ and $\coo(\mathbb S^{\polar}/2) < 1/\sqrt{2}$. If $2< q < \infty$, then $\soo(\mathbb S^{\polar}/2) > 1/\sqrt{2}$ and $\coo(\mathbb S^{\polar}/2) > 1/\sqrt{2}$.
The functions $\soo \tho$
and $\coo \tho$ are plotted for $q = 5/4$ 
at Figs.~\ref{fig:sinlpm} and~\ref{fig:coslpm}, and for $q = 4$ 
at Figs.~\ref{fig:sinlpb} and~\ref{fig:coslpb}.

In the $\ell_p$ case the functions $\soo \tho$ and $\coo \tho$ are parametrized by hypergeometric functions and system~\eq{Ham3} seems hard to be explicitly integrable\footnote{Nonetheless it is Liouville integrable as we have shown.}; although, the qualitative structure of the phase portrait is easily constructed below. The sub-Riemannian case $\Omega = \{\|u\|_{2} \leq 1\}$ leads to the pendulum equation and is integrable in elliptic functions. 

\twofiglabels{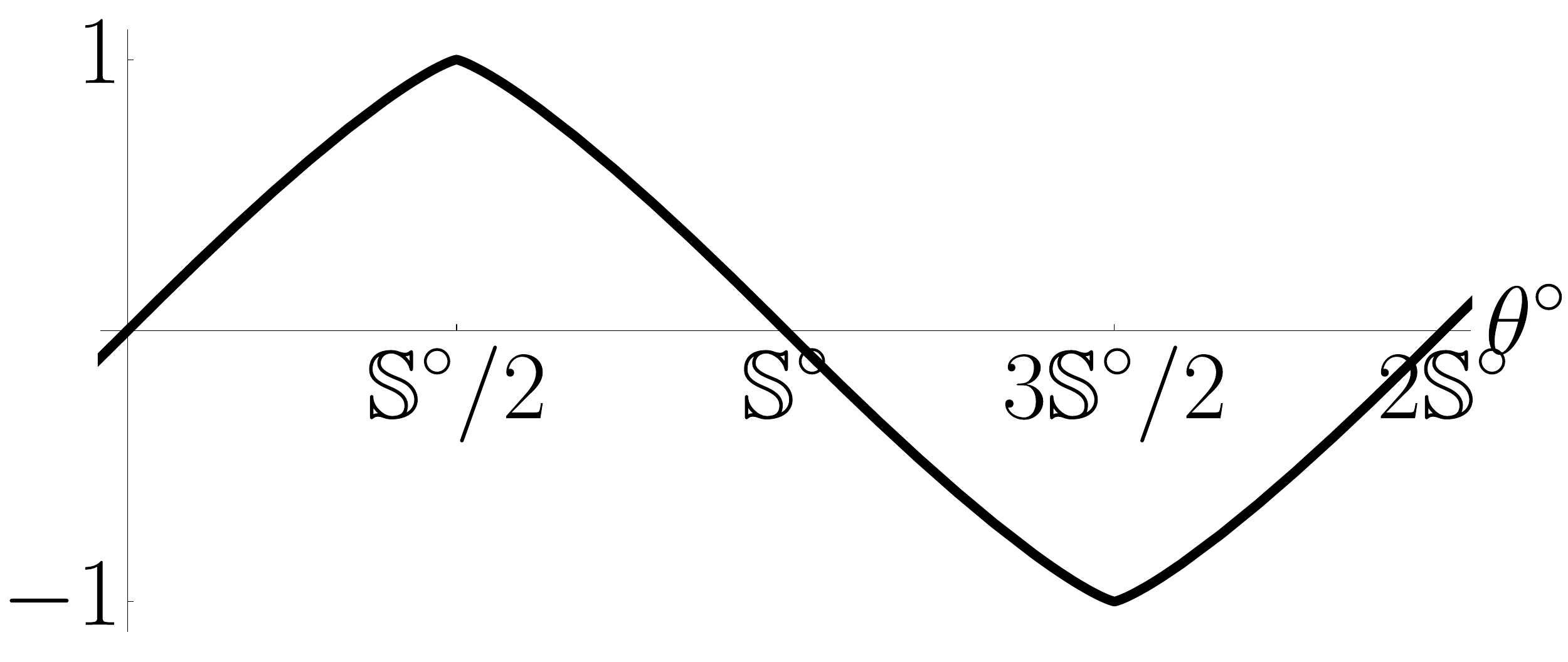}{Plot of $\soo \tho$: $\ell_p$, $q <2$}{fig:sinlpm}{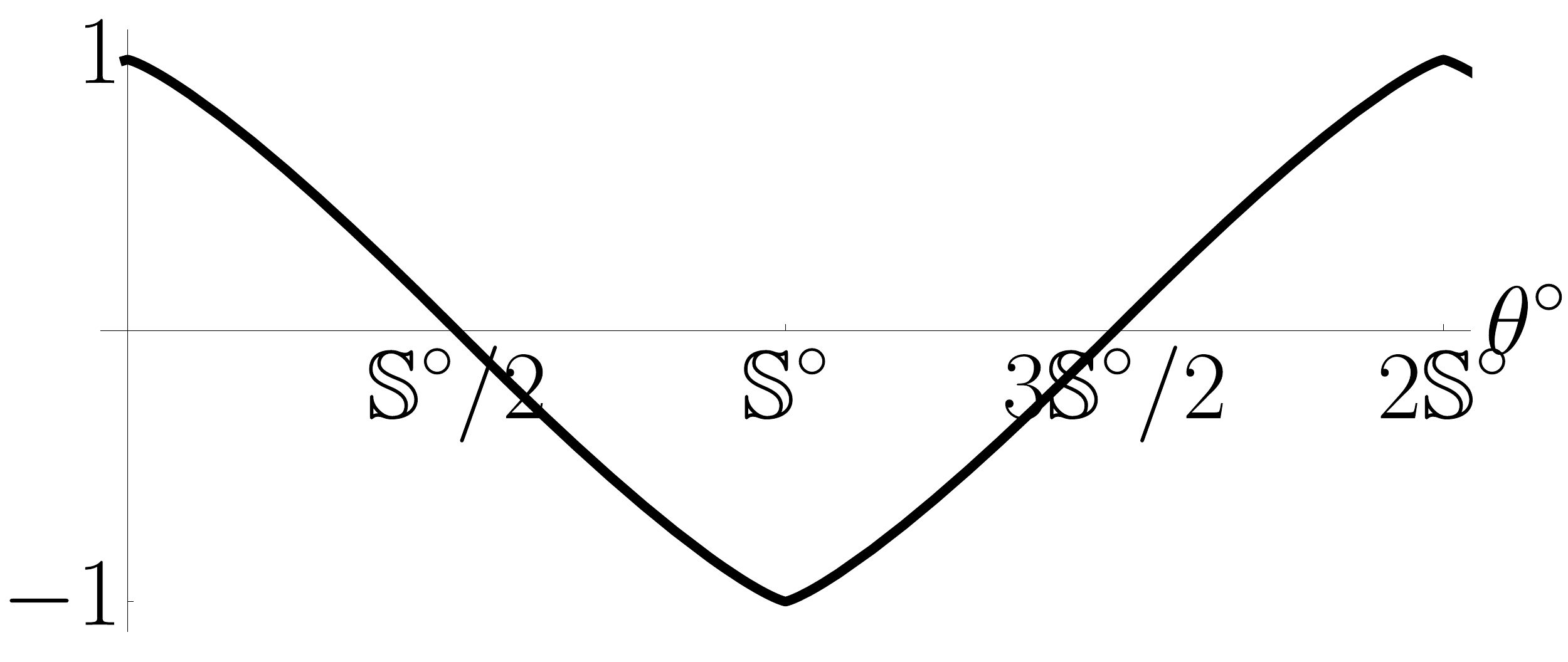}{Plot of $\coo \tho$: $\ell_p$, $q <2$}{fig:coslpm}{0.35}{0.35}

\twofiglabel{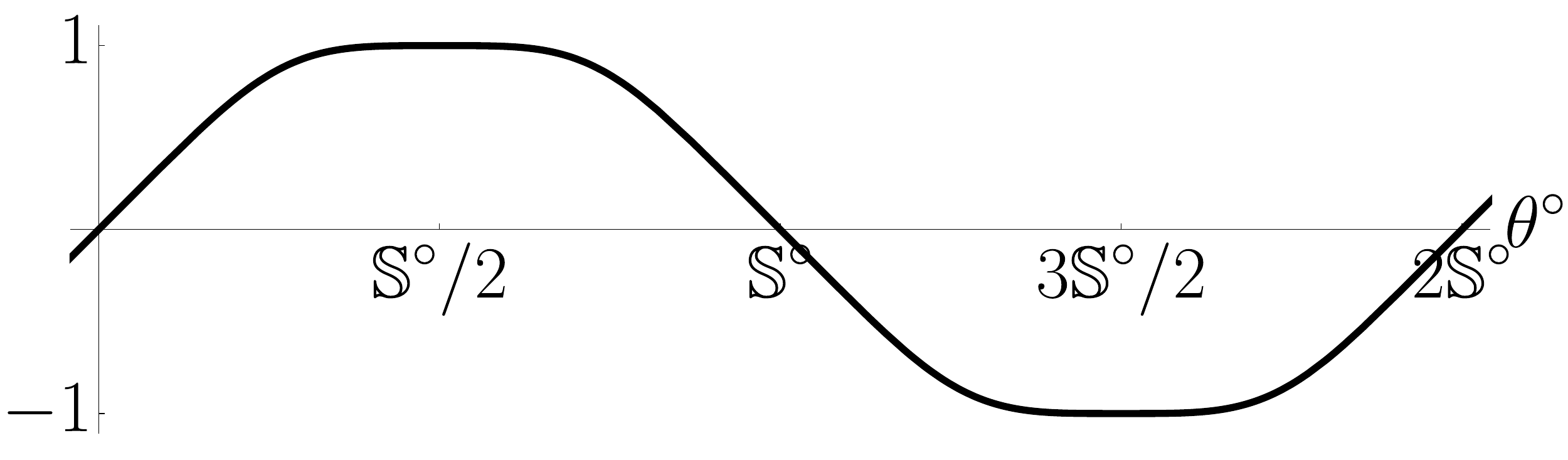}{Plot of $\soo \tho$: $\ell_p$, $q >2$}{fig:sinlpb}{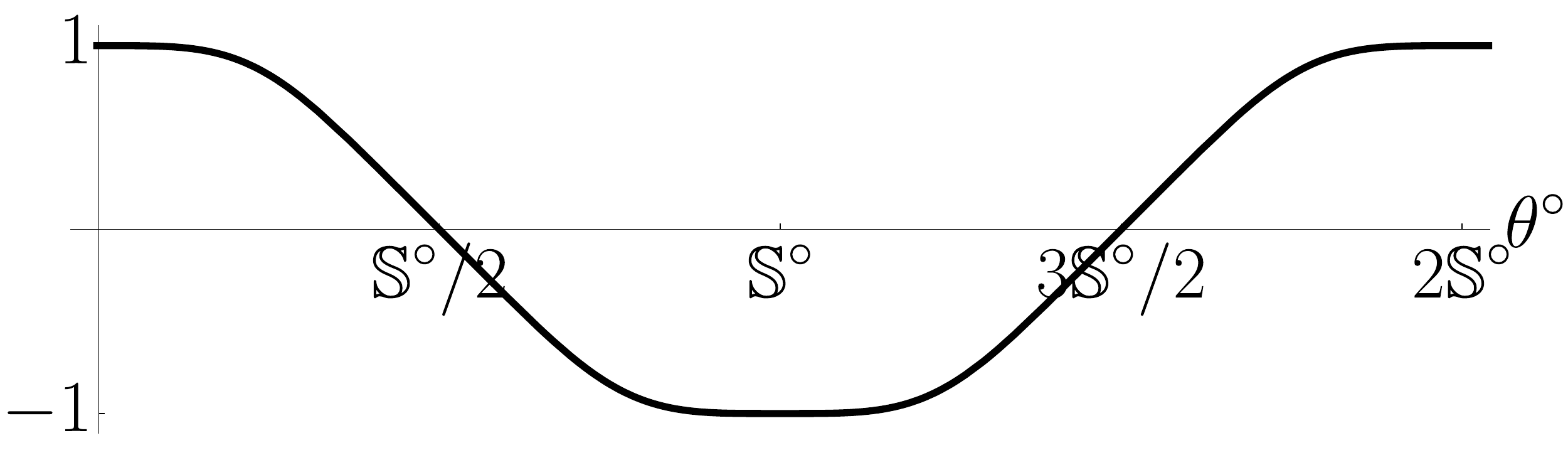}{Plot of $\coo \tho$: $\ell_p$, $q >2$}{fig:coslpb}

The functions $\soo^2 \tho$, $\coo^2 \tho$ have period $\mathbb S^{\polar}$ and are even. 
The functions $\soo^2 \tho$
and $\coo^2 \tho$ are plotted for $q = 5/4$ 
at Figs.~\ref{fig:sin2lpm} and~\ref{fig:cos2lpm}, and for $q = 4$ 
at Figs.~\ref{fig:sin2lpb} and~\ref{fig:cos2lpb}.

\twofiglabels{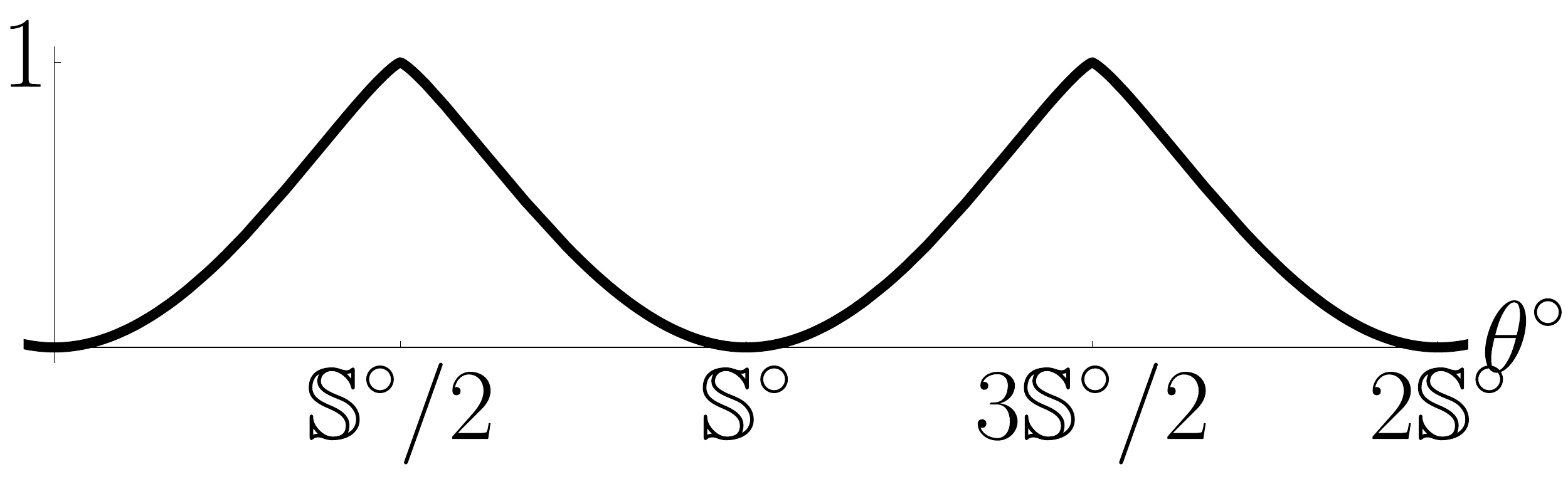}{Plot of $\soo^2 \tho$: $\ell_p$, $q <2$}{fig:sin2lpm}{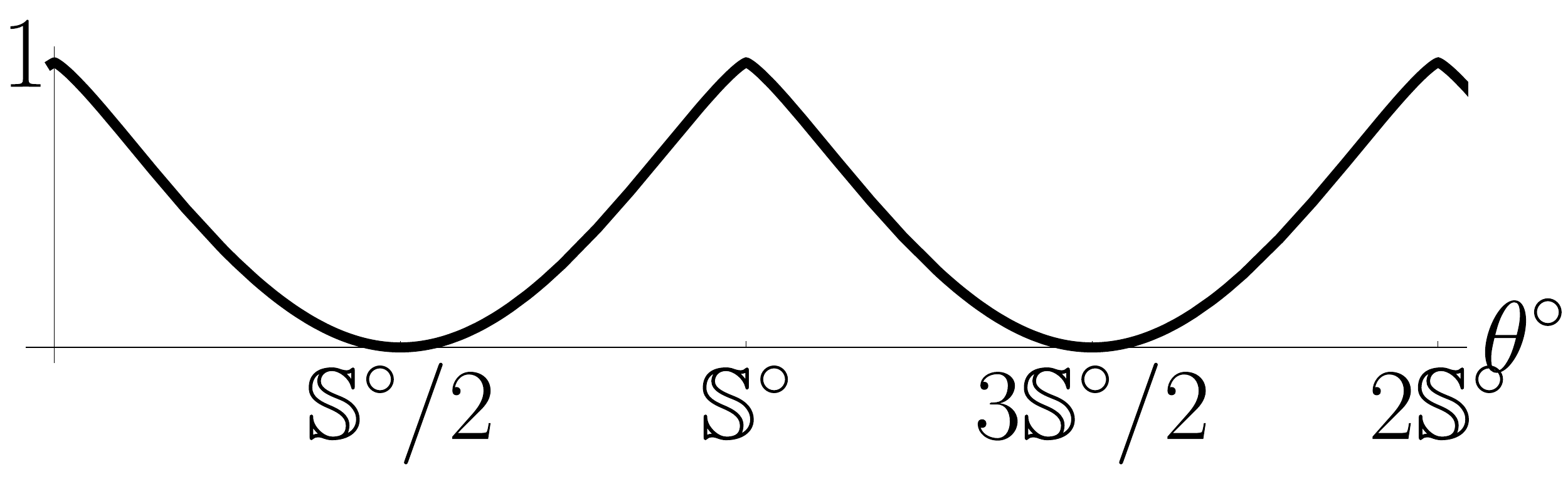}{Plot of $\coo^2 \tho$: $\ell_p$, $q <2$}{fig:cos2lpm}{0.35}{0.35}

\twofiglabel{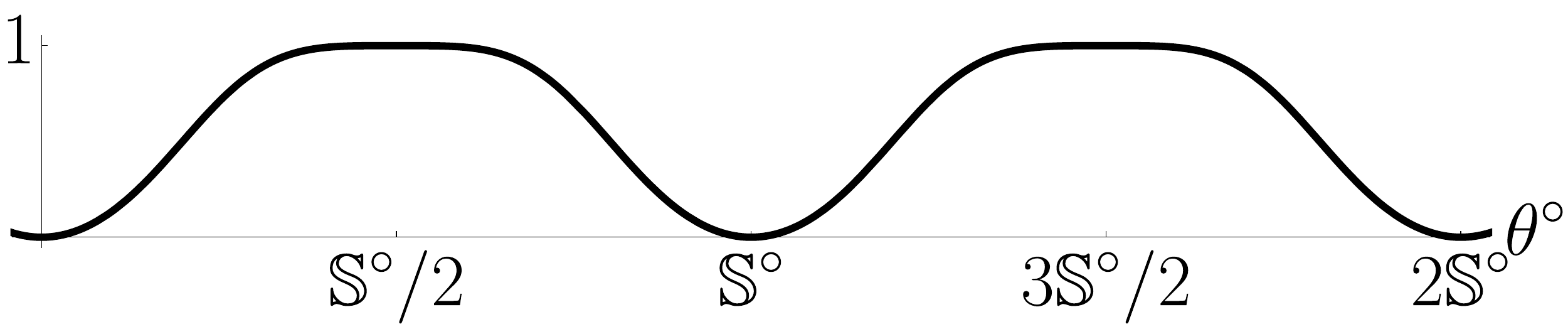}{Plot of $\soo^2 \tho$: $\ell_p$, $q >2$}{fig:sin2lpb}{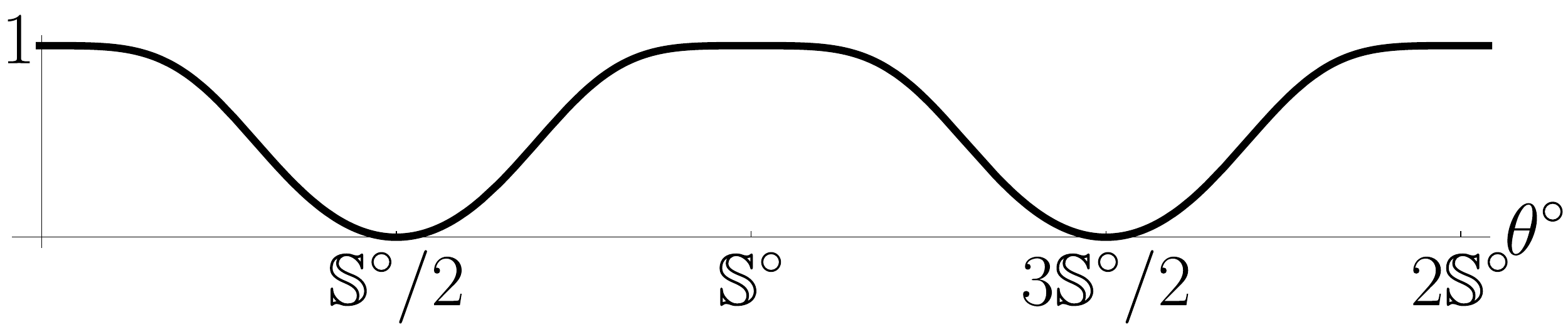}{Plot of $\coo^2 \tho$: $\ell_p$, $q >2$}{fig:cos2lpb}

In the case $q \geq 2$ the boundary $\partial \Omega^{\polar}$ and the Hamiltonian $H$ are $C^2$-smooth.
By Th.~\ref{th:Lie-strict}, 
 all extremals are bang and smooth. There is no Cauchy nonuniqueness of extremals.

In the case $q \in (1, 2)$ the boundary $\partial\Omega^{\polar}$ and the Hamiltonian $H$ are $C^2$-smooth if $h_1h_2 \neq 0$, and belong to $C^1\setminus C^2$ if $h_1h_2 = 0$, $h_1^2+h_2^2\neq 0$. In this case there are  singular  {and} mixed extremals. Cauchy nonuniqueness due to nonsmoothness appears in some cases, see Th.~\ref{th:lp} below. 

\medskip \textbf{Case 1: $a=-1$, $b=1$.}
The function $\Un(\tho) = - (\soo^2 \tho + \coo^2 \tho)/2$ has period $\mathbb S^{\polar}/2$ and is even with respect to $\tho = \mathbb S^{\polar}/4$.

Let $1 < q < 2$, then $\Un(\tho)$ has minima at $\tho = 0$ and $ \mathbb S^{\polar}/2$ {(which are singular extremals, since $\Un(\tho)$ is not twice differentiable for $\tho=0, \mathbb S^{\polar}/2$)},  and maximum at $\tho =  \mathbb S^{\polar}/4$ (this follows from the mutual disposition of the sphere $\Omega^{\polar}$ and the circles $h_1^2 + h_2^2 = \const$), see Fig.~\ref{fig:U1lpm}.
Thus system~\eq{Ham3} has center  equilibria at $(\tho, h_3) = (0, 0)$ and $(\tho, h_3) = (\mathbb S^{\polar}/2, 0)$, and  saddle equilibrium at $(\tho, h_3) = (\mathbb S^{\polar}/4, 0)$, see Fig.~\ref{fig:phase1lpm}.  The function $\Un(\tho)$ is $C^2$-smooth when $\tho \neq \So n/2$, thus the integral $I$~\eq{I} diverges, and there is no Cauchy nonuniqueness of extremals. {Hence, there are no mixed extremals.} 

\twofiglabels
{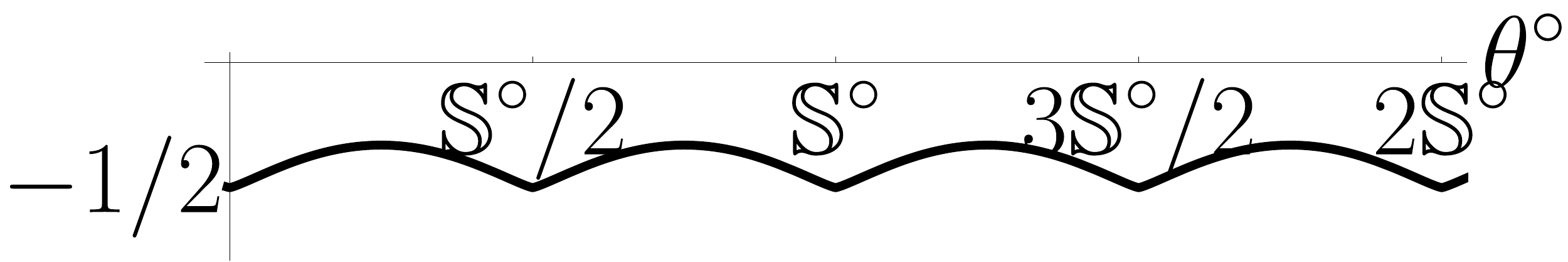}{{\small Plot of $\Un(\tho)$: $\ell_p$, $q<2$, Case 1}}{fig:U1lpm}
{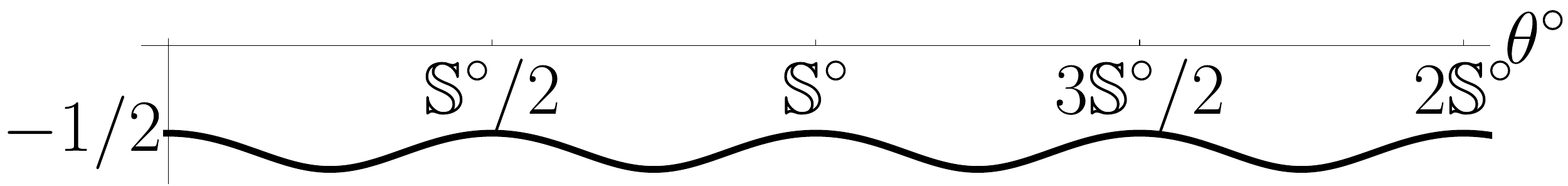}{{\small Plot of $\Un(\tho)$: $\ell_p$, $q>2$, Case 1}}{fig:U1lpb}{0.35}{0.47}

Let $2 \leq q < \infty$, then $\Un(\tho)$ has maxima at $\tho = 0$ and $ \mathbb S^{\polar}/2$,  and minimum at $\tho =  \mathbb S^{\polar}/4$, see~\ref{fig:U1lpb}. 
Thus system~\eq{Ham3} has saddle equilibria at $(\tho, h_3) = (0, 0)$ and $(\tho, h_3) = (\mathbb S^{\polar}/2, 0)$, and center equilibrium at $(\tho, h_3) = (\mathbb S^{\polar}/4, 0)$, see Fig.~\ref{fig:phase1lpb}. 

\twofiglabels
{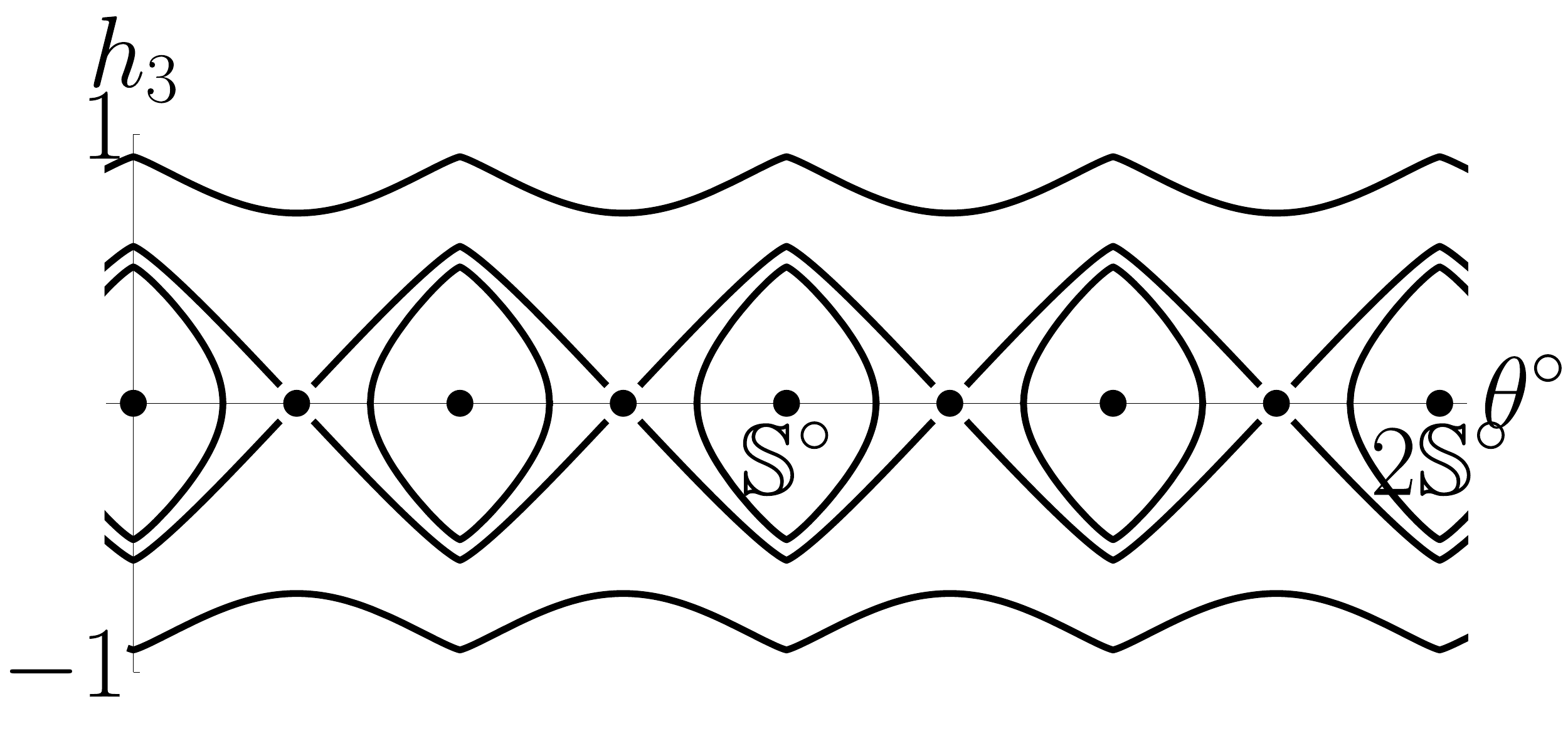}{Phase portrait of \eq{Ham3}: $\ell_p$, $q<2$, Case 1}{fig:phase1lpm}
{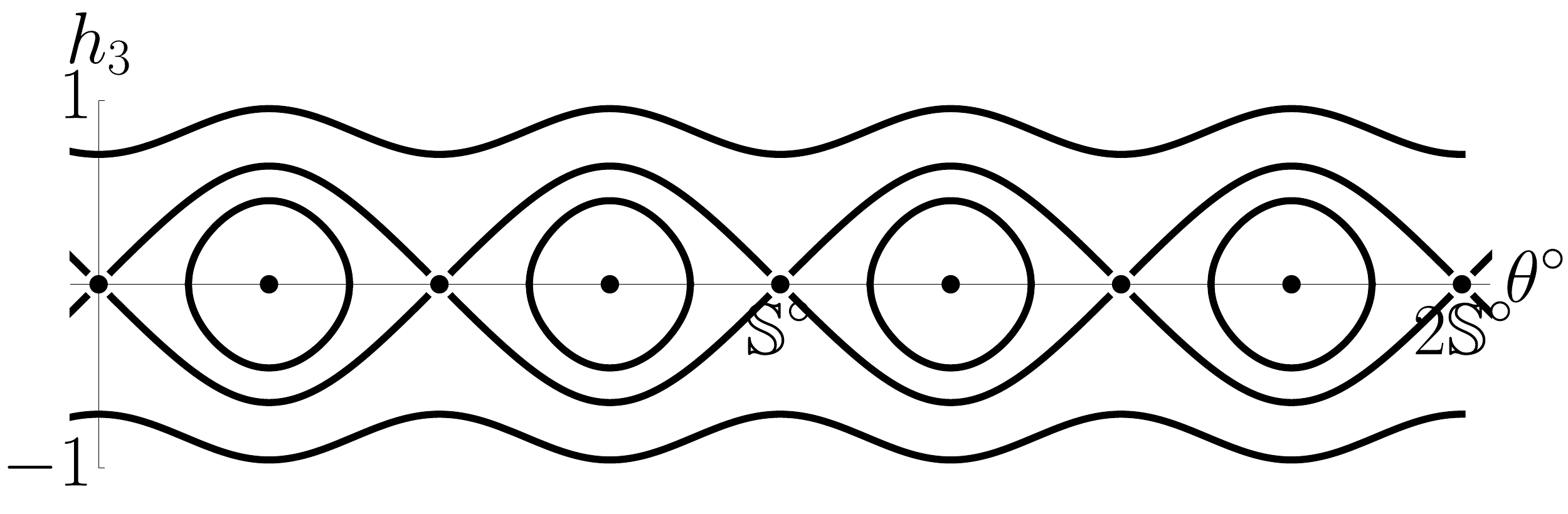}{Phase portrait of \eq{Ham3}: $\ell_p$, $q>2$, Case 1}{fig:phase1lpb}{0.35}{0.47}

\medskip \textbf{Case 2: $a=0$, $b=1$.}
The function $\Un(\tho)$ has period $\mathbb S^{\polar}$ and is even with respect to $\tho = \mathbb S^{\polar}/2$.  
It has minima at $\tho = 0$ and $ \mathbb S^{\polar}$,  and maximum at $\tho =  \mathbb S^{\polar}/2$, see Figs.~\ref{fig:U2lpm}, \ref{fig:U2lpb}.
Thus system~\eq{Ham3} has center  equilibria at $(\tho, h_3) = (0, 0)$ and $(\tho, h_3) = (\mathbb S^{\polar}, 0)$, and  saddle equilibrium at $(\tho, h_3) = (\mathbb S^{\polar}/2, 0)$, see Figs.~\ref{fig:phase2lpm}, \ref{fig:phase2lpb}.

In the case $1< q < 2$ the function $\Un(\tho)$ is $C^1 \setminus C^2$ at $\tho = \So/2$ {(which is a general singular extremal). There is no special singular extremals.}.
Since 
\be{ascos}
\cos_{\Omega^{\polar}} (\tho) \sim \So/2-\tho, \qquad
\tho \to \So/2,
\ee
then the integral $I$~\eq{I} diverges, and there is no Cauchy nonuniqueness of extremals. {Hence there are no mixed extremals.}

\twofiglabels
{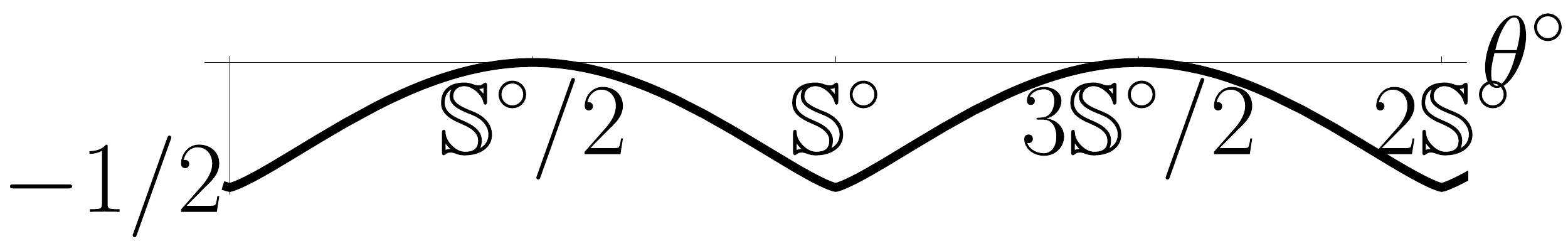}{{\small Plot of $\Un(\tho)$: $\ell_p$, $q<2$, Case 2}}{fig:U2lpm}
{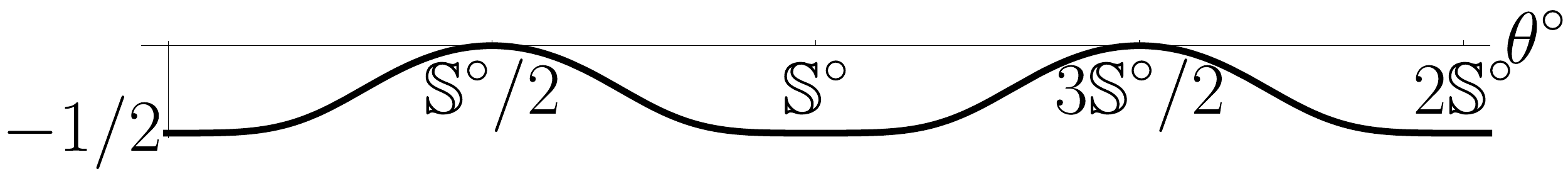}{{\small Plot of $\Un(\tho)$: $\ell_p$, $q>2$, Case 2}}{fig:U2lpb}{0.35}{0.47}

\twofiglabels
{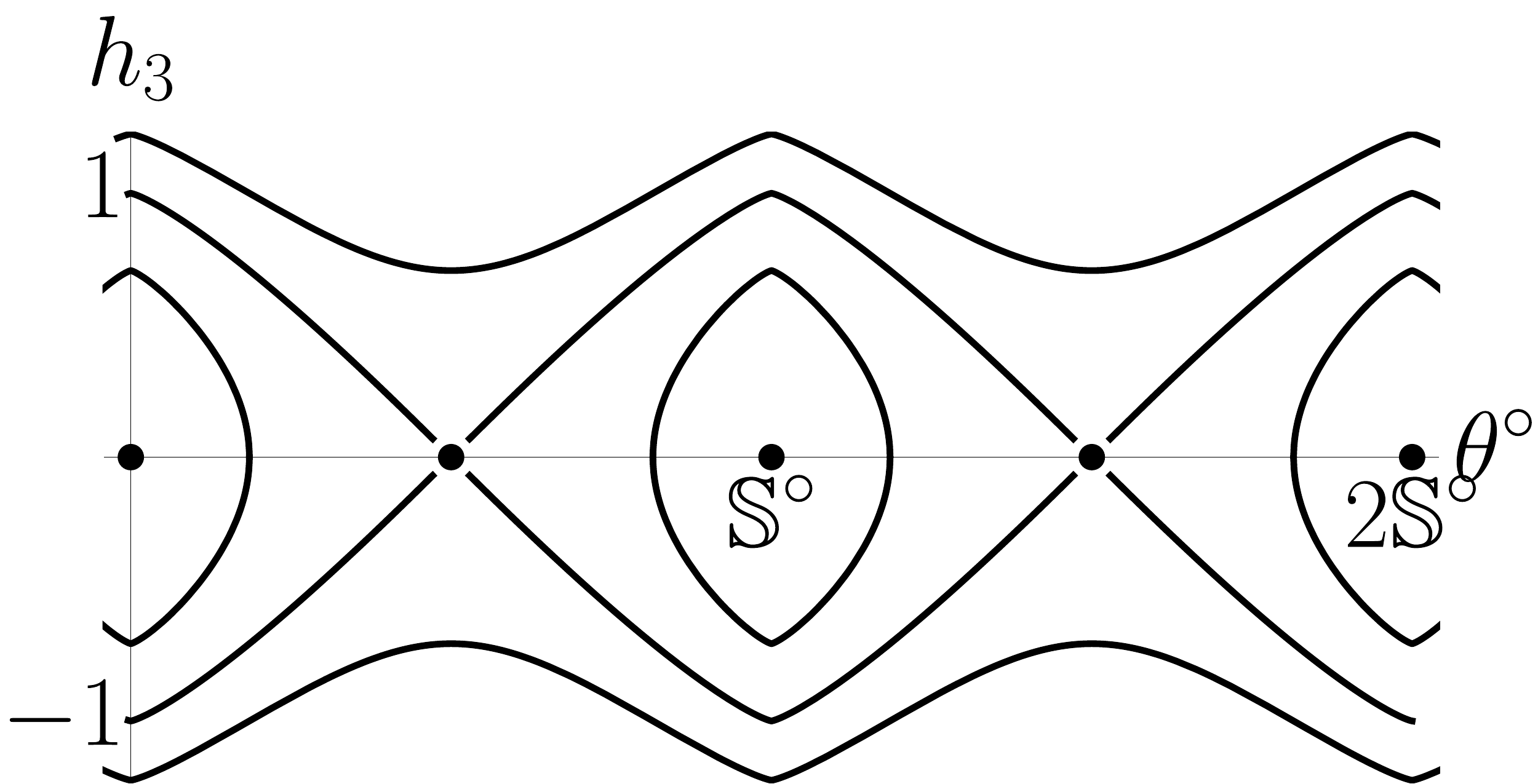}{Phase portrait of \eq{Ham3}: $\ell_p$, $q<2$, Case 2}{fig:phase2lpm}
{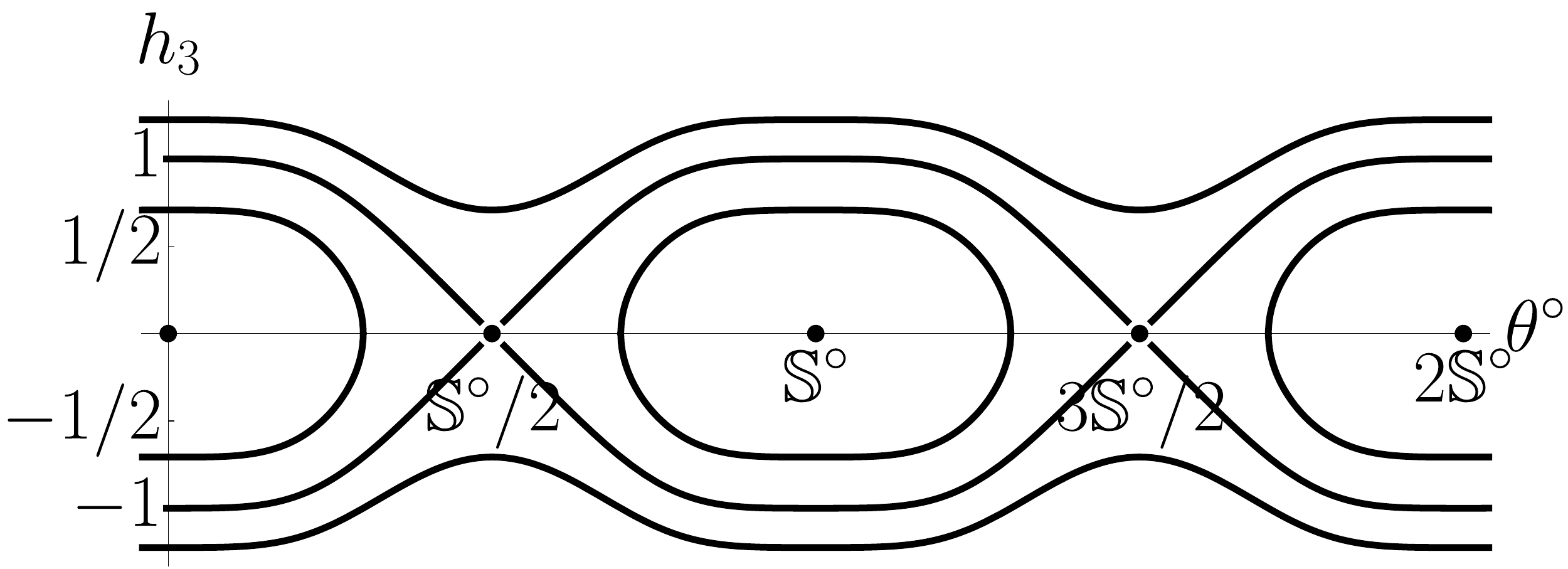}{Phase portrait of \eq{Ham3}: $\ell_p$, $q>2$, Case 2}{fig:phase2lpb}{0.35}{0.47}

\medskip \textbf{Case 3: $a=1$, $b=1$.}
The function $\Un(\tho) = (\soo^2 \tho-  \coo^2 \tho)/2$ 
has period $\mathbb S^{\polar}$ and is even with respect to $\tho = \mathbb S^{\polar}/2$.  
It has minima at $\tho = 0$ and $ \mathbb S^{\polar}$,  and maximum at $\tho =  \mathbb S^{\polar}/2$, see Figs.~\ref{fig:U3lpm},  \ref{fig:U3lpb}.
Thus system~\eq{Ham3} has center  equilibria at $(\tho, h_3) = (0, 0)$ and $(\tho, h_3) = (\mathbb S^{\polar}, 0)$, and  saddle equilibrium at $(\tho, h_3) = (\mathbb S^{\polar}/2, 0)$ {(which are singular extremals)}, see Figs.~\ref{fig:phase3lpm}, \ref{fig:phase3lpb}. 

In the case $1< q < 2${,} the function $\Un(\tho)$ is $C^1 \setminus C^2$ at $\tho = \So/2$.
The asymptotics~\eq{ascos} and
$$
\sin_{\Omega^{\polar}} (\tho) = 1 -  \frac{1}{q}(\So/2-\tho)^q + O(\So/2-\tho)^{2q}, \qquad
\tho \to \So/2,
$$
imply that the integral $I$~\eq{I} converges, and there is Cauchy nonuniqueness of extremals due to nonsmoothness. {Hence there are mixed extremals.}

\twofiglabels
{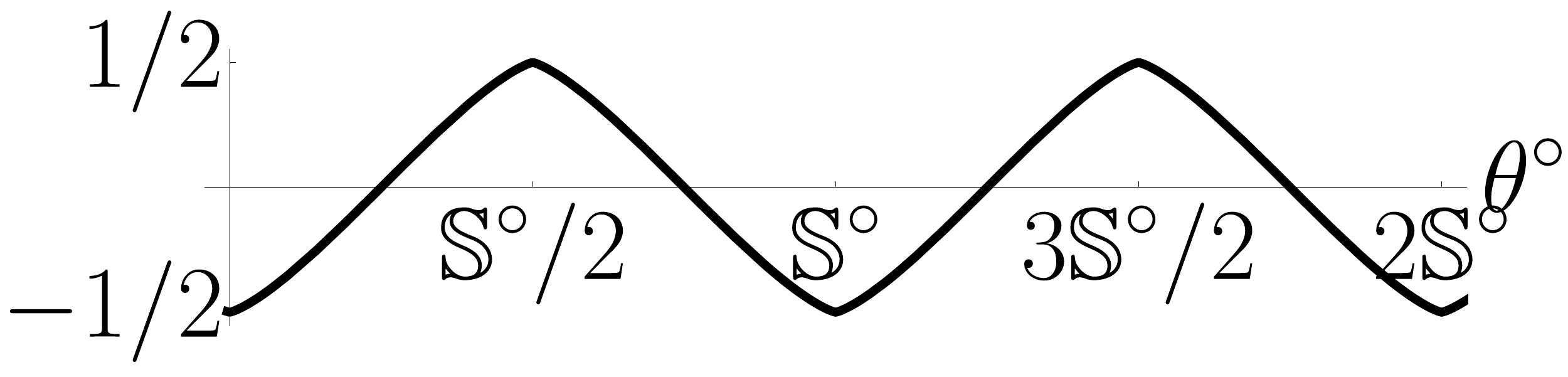}{{\small Plot of $\Un(\tho)$: $\ell_p$, $q<2$, Case 3}}{fig:U3lpm}
{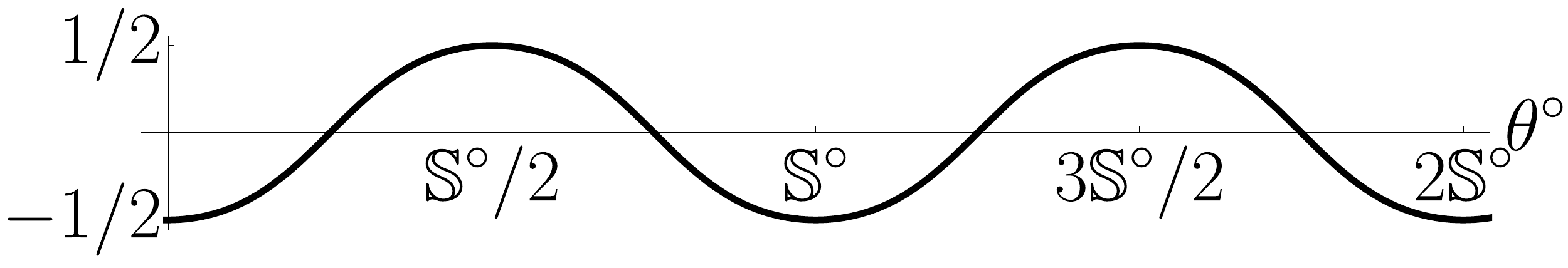}{{\small Plot of $\Un(\tho)$: $\ell_p$, $q>2$, Case 3}}{fig:U3lpb}{0.35}{0.47}

\twofiglabels
{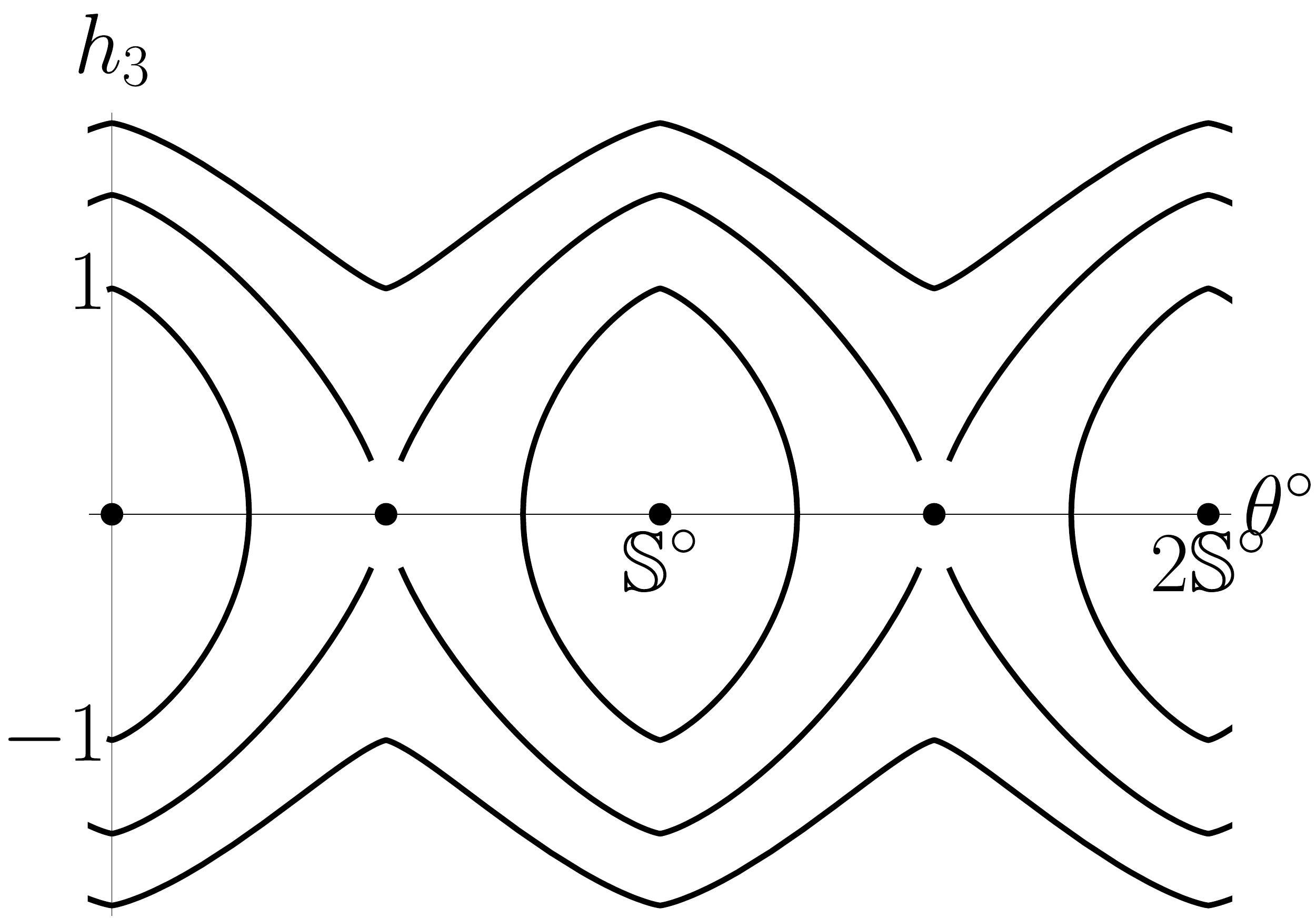}{Phase portrait of \eq{Ham3}: $\ell_p$, $q<2$, Case 3}{fig:phase3lpm}
{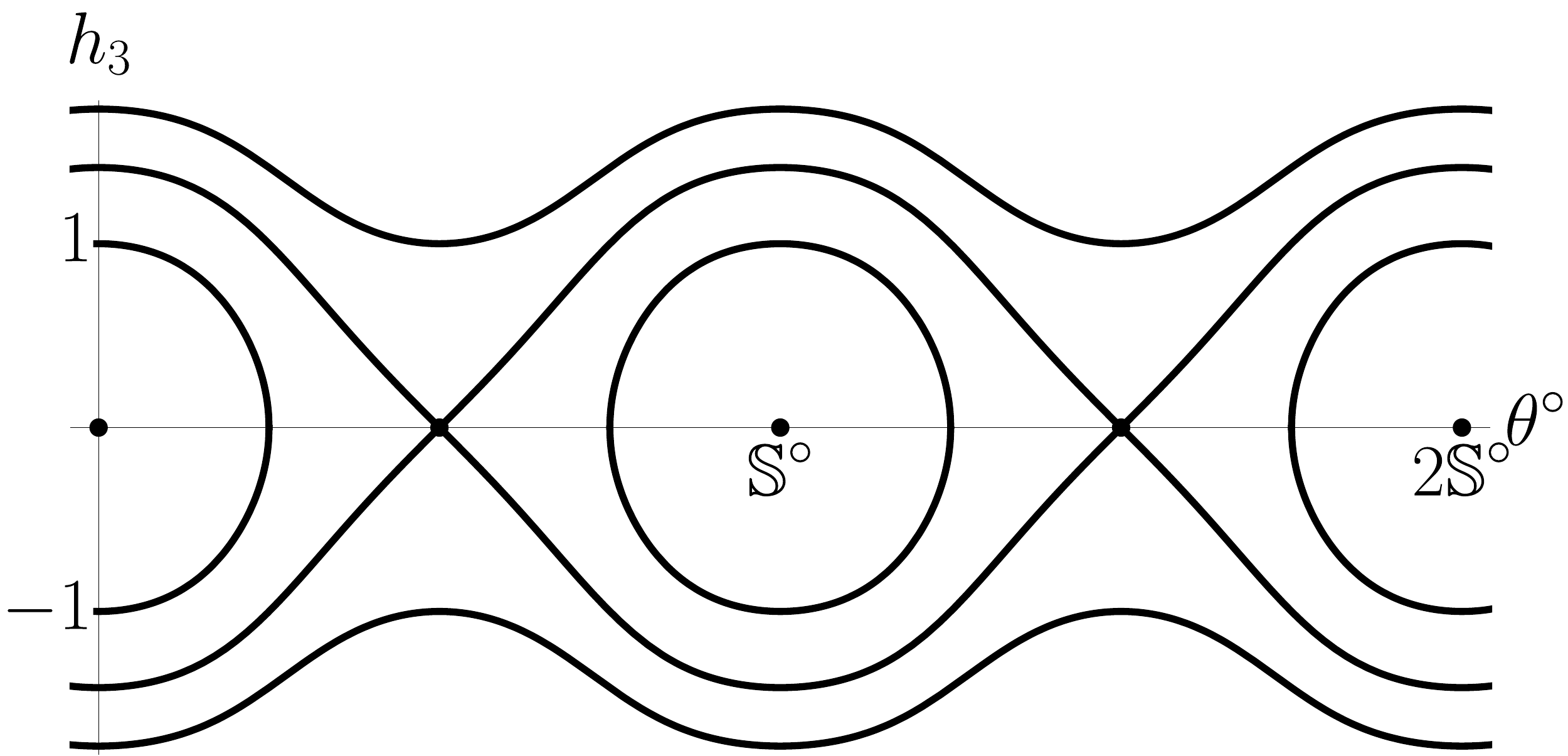}{Phase portrait of \eq{Ham3}: $\ell_p$, $q>2$, Case 3}{fig:phase3lpb}{0.35}{0.47}

\medskip \textbf{Case 4: $a=1$, $b=0$.}
The function $\Un(\tho) = \soo^2 \tho/2$ 
has period $\mathbb S^{\polar}$ and is even with respect to $\tho = \mathbb S^{\polar}/2$.  
It has minima at $\tho = 0$ and $ \mathbb S^{\polar}$,  and maximum at $\tho =  \mathbb S^{\polar}/2$, see Figs.~\ref{fig:U4lpm}, \ref{fig:U4lpb} .
Thus system~\eq{Ham3} has center  equilibria at $(\tho, h_3) = (0, 0)$ and $(\tho, h_3) = (\mathbb S^{\polar}, 0)$, and  saddle equilibrium at $(\tho, h_3) = (\mathbb S^{\polar}/2, 0)$, see Figs.~\ref{fig:phase4lpm}, \ref{fig:phase4lpb}.

Similarly to Case 3, {if $1 < q < 2$, there are singular extremals, Cauchy nonuniqueness, and hence mixed extremals.}

\twofiglabels
{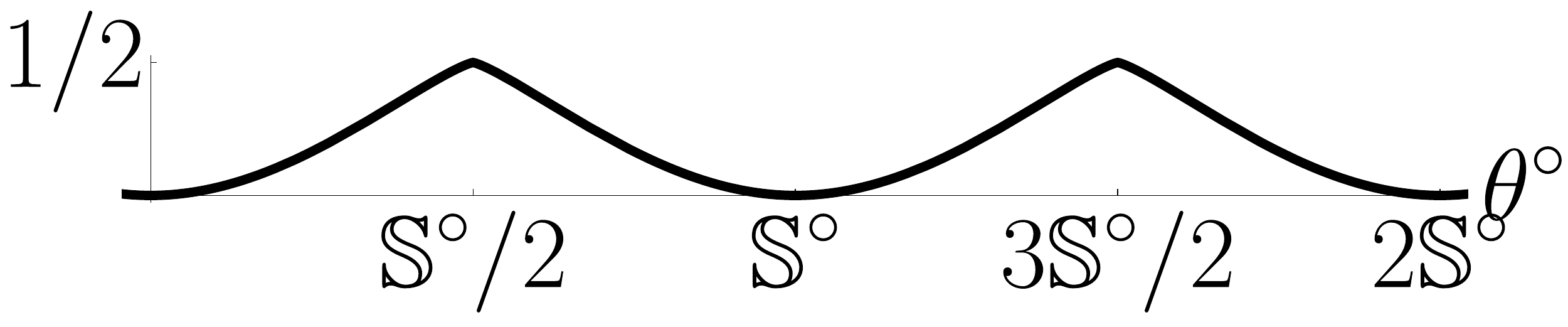}{{\small Plot of $\Un(\tho)$: $\ell_p$, $q<2$, Case 4}}{fig:U4lpm}
{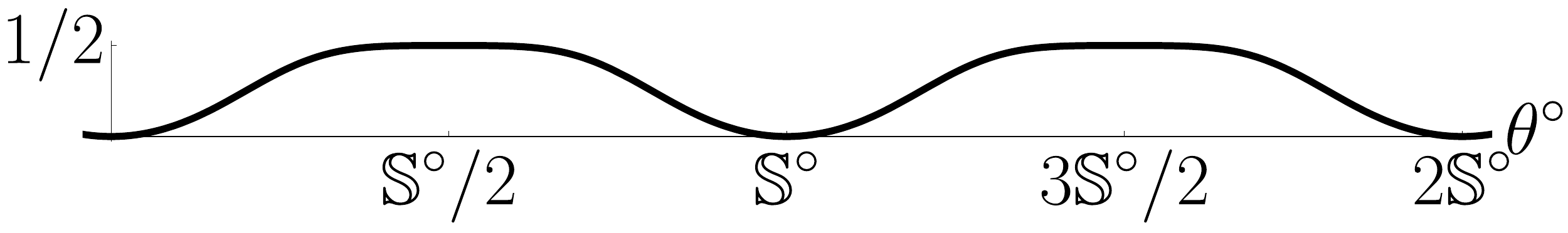}{{\small Plot of $\Un(\tho)$: $\ell_p$, $q>2$, Case 4}}{fig:U4lpb}{0.35}{0.47}

\twofiglabels
{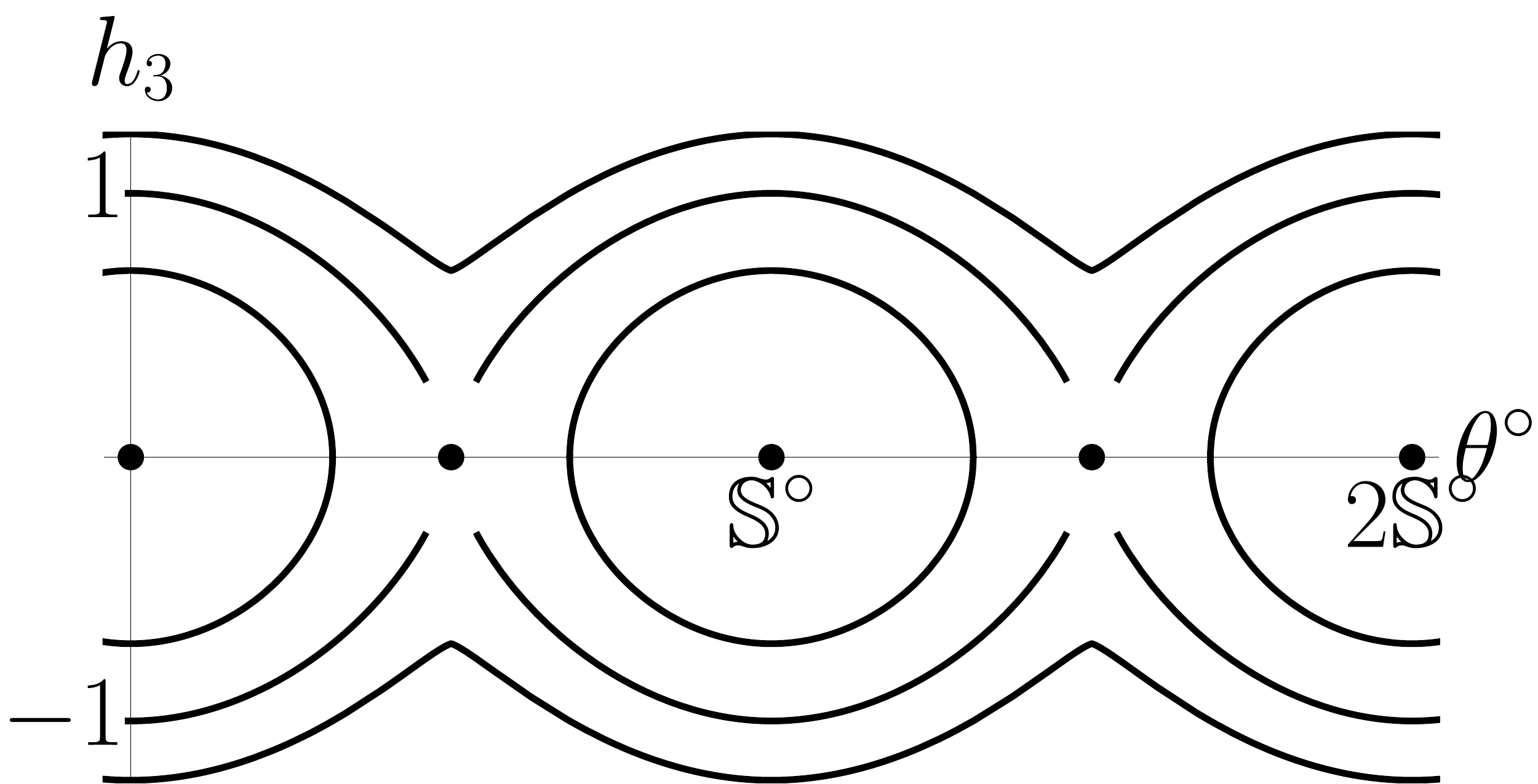}{Phase portrait of \eq{Ham3}: $\ell_p$, $q<2$, Case 4}{fig:phase4lpm}
{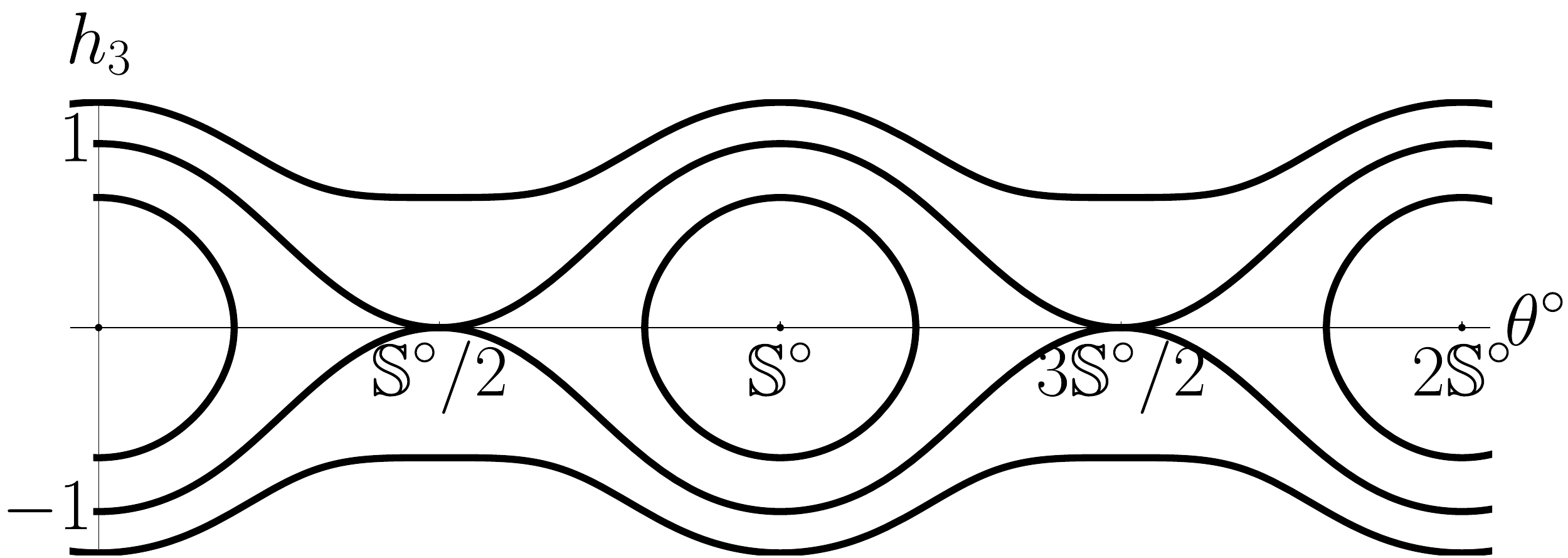}{Phase portrait of \eq{Ham3}: $\ell_p$, $q>2$, Case 4}{fig:phase4lpb}{0.35}{0.47}

\medskip \textbf{Case 5: $a=1$, $b=-1$.}
The function $\Un(\tho) = (\soo^2 \tho + \coo^2 \tho)/2$ has period $\mathbb S^{\polar}/2$ and is even with respect to $\tho = \mathbb S^{\polar}/4$. 

Let $1 < q < 2$, then $\Un(\tho)$ has maxima at $\tho = 0$ and $ \mathbb S^{\polar}/2$ {(which are singular extremals)},  and minimum at $\tho =  \mathbb S^{\polar}/4$, see~\ref{fig:U5lpm}. 
Thus system~\eq{Ham3} has saddle equilibria at $(\tho, h_3) = (0, 0)$ and $(\tho, h_3) = (\mathbb S^{\polar}/2, 0)$, and center equilibrium at $(\tho, h_3) = (\mathbb S^{\polar}/4, 0)$, see Fig.~\ref{fig:phase5lpm}.

Similarly to Case 3,  {if $1 < q < 2$, there are Cauchy nonuniqueness, and hence mixed extremals.}

\twofiglabels
{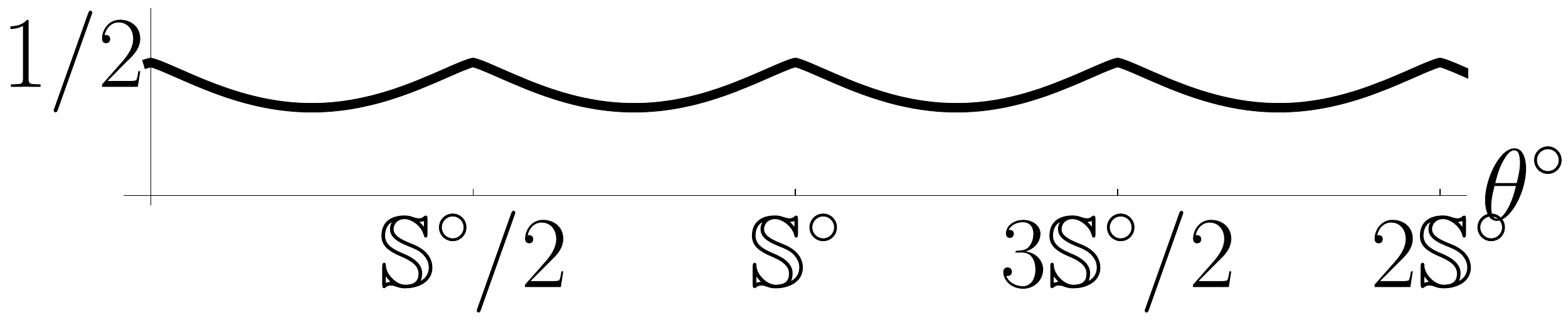}{{\small Plot of $\Un(\tho)$: $\ell_p$, $q<2$, Case 5}}{fig:U5lpm}
{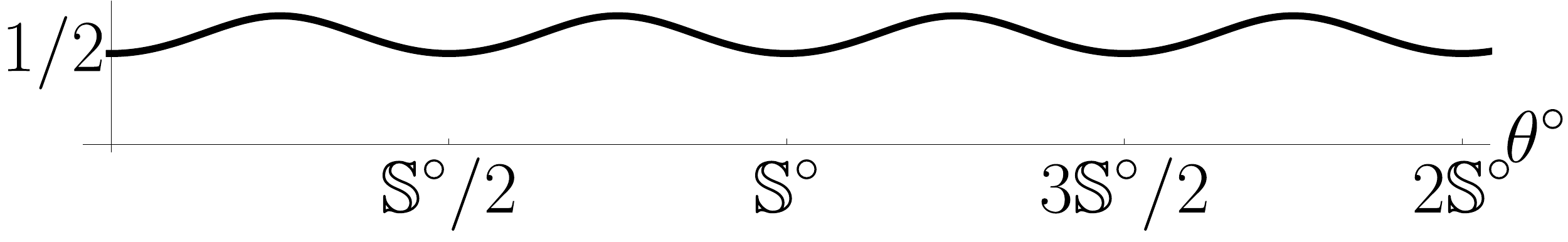}{{\small Plot of $\Un(\tho)$: $\ell_p$, $q>2$, Case 5}}{fig:U5lpb}{0.35}{0.47}

Let $2 < q < \infty$, then $\Un(\tho)$ has minima at $\tho = 0$ and $ \mathbb S^{\polar}/2$,  and maximum at $\tho =  \mathbb S^{\polar}/4$, see~\ref{fig:U5lpb}.
Thus system~\eq{Ham3} has center  equilibria at $(\tho, h_3) = (0, 0)$ and $(\tho, h_3) = (\mathbb S^{\polar}/2, 0)$, and  saddle equilibrium at $(\tho, h_3) = (\mathbb S^{\polar}/4, 0)$, see Fig.~\ref{fig:phase5lpb}.  

\twofiglabels
{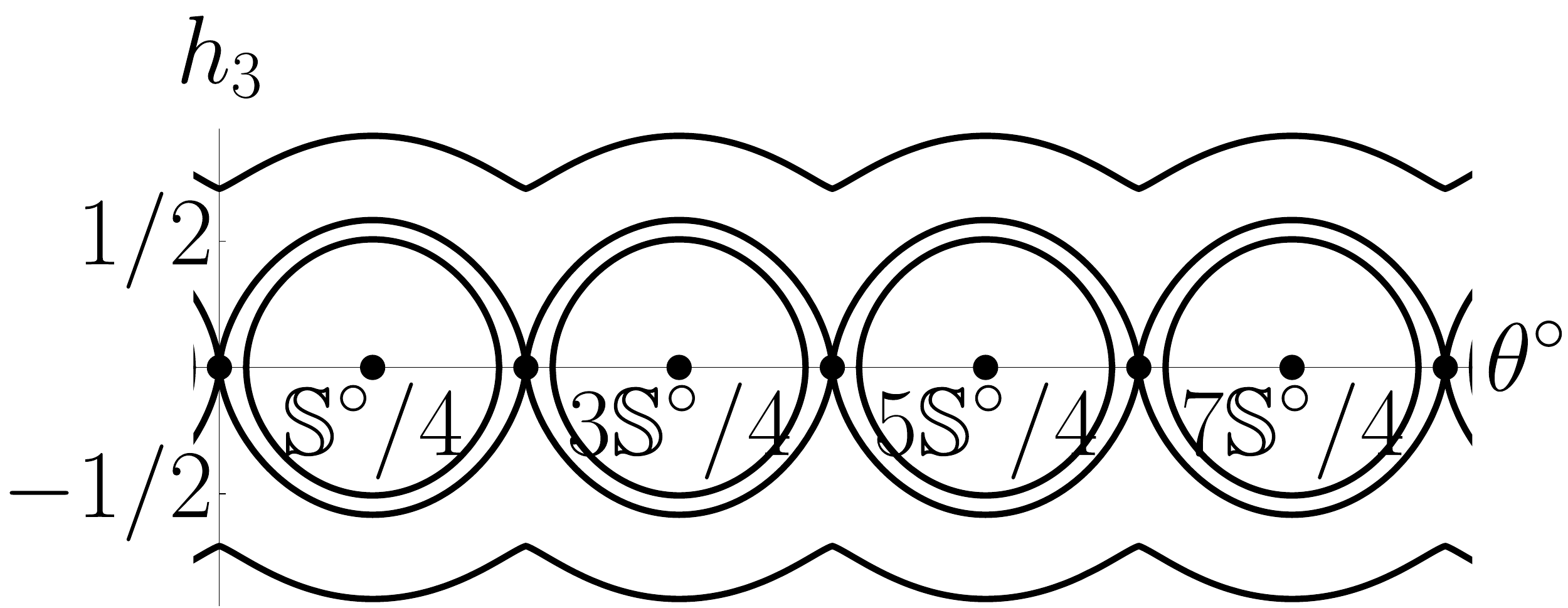}{Phase portrait of \eq{Ham3}: $\ell_p$, $q<2$, Case 5}{fig:phase5lpm}
{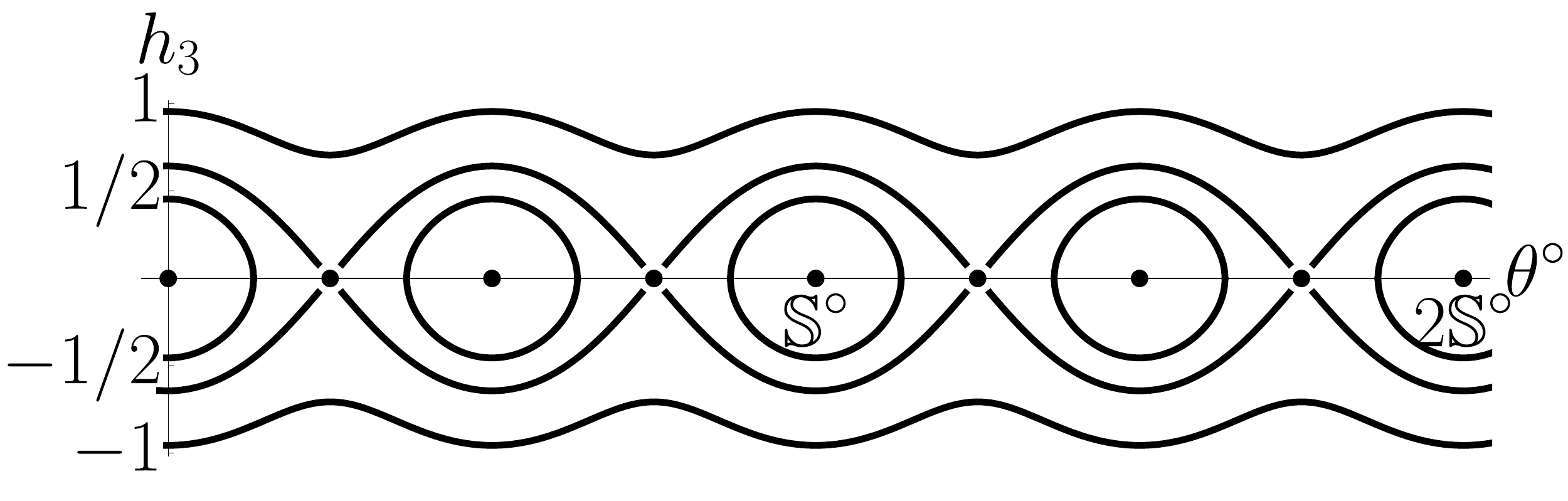}{Phase portrait of \eq{Ham3}: $\ell_p$, $q>2$, Case 5}{fig:phase5lpb}{0.35}{0.47}

Summing up, we proved the following statement.

\begin{thm}\label{th:lp}
Let $\Omega = \{ u \in \R^2 \mid \|u\|_p = 1\}$, $p \in (1, + \infty)$. {If $1<p<2$,} then all extremals in problem~\eq{pr21}--\eq{pr23} are bang. {If $p>2$, there are singular extremals.} There  {are} Cauchy nonuniqueness  {and mixed extremals} iff $(\, 2 < p < \infty$ and $a=1\,)$.

\end{thm}

\section{Rolling of a ball on a plane}
\label{sec:rolling_ball}

Another classic problem is about optimal control of a 3D ball, which rolls on a horizontal plane without slipping or twisting. Let us introduce a convenient coordinate system to describe this motion. Denote by $(x,y)\in\R^2$ coordinates of the contact point on the plane. We will use quaternions
\[
	z = z_0 + z_1\ii + z_2\jj + z_3\kk\in\mathbb{H} \cong \R^4,\quad |z|=1,
\]
\noindent to describe rotation of the ball with respect to its initial orientation. The ball is controlled by velocity of its upper point. So we have the following rotation equation:
\[
	\dot z = \frac12 z (\dot y\ii-\dot x\jj).
\]

We assume that the control $(u_1,u_2)=(\dot x,\dot y)$ belongs to a compact convex set $\Omega\subset\R^2$ and $0\in\mathrm{int}\,\Omega$ as usual. So we have the following time-minimizing problem:
\begin{equation}
\label{problem:ball}
		T\to\min,\qquad\dot x=u_1,\quad \dot y=u_2,\quad \dot z = \frac12z(u_2\ii - u_1\jj),\qquad u=(u_1,u_2)\in\Omega
\end{equation}
\noindent with some given end points (for example $x(0)=y(0)=0$, $z(0)=1$ and $x(T)=x_f$, $y(T)=y_f$, $z(T)=z_f$). Thus we are searching for a motion such that the ball moves from the initial position $(0,0)$ to the final one $(x_f,y_f)$, rotates by $z_f$, the velocity of contact point on the plane belongs to $-\Omega$, and this motion takes minimal possible time. The contact point during this motion draws on the plane a curve~$\gamma$. The motion time is equal to length of $\gamma$ in this Finsler quasimetric on the plane with $-\Omega$ as a unit ball.
So our time-minimizing problem can be reformulated as follows. Let us call a path connecting $(0,0)$ and $(x_f,y_f)$ admissible if the contact point moving along the path rotates the sphere by $z_f$. So the problem is to find an admissible path having minimal length in the Finsler quasimetric on the plane.
The case when $\Omega$ is the unit Euclidean disc (i.e.,\ we have the Euclidean metric on the plane) was integrated for the first time in~\cite{Jurdjevic}. 

Let us write down the Pontryagin maximum principle. From the Hamiltonian point of view the situation is very simple. We have a left-invariant Hamiltonian system on group $\mathfrak{G}=\R^2\times SO(3)$, $\dim\mathfrak{G}=5$. The vertical subsystem is a Hamiltonian system on the Lie coalgebra $\mathfrak{g}^*=\R^{2*}\times so(3)^*$ under the corresponding Lie-Poisson bracket. Since the group $\R^2$ is commutative, the rank of the Lie-Poisson bracket on $\mathfrak{g}^*$ at a general point is~$2$. So we have 3 independent Casimirs on $\mathfrak{g}^*$, general symplectic leaves are 2-dimensional, and any Hamiltonian system on $\mathfrak{g}^*$ is integrable in Liouville sense. On the whole cotangent bundle $T^*\mathfrak{G}$ we have non-degenerate left- and right-invariant symplectic structure. Any left-invariant (smooth) Hamiltonian system is again integrable in Liouville sense, since it always has 5 independent first integrals: the left-invariant Hamiltonian $H$ itself, left (which coincide with right) translations of 3 Casimirs, and right-invariant Hamiltonian $H_R$ that is right translation of~$H|_{\mathfrak{g}^*}$ (remind that Hamiltonian coming from PMP is usually nonsmooth).

Despite the written above ideas, we need to write formulas of explicit integration that are simple and convenient to work with. This is also very important, since the Hamiltonian of Pontryagin maximum principle is not smooth, and there can appear singular extremals and Cauchy non-uniqueness phenomenon.

Denote by $(p,q)\in\R^2$ the conjugate variables to $(x,y)$ and by $r\in\mathbb{H}$ the conjugate variable to $z$. The dot product of $z$ and $r$ is $\Re (z\bar r)$, so
\[
	\mathcal{H} = pu_1 + qu_2 - \frac12\Re(z\jj\bar r)u_1 + \frac12\Re(z\ii\bar r)u_2 =
	\big(p-\frac12\Re(z\jj \bar r)\big)u_1 + \big(q + \frac12\Re(z\ii \bar r)\big)u_2.
\]

The procedure of integration stays the same as in the previous sections. To find formulas for solutions we write maximum of $\mathcal{H}$ in $u\in\Omega$ in terms of the support function $s_\Omega$ of the set~$\Omega$ as usual:
\[
	H=\max_{u\in\Omega}\mathcal{H} = 
	s_\Omega(p-\frac12\Re(z\jj\bar r),q + \frac12\Re(z\ii\bar r)).
\]
\noindent  We denote the arguments of $s_\Omega$ by $h_1=p-\frac12\Re(z\jj\bar r)$ and $h_2=q + \frac12\Re(z\ii\bar r)$. Direct substitution of $\dot p = -\mathcal{H}_x = 0$, $\dot q = -\mathcal{H}_y=0$ and $\dot r = -\mathcal{H}_z = \frac12r(\ii u_2-\jj u_1)$ gives
\[
	\dot h_1 = -h_3 u_2,\quad \dot h_2 = h_3 u_1
\]
\noindent where $h_3 = -\frac12\Re (z\kk\bar r)$. Second differentiation gives
\[
	\dot h_3 = -\frac12\Re(\ii z\bar r) u_1 - \frac12\Re(\jj z\bar r)u_2.
\]

The right-hand side here is equal to $(q-h_2)u_1 + (h_1-p)u_2$. Consequently, equations on $h_1$, $h_2$ and $h_3$ do not depend on the other variables and form (together with $\dot p=\dot q=0$) the vertical subsystem:
\begin{equation}
\label{eq:roling_ball_vertical_subsystem}
	\dot h_1 = -h_3u_2,\quad
	\dot h_2 = h_3 u_1,\quad
	\dot h_3 = (q-h_2)u_1 + (h_1-p)u_2.
\end{equation}

The case $H=0$ leads to abnormal extremals: if $h_1=h_2\equiv0$, then $h_3\equiv0$, $(u_1,u_2)\parallel(p,q)$, and $(u_1,u_2)\in\partial\Omega$, which are straight lines in the plane $(x,y)$ independently on $\Omega$. 

Let us integrate system~\eqref{eq:roling_ball_vertical_subsystem} and find formulas for solutions for the case $H>0$ by using convex trigonometry. Again, since $H=s_{\Omega}(h_1,h_2)=\const$, the point $(h_1,h_2)$ moves along the boundary of the polar set $\Omega^\polar$ stretched by $H$ times:
\[
(h_1,h_2)\in H\partial\Omega^\polar.
\]
\noindent Put
\[
h_1 = H\cos_{\Omega^\polar}\theta^\polar\qquad\mbox{and}\qquad h_2=H\sin_{\Omega^\polar}\theta^\polar.
\]
\noindent If $(u_1,u_2)$ is an optimal control, then $h_1u_1+h_2u_2=H$ and $u\in\partial\Omega$. So again by the generalized Pythagorean identity (see Subsec.~\ref{subsec:CT_properties}) we get
\[
u_1=\cos_\Omega\theta\qquad\mbox{and}\qquad u_2=\sin_\Omega\theta
\]
\noindent for an angle $\theta\leftrightarrow\theta^\polar$. According to the inverse polar change of coordinates (see Subsec.~\ref{subsec:CT_properties}), for a.e.\ $t$, we have
\begin{equation}
\label{ball:dot_theta_polar}
	\dot\theta^\polar = \frac{h_1\dot h_2 - h_2\dot h_1}{H^2} = 
	\frac{h_3(h_1u_1+h_2u_2)}{H^2} = \frac{h_3}{H}.
\end{equation}
\noindent Also from~\eqref{eq:roling_ball_vertical_subsystem} we know that
\begin{equation}
\label{ball:dot_h_3}
	\dot h_3 = (q-H\sin_{\Omega^\polar}\theta^\polar)\cos_\Omega\theta + 
	(H\cos_{\Omega^\polar}\theta^\polar-p)\sin_\Omega\theta.
\end{equation}

It is easy to find a first integral for the last two equations using Theorem~\ref{thm:general_energy_integral}. Let us show the direct way to do it (which leads to the same result). If we multiply the second equation by $\dot{\theta^\polar}$, then, using the differentiation formulae for $\cos_\Omega$ and $\sin_\Omega$ (see Subsec.~\ref{subsec:CT_properties}), we get
\[
	H\ddot{\theta^\polar}\dot{\theta^\polar} = \dot h_3 \dot{\theta^\polar}=  \frac{d}{dt}\Big[
		p\cos_{\Omega^\polar}\theta^\polar + q\sin_{\Omega^\polar}\theta^\polar -
		\frac12 H (\cos_{\Omega^\polar}^2\theta^\polar + \sin_{\Omega^\polar}^2\theta^\polar)
	\Big].
\]
\noindent Thus the first integral has the following form (recall that $p$, $q$, and $H$ are constants):
\begin{equation}
\label{eq:ball_energy}
	\mathbb{E} = \frac12H\big(
		\dot{\theta^\polar}^2 +
		\cos_{\Omega^\polar}^2\theta^\polar + \sin_{\Omega^\polar}^2\theta^\polar
	\big)
	- p\cos_{\Omega^\polar}\theta^\polar - q\sin_{\Omega^\polar}\theta^\polar = 
	\frac{1}{2H}h_3^2 + H\mathcal{U}(\theta^\polar).
\end{equation}
\noindent Let us again emphasize that the control $u=(u_1,u_2)$ can easily be found from $\theta^\polar$ by the relation $\theta\leftrightarrow\theta^\polar$.

If $\Omega$ is the unit disc, then equation~\eqref{eq:ball_energy} becomes the classical pendulum equation. 
In general case, equation~\eqref{eq:ball_energy} can be written in terms of kinetic and potential energy $\mathbb{E}/H=\frac12\dot{\theta^\polar}^2 + \mathcal{U}(\theta^\polar)$, where $\mathcal{U}$ coincides up to the constant $\frac{p^2+q^2}{H^2}$ with the squared distance from the point $(\frac{p}{H},\frac{q}{H})$ to the point $(\cos_{\Omega^\polar}\theta^\polar,\sin_{\Omega^\polar}\theta^\polar)$ on the boundary $\partial \Omega^\polar$. So the topological structure of the phase portrait of equation~\eqref{eq:ball_energy} is determined by minima and maxima points on $\partial\Omega^\polar$ of the distance to the point $(\frac{p}{H},\frac{q}{H})$. At any local minimum we have a fixed point of the center type, and at any local maximum we have a fixed point of the saddle type.

\begin{thm}
	If problem~\eqref{problem:ball} is sub-Riemannian (i.e.,\ the set $\Omega$ is an ellipse centered at the origin), then the vertical subsystem~\eqref{eq:roling_ball_vertical_subsystem} for the normal case $H>0$ is integrated in elliptic functions.
\end{thm}

\begin{proof}
	Let $a$ and $b$ be semi-axes of $\Omega$. The polar set $\Omega^\polar$ is an ellipse too: $\partial\Omega^\polar=\{(\frac1a\cos s,\frac1b\sin s)\}$. Using formulae obtained in Subsec.~\ref{subsec:explicit_sets}, we get $s=ab\,\theta^\polar$, $\cos_{\Omega^\polar}\theta^\polar=\frac1a\cos s$ and $\sin_{\Omega^\polar}\theta^\polar=\frac1b\sin s$. Hence, equation \eqref{eq:ball_energy} leads to the following ODE on the parameter~$s$:
	\[
		\dot s = ab\,\dot{\theta^\polar}=\pm\sqrt{\tilde{\mathbb{E}} - b^2(\cos s-\tilde p)^2 - a^2(\sin s-\tilde q)^2}.
	\]
	\noindent Here $\tilde{\mathbb{E}}=\frac{a^2b^2}{H}(2\mathbb{E}H+p^2+q^2)$, $\tilde p=\frac{ap}{H}$, and $\tilde q=\frac{bq}{H}$ are some constants. This equation can be integrated explicitly by elliptic functions in a standard way (the explicit form is of no importance here).
\end{proof}

Moreover, in the sub-Riemannian case, we immediately find $u_1=\cos_\Omega\theta=\frac{d}{d\theta^\polar}\sin_{\Omega^\polar}\theta^\polar = \frac{ds}{d\theta^\polar}\frac{d}{ds}\big(\frac1b\sin s\big) = a\cos s$ and $u_2=\sin_\Omega\theta=-\frac{d}{d\theta^\polar}\cos_{\Omega^\polar}\theta^\polar = -\frac{ds}{d\theta^\polar}\frac{d}{ds}\big(\frac1a\cos s\big) = b\sin s$.

If $\Omega^\polar$ has $W^2_\infty$ boundary (in this case $\Omega$ is necessarily strictly convex), then there is a solution uniqueness by Proposition~\ref{prop:uniqueness_of_solution_when_Omega_smooth}. Let us now consider the case when $\Omega$ is a polygon.

\begin{thm}
\label{thm:sphere_polygon}
	Suppose that $\Omega$ is a polygon. Then there are two types of normal extremals:
	\begin{enumerate}
		\item[$(1)$] If a control $\hat u$ moves arbitrarily along one fixed edge $e$ in $\Omega$, then the obtained extremal is optimal. In this case, the point $(\frac pH,\frac qH)$ coincides with the vertex of $\Omega^\polar$ corresponding to the edge $e$.
		
		\item[$(2)$] In another case, the control $\hat u$ is bang-bang (i.e.,\ piecewise constant). The set of its possible values is finite\footnote{Only values on a set of positive measure are considered.}: a control $\hat u=(\cos_\Omega\hat\theta,\sin_\Omega\hat \theta)$ is either a vertex of $\Omega$ or a point on an edge $e$ of $\Omega$, given by the condition
		\begin{equation}
		\label{ball_singular_control}
			(\cos_\Omega\hat\theta,\sin_\Omega\hat\theta) \parallel
			\left(\cos_{\Omega^\polar}\theta^\polar_0-\frac{p}{H},\sin_{\Omega^\polar}\theta^\polar_0-\frac{q}{H}\right)
		\end{equation}
		\noindent where $\theta^\polar_0$ determines the vertex of $\Omega^\polar$ corresponding to edge $e$. The last case is allowed by PMP if and only if $\mathbb{E}=H\,\mathcal{U}(\theta^\polar_0)$.
	\end{enumerate}
\end{thm}

\begin{proof}
	The first part is easy to prove. Indeed, if $\hat u(t)$ belongs to a fixed edge $e$ of $\Omega$ for a.e.\ $t$, then $(x(T),y(T))$ belongs to the boundary of the reachable set (in time $T$) for the system $(\dot x,\dot y)\in \Omega$. So there is no way to reach this point faster, and hence the trajectory is optimal.
	
	The second part follows from Propositions~\ref{prop:uniqueness_of_solution_when_Omega_smooth}, \ref{prop:uniqueness_if_zero_not_in_U_derivative} and Theorem~\ref{thm:absense_of_uniqueness_for_general_ct_ode}. Indeed, if an extremal has no singular parts, then the control is piecewise constant. Singular parts are described in Section~\ref{subsec:theta_in_general_ct_ode}. General singular parts are described in item~$\eqref{ct_ode_general_singular}$ in Section~\ref{subsec:theta_in_general_ct_ode}, where it is proved that, in this case, solution must stay at a corner, and singular control is uniquely determined by the condition $df\parallel (\cos_\Omega\theta,\sin_\Omega\theta)$ where $f=\frac{1}{2}(x-\frac{p}{H})^2 + \frac12(y-\frac{q}{H})^2$ for our problem. Special singular controls are described in item~$\eqref{ct_ode_special_singular}$ in Section~\ref{subsec:theta_in_general_ct_ode} and may appear only if $df=0$ at a corner of $\Omega^\polar$ (see Section~\ref{subsec:theta_in_general_ct_ode}). The function $f$ has null derivative $df(x,y)=0$ only at its global minimum $(x,y)=(\frac{p}{H},\frac{q}{H})$. Thus if $\dot\theta^\polar(0)=0$, the point  $Q=(\cos_{\Omega^\polar}\theta^\polar(0),\sin_{\Omega^\polar}\theta^\polar(0))$ is a corner of $\Omega^\polar$, and $Q=(\frac{p}{H},\frac{q}{H})$, then there exists a unique solution $\theta^\polar(t)\equiv\const$, which determines special singular extremals described in the first part.
\end{proof}

Applying Remark~\ref{remark:two_different_singular_controls} to the current setting, we obtain that a control $\hat u=(\cos_\Omega\hat\theta,\sin_\Omega\hat \theta)\in e$ satisfying~\eqref{ball_singular_control} is a general singular control (on the edge~$e$). If a singular control moves arbitrary along an edge $e$, then it is a special singular control (on the edge $e$).

\begin{remark}
	On any interval with constant control, the functions $x(t)$ and $y(t)$ are linear. Moreover, if $\Omega$ is a polygon, then it is easy to determine the length of constancy interval. If a control is singular, then the length is arbitrary. In the other case, the point $Q_{\theta^\polar}$ moves along an edge of $\Omega^\polar$. Since the function $\mathcal{U}$ is quadratic on any edge, the equation $\mathbb{E}/H=\frac12\dot{\theta^\polar}^2+\mathcal{U}(\theta^\polar)$ has the form
	\[
		\dot{\theta^\polar} = \pm\sqrt{a{\theta^\polar}^2 + b\theta^\polar + c} 
	\]
	\noindent where $a$, $b$, and $c$ are some constants. Moreover, $a=-(\cos_\Omega^2 \theta+\sin_\Omega^2 \theta)<0$ for $\theta\leftrightarrow\theta^\polar$, $\theta=\const$. Consequently, this equation can be easily solved in classical trigonometric functions. Hence, the sequence of control jumps can be easily derived from the phase portrait of the equation~\eqref{eq:ball_energy}, which can be constructed similarly to Sec.~\ref{sec:3DLie}.
\end{remark}

\section{Yachts}\label{sec:yachts}

\begin{defn}
Consider an optimal control problem whose differential system is defined by the classic trigonometric functions $\cos$, $\sin$.
An \textit{$\Omega$-modification} of such a problem is the same problem with
$\cos$ changed to $\cos_{\Omega}$ and 
$\sin$ to $\sin_{\Omega}$, where a convex compact set~$\Omega$ with $0\in\mathrm{int}\,\Omega$ defines the modification of the differential system. 
\end{defn}

Let us state the following control system:
\begin{align}
		\dot q &= u_1 X_1 + u_2 X_2, \qquad q = (x,y,\theta)\in \R_{x,y}^2 \times \big(\R_{\theta}/(2\mathbb{S}\Z)\big), \label{sys-se2}
\end{align}
with modified vector fields 
\begin{align}
		X_1 &= (\cos_\Omega \theta, \sin_\Omega \theta, 0), \qquad X_2 = (0,0,1),
\end{align}
we fix boundary conditions and integral cost functional in general form
\begin{align}
		q(0) = q_0, \qquad q(T) = q_1, \label{boundary-se2} \\
J = \int_0^T f(u_1,u_2) dt \to \min. \label{Jmin}
\end{align}
Also, we assume a possible restriction on the control 
\begin{align}
(u_1,u_2) \in U \subset \R^2. \label{restr}
\end{align}

\begin{remark}
One may consider non-factorized domain for the angle parameter $\theta \in \R_{\theta}$ which does not change extremal controls. However, in the~corresponding problem, optimality question can be investigated easier than in the original for some cases. Below, we will refer to such a problem as ``non-factorized problem'' while concerning optimality of obtained extremal controls.
\end{remark}

$\Omega$-modifications of the following four classic optimal control problems appear as specifications of problem~(\ref{sys-se2})--(\ref{restr}):
\begin{enumerate}
\item Euler's elastic problem~\cite{euler} with 
\begin{align}
U = U_{E} = \{(u_1, u_2)\in \R^2 \mid u_1 = 1\}, \qquad f = f_E = \frac{u_2^2}{2}. \label{UfE}
\end{align}
\item Markov-Dubins car~\cite{markov, dubins} with 
\begin{align}
U = U_{MD}=\{(u_1, u_2) \in \R^2 \mid u_1 = 1, |u_2| \leq 1\}. \qquad f = 1, \label{UfMD}
\end{align}
\item Reeds-Shepp car~\cite{reeds-shepp} with 
\begin{align}
U = U_{RS}=\{(u_1, u_2) \in \R^2 \mid |u_1| \leq 1, |u_2| \leq 1\}, \qquad f = 1. \label{UfRS}
\end{align}
Notice, that originally Reeds and Shepp used condition $|u_1| = 1$ instead of $|u_1| \leq 1$, but the corresponding problems are equivalent in the sense of 
convexity of the set of possible velocities~\cite{notes}.
\item sub-Riemannian problem on the group of Euclidean motions of $2$-dimensional plane~\cite{cut_sre2} with 
\begin{align}
U = \R^2, \qquad f = f_{SR} =  \frac{u_1^2 + u_2^2}{2} \label{UfSR1}
\end{align}
or equivalently with 
\begin{align}
U = U_{SR} = \{(u_1, u_2) \in \R^2 \mid u_1^2 + u_2^2 \leq 1\}, \qquad f = 1. \label{UfSR2}
\end{align}
\end{enumerate}

All four classic problems express control models for a car-like robot moving on a horizontal plane. $\Omega$-modifications of those problems can be understood as control models for a yacht in a sea (or car moving on a non-horizontal plane). For the first and the second models with $u_1=1$ a domain $\Omega$ gives a description of velocity vector at $(x,y)$ as $(\cos_\Omega \theta,\sin_\Omega \theta)$. Explicit form of $\Omega$ can be obtained from external disturbances such as wind or water stream (or angle of inclined plane), so $\Omega$ is a disc for no external disturbances. For the third and the forth models we should assume that $\Omega$ is symmetric, i.e., $\Omega = -\Omega$ and yacht can move forward and backward with the same speed. For non-symmetric case with $\Omega \neq -\Omega$ in the third and the forth problems, the physical interpretation fails, however, the mathematical problem follows the same solution presented below. 

\begin{remark}
It is possible to set $q_0 = (0,0,0)$ for all four classic problems. However, in our $\Omega$-modification case we assume $q_0 = (0,0, \theta_0)$ with $\theta_0 \in \R_{\theta}/(2\mathbb{S}\Z)$ (even for ``non-factorized problem'') to emphasize the importance of initial direction $\theta_0$ (one may fix it vanishing through rotating of the domain $\Omega$ by the corresponding angle $\theta_{\Omega}^{-1} (\theta_0)$, see Proposition~\ref{prop:angle-sum}). 
\end{remark}

Let us apply PMP to the general problem~(\ref{sys-se2})--(\ref{restr}) and then proceed with each of its specifications separately. We have the following Hamiltonian function of PMP:
\[
	\mathcal{H} = \psi_0 f(u_1,u_2) + (\psi_1 \cos_{\Omega} \theta + \psi_2 \sin_{\Omega} \theta) u_1 + \psi_3 u_2,
\]
where $\psi = (\psi_0, \psi_1, \psi_2, \psi_3)$ is a vector of adjoint variables, here $\psi_0 \leq 0$ is a constant.

The vertical subsystem of the Hamiltonian system takes the form:
\begin{equation}\label{vert-yacht}
\begin{cases}
\dot{\psi}_1 = \dot{\psi}_2 = 0,\\
\dot{\psi}_3 = (\psi_1 \sin_{\Omega^\polar} \theta^\polar - \psi_2 \cos_{\Omega^\polar} \theta^\polar) u_1.
\end{cases}
\end{equation}  

The maximality condition depends on the domain $U$ and the function $f(u_1,u_2)$ which defines cost functional~(\ref{Jmin}).  
The case $\psi_1^2 + \psi_2^2 = 0$ is considered for each problem separately further in the subsections.
Suppose $\psi_1^2 + \psi_2^2 \neq 0$. We assume $\psi_1^2 + \psi_2^2 = 1$ without loss of generality. 

\begin{prop}[see~\cite{CT1}]
\label{prop:angle-sum}
    Let $e^{\ii \alpha}$ denote the clockwise rotation of $\R^2$ by the angle $\alpha \in S^1$ around the origin. Then
    \begin{align*}
        \cos_\Omega\theta \cos\alpha  +\sin_\Omega\theta\sin\alpha  &= \cos_{e^{\ii\alpha }\Omega}(\theta-\theta_\Omega(\alpha)),\\
        \sin_\Omega\theta\cos\alpha -\cos_\Omega\theta \sin\alpha &= \sin_{e^{\ii\alpha }\Omega}(\theta-\theta_\Omega(\alpha)),
    \end{align*}
where the function $\theta_{\Omega}: S^1 \to \R/(2 \mathbb{S}\Z)$ sets the correspondence between the classic and generalized angles.
\end{prop}

Define an angle $\alpha = \const \in S^1$ satisfying $\psi_1 = \cos \alpha, \psi_2 = \sin \alpha$. 
Using Proposition~\ref{prop:angle-sum} we rewrite vertical subsystem~(\ref{vert-yacht}) as follows:
\begin{equation}\label{vert-yacht2}
\dot{\psi}_3 = (\cos \alpha \sin_{\Omega^\polar} \theta^\polar - \sin \alpha \cos_{\Omega^\polar} \theta^\polar) u_1 = \sin_{e^{i \alpha} \Omega^\polar} (\theta^\polar - \theta_{\Omega^\polar}(\alpha)) u_1. 
\end{equation}
With $\tilde{\alpha}^\polar = \theta_{\Omega^\polar}(\alpha), \tilde{\Omega}^\polar = e^{i \alpha} \Omega^\polar, \tilde{\theta}^\polar = \theta^\polar - \tilde{\alpha}^\polar$ we have
\begin{equation}\label{vert-md3}
\dot{\psi}_3 = u_1 \sin_{\tilde{\Omega}^\polar} \tilde{\theta}^\polar. 
\end{equation}

Moreover, with $\tilde{\alpha} = \theta_{\Omega}(\alpha), \tilde{\Omega} = e^{i \alpha} \Omega, \tilde{\theta} = \theta - \tilde{\alpha}$, the Hamiltonian function of PMP takes the following form:
\begin{align}
	\mathcal{H} = \psi_0 f(u_1,u_2) + u_1 \cos_{\tilde{\Omega}} \tilde{\theta} + u_2 \psi_3.
\end{align}
For the time minimization problems with $f = 1$ we will use the short form of the Hamiltonian function 
\begin{align}
\mathcal{H} = u_1 \cos_{\tilde{\Omega}} \tilde{\theta} + u_2 \psi_3. \label{H-short}
\end{align}

Now we proceed with each specification for $U$ and $f$ separately.

\subsection{Euler's elastic problem}\label{ela-subsec}
We start with an $\Omega$-modification of one of the oldest optimal control problem defined by~(\ref{sys-se2})--(\ref{restr}) with~(\ref{UfE}). We have $u_1 \equiv 1$ and we should find $u_2$.



Consider the abnormal case $\psi_0 = 0$. From the maximality condition for $\mathcal{H}$ we have $\psi_3 \equiv 0$, therefore $\psi_1^2 + \psi_2^2 \neq 0$ and $\dot{\psi}_3 \equiv 0$ yields $\sin_{\tilde{\Omega}^\polar} \tilde{\theta}^\polar \equiv 0$ (see~(\ref{vert-md3})). Let us interpret this condition geometrically.
Since $\tilde{\Omega}$ is compact, there are the minimal value and the maximal value of $\cos_{\tilde{\Omega}} \tilde{\theta}$ denoted by $m_1<0$ and $m_2>0$. Notice that $\sin_{\tilde{\Omega}^\polar} \tilde{\theta}^\polar \equiv 0$ iff $\cos_{\tilde{\Omega}} \tilde{\theta} \equiv m_1$ or $\cos_{\tilde{\Omega}} \tilde{\theta}\equiv m_2$.
Therefore, 
if 
$\partial \Omega$  contains no edges, 
then $u_2 \equiv 0, \forall \alpha \in S^1$ with $x(t) = t \cos_{\Omega} \theta_0, \ y(t) = t \sin_{\Omega} \theta_0, \ \theta(t) = \theta_0$. Otherwise, for every edge of $\Omega$ with $\theta(t) \in [\theta_-, \theta_+]$, we have arbitrary $u_2 (t)$ for $\theta(t) \in (\theta_-, \theta_+)$ and $\mp u_2 (t) \geq 0$ when corresponding $\theta(t) = \theta_{\pm}$.

Normal case: $\psi_0 = -1$. From the maximality condition for $\mathcal{H} = - \frac{u_2^2}{2} + \cos_{\tilde{\Omega}} \tilde{\theta} + u_2 \psi_3$ we have $\dot{\tilde{\theta}} = u_2 = \psi_3$. 

Suppose $\psi_1 = \psi_2 = 0$. Then $\dot{\psi}_3\equiv0$, so $u_2 = \psi_3 = c \equiv \const$, i.e., $\theta = c t + \theta_0$. Explicit formulas for $x,y$ can be obtained via integration 
\begin{align}
x(t) = \frac{1}{c} \int_{\theta_0}^{c t + \theta_0} \cos_{\Omega} \theta d \theta, \qquad y(t) = \frac{1}{c} \int_{\theta_0}^{c t + \theta_0} \sin_{\Omega} \theta d \theta. \label{xy-dubins}
\end{align}

\begin{figure}[ht]
	\centering
	\includegraphics[width=0.195\textwidth]{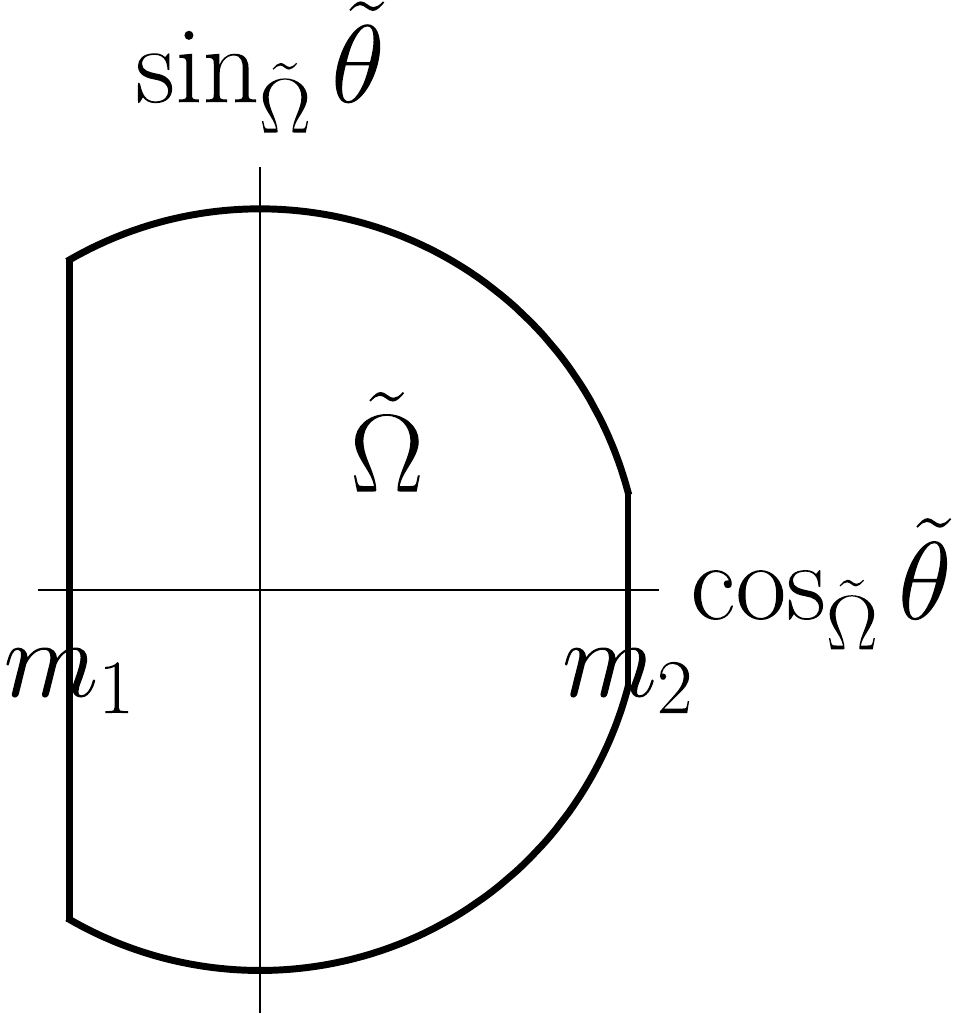} \qquad \includegraphics[width=0.32\textwidth]{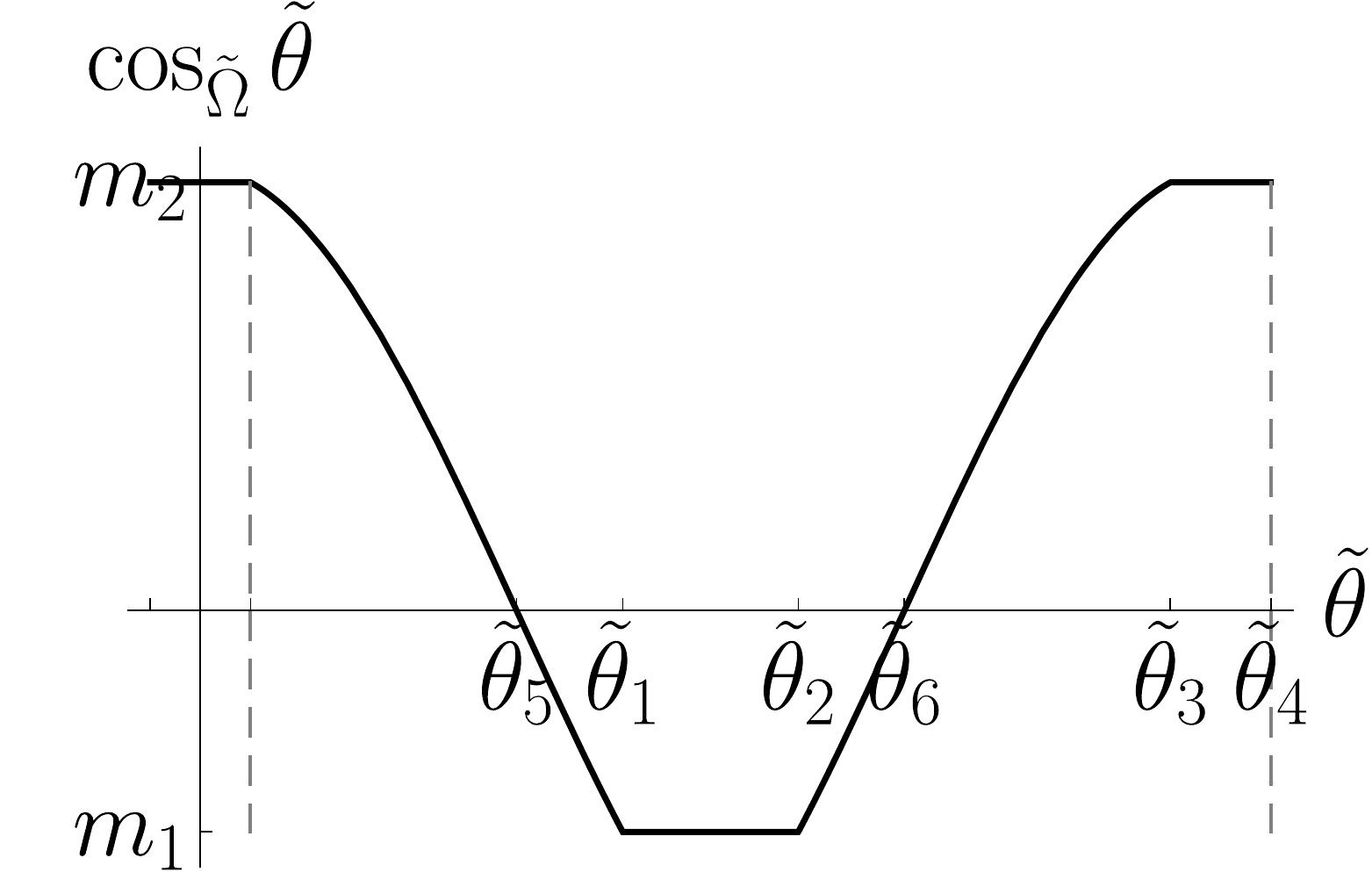} \qquad \includegraphics[width=0.32\textwidth]{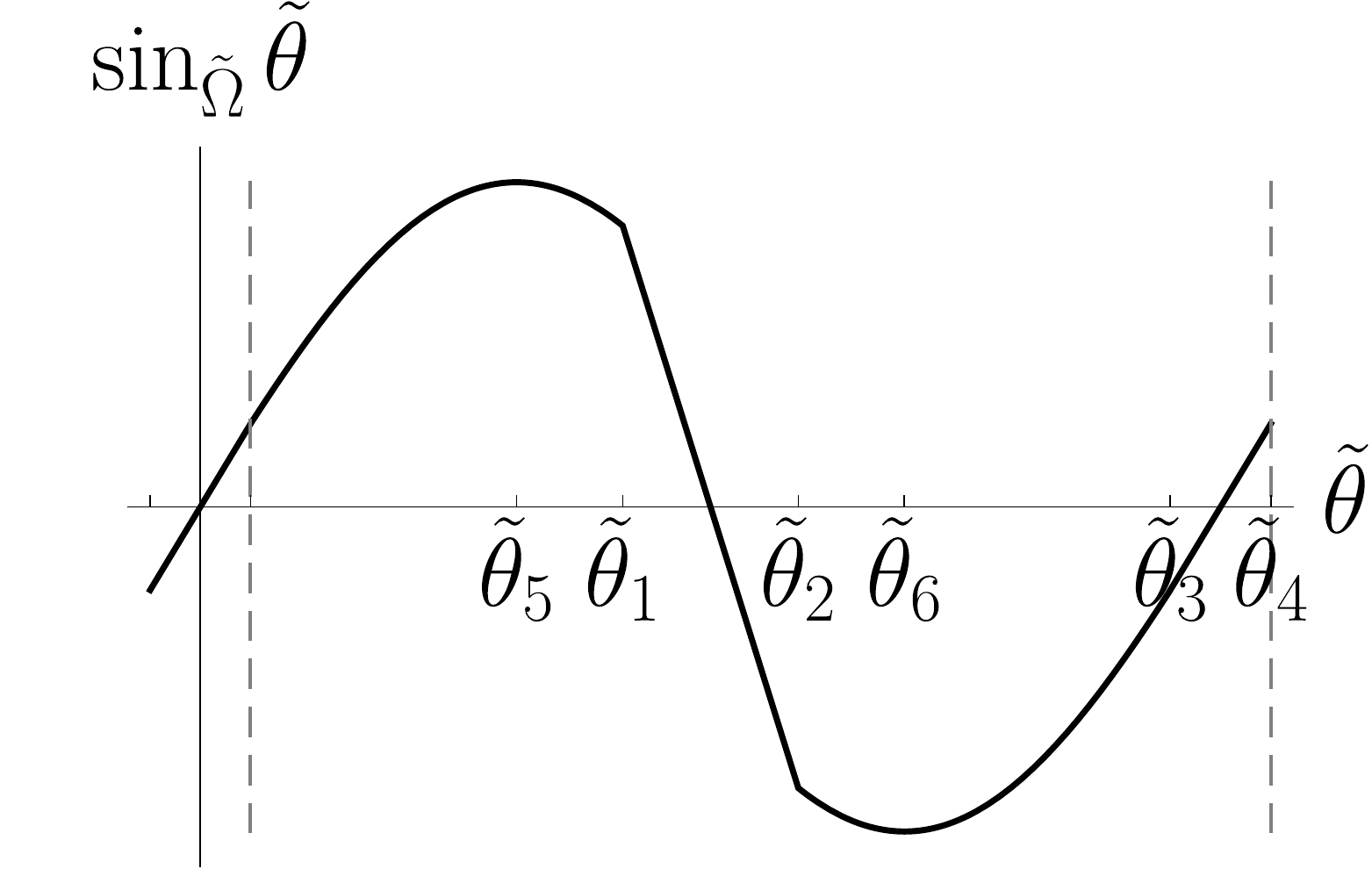} 
\caption{An example of $\tilde{\Omega}$ with the corresponding functions $\cos_{\tilde{\Omega}}, \sin_{\tilde\Omega}$ }
	\label{fig:DomainO}
\end{figure}

If $\psi_1^2 + \psi_2^2 \neq 0$, then we get~(\ref{vert-md3}) with the maximized Hamiltonian $H = \frac{1}{2}\dot{\tilde{\theta}}^2 + \cos_{\tilde{\Omega}} \tilde{\theta}$, which determines the phase portrait of the system (see Fig.~\ref{fig:phase_portrait_ela}).
Since $\tilde{\Omega}$ is compact, there are the minimal value and the maximal value of $\cos_{\tilde{\Omega}} \tilde{\theta}$ denoted by $m_1<0$ and $m_2>0$ (see Fig.~\ref{fig:DomainO}).

The~following cases are possible:
\begin{enumerate}
\item $H < m_1$ --- no solutions.
\item $H = m_1 \Rightarrow u_2 = \dot{\theta} \equiv 0, \cos_{\tilde{\Omega}} \tilde{\theta} \equiv  m_1$. The corresponding points are $\tilde{\theta} \in \cos_{\tilde{\Omega}}^{-1} m_1 \equiv [\tilde{\theta}_1, \tilde{\theta}_2] (\mod 2 \mathbb{S})$, see Fig.~\ref{fig:DomainO} at the center. Such points with condition $\dot{\theta} = u_2 = \psi_3 = 0$ set stable equilibria (see Fig.~\ref{fig:phase_portrait_ela}) and on $(x,y)$ we have straight lines with $\theta(t) \equiv \theta_0 \in [\tilde{\theta}_1 +\tilde{\alpha}, \tilde{\theta}_2+\tilde{\alpha}]$.
\item $H \in (m_1, m_2) \Rightarrow \dot{\theta} \not\equiv 0, \cos_{\tilde{\Omega}}\tilde\theta\leq H$. There exist two points $\tilde\theta_H^-\not\equiv \tilde\theta_H^+ (\mod 2\mathbb{S})$, s.t. $\cos_{\tilde{\Omega}}\tilde\theta_H^- = \cos_{\tilde{\Omega}}\tilde\theta_H^+ =  H$ and $\cos_{\tilde{\Omega}}\tilde\theta < H$ for $\tilde\theta \in (\tilde\theta_H^-,\tilde\theta_H^+)$. Here we have periodic behaviour of $\theta$ and quasiperiodic of $x$ and $y$, the corresponding trajectory is called inflectional with inflection points at $\theta(t) = \tilde\theta_H^-+\tilde{\alpha}$ and $\theta(t) = \tilde\theta_H^+ +\tilde{\alpha}$ satisfying $\dot{\theta}(t)=0$.
\item $H = m_2$. Denote $[\tilde\theta_3, \tilde\theta_4] \equiv \cos_{\tilde{\Omega}}^{-1} m_2 (\mod 2\mathbb{S})$. If $\psi_3 \equiv 0$, then  $u_2 = \dot{\theta} \equiv 0$ and points $\tilde{\theta} \in (\tilde{\theta}_3, \tilde{\theta}_4)$ are equilibria, however points $\tilde{\theta} \in \{\tilde{\theta}_3, \tilde{\theta}_4\}, \psi_3 = 0$ may be not fixed a priori. If $\psi_3 \neq 0$, then the corresponding solutions are two separatrices, which can have two types of approaching to points $\tilde{\theta} \in \{\tilde{\theta}_3, \tilde{\theta}_4\}$ with $\psi_3 = 0$. If $\tilde{\Omega}$ is $C^2$ in the neighborhood of the approaching point with $\tilde{\theta} = \tilde{\theta}_3$ or  $\tilde{\theta} = \tilde{\theta}_4$ then the separatrix is approaching for infinite time to the point, if there is a corner at the corresponding point of $\tilde{\Omega}$ then the separatrix approaches for a finite time (similar to Sec.~\ref{sec:3DLie}). This case provides the main difference for behaviour of the same system~(\ref{vert-md3}) in problems~\ref{ela-subsec} and~\ref{md-subsec}, see further. This case corresponds to the critical (one loop) elastica for classic formulation of the problem, however, $\Omega$-modification of the problem admits for the corresponding trajectory to have an arbitrary number of separatrices with straight segments (of arbitrary length) in between when $\tilde{\Omega}$ is not $C^2$. Detailed analysis of such a situation can be found in Sec.~\ref{sec:second-order} (up to change $\tilde\theta$ to $\theta^\polar$, $H$ to $\mathbb{E}$ and $\cos_{\tilde{\Omega}} \tilde{\theta}$ to $\mathcal{U}(\theta^\polar)$).
\item $H \in (m_2, + \infty) \Rightarrow  \psi_3 \neq 0$, this case provides so-called non-inflectional ($\sgn \dot{\theta} \equiv \pm 1$) solutions with periodic $\theta$ and quasiperiodic $x$ and $y$.
\end{enumerate}

\begin{figure}[htbp]
\includegraphics[width=0.49\textwidth]{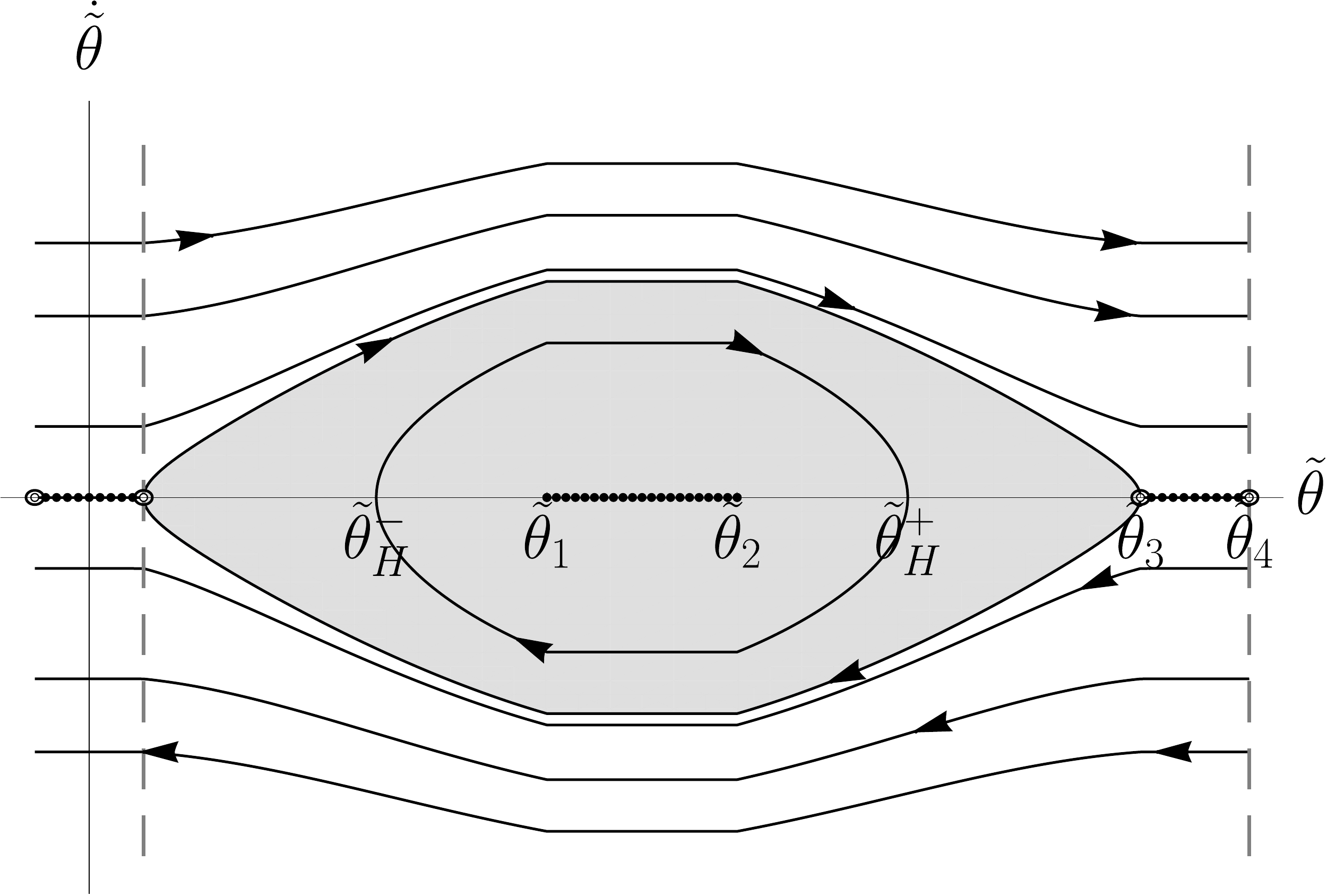}
\hfill
\includegraphics[width=0.49\textwidth]{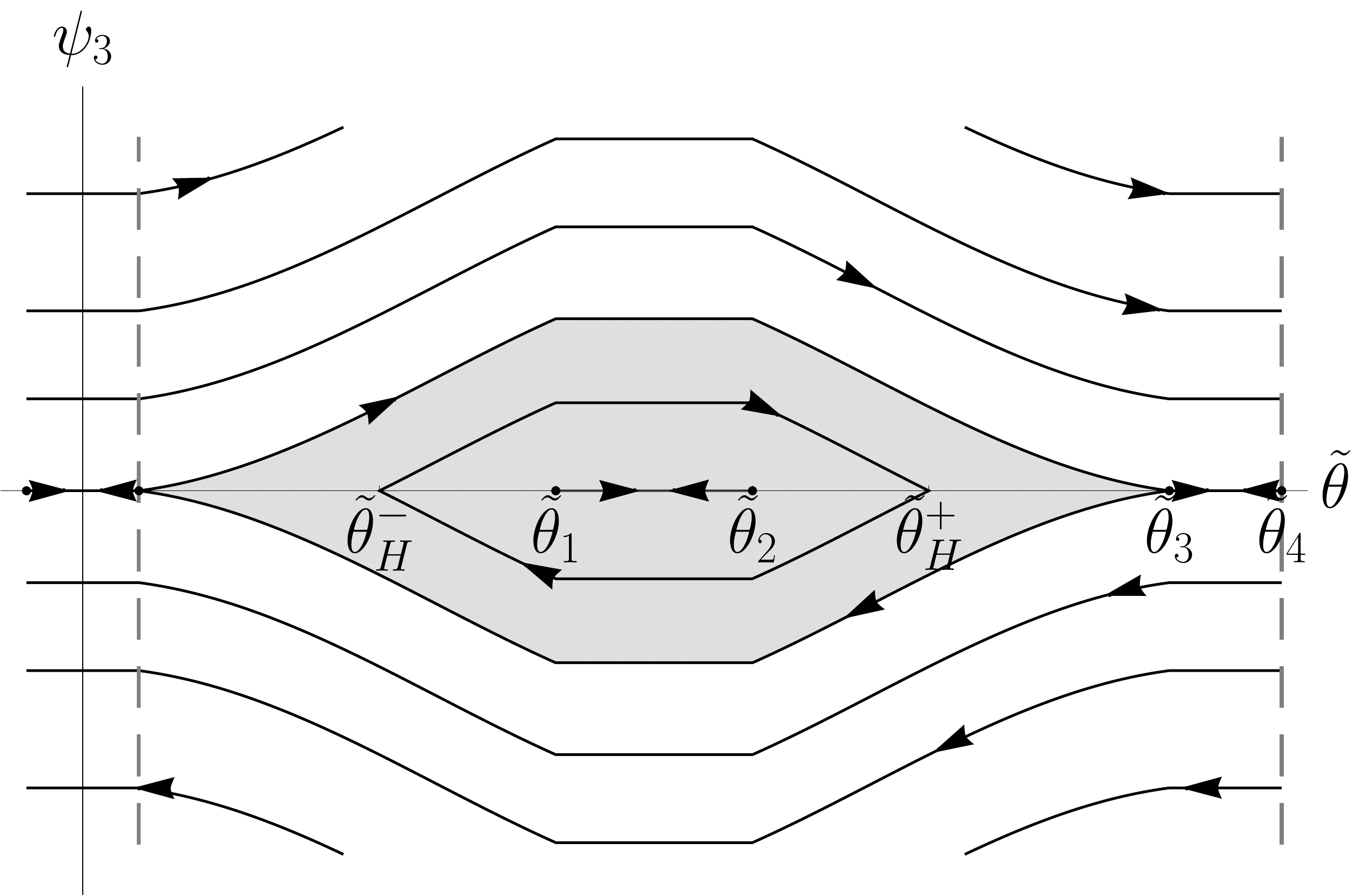}
\\
\parbox[t]{0.495\textwidth}{\caption{Phase portrait for $\Omega$-modifications of Euler elastica problem in the cylinder $(\tilde{\theta}, \dot{\tilde{\theta}})$ for $\tilde{\Omega}$ presented in Fig.~\ref{fig:DomainO}}\label{fig:phase_portrait_ela}}
\hfill
\parbox[t]{0.495\textwidth}{\caption{Phase portrait for $\Omega$-modifications of Markov-Dubins problem in the cylinder $(\tilde{\theta}, \psi_3)$ for $\tilde{\Omega}$ presented in Fig.~\ref{fig:DomainO}}\label{fig:phase_portrait_MD}}
\end{figure}

\subsection{Markov-Dubins problem}\label{md-subsec}
An $\Omega$-modification of Markov-Dubins problem is defined by~(\ref{sys-se2})--(\ref{boundary-se2}) with~(\ref{UfMD}), the problem is to find control $u_2 \in [-1,1]$.

Suppose $\psi_1 = \psi_2 = 0$. Then $\psi_3 \equiv \const \neq 0$ and from the maximum condition for $\mathcal{H} = u_2 \psi_3$ we have $H = |\psi_3|$ and $u_2 = \sgn \psi_3 \equiv \pm 1$. The corresponding trajectories on $(x,y)$ depend on the shape of $\Omega$ and can be obtained via integration as~(\ref{xy-dubins}), an example for $\Omega = \tilde\Omega$ presented on Fig.~\ref{fig:DomainO} is given on Fig.~\ref{fig:xy1}. Those two trajectories always meet at $t=2 \mathbb{S}$ after making one loop for any shape of $\Omega$, an example is given on Fig.~\ref{fig:xy1} at the right. 

\begin{figure}[ht]
	\centering
	\includegraphics[width=0.29\textwidth]{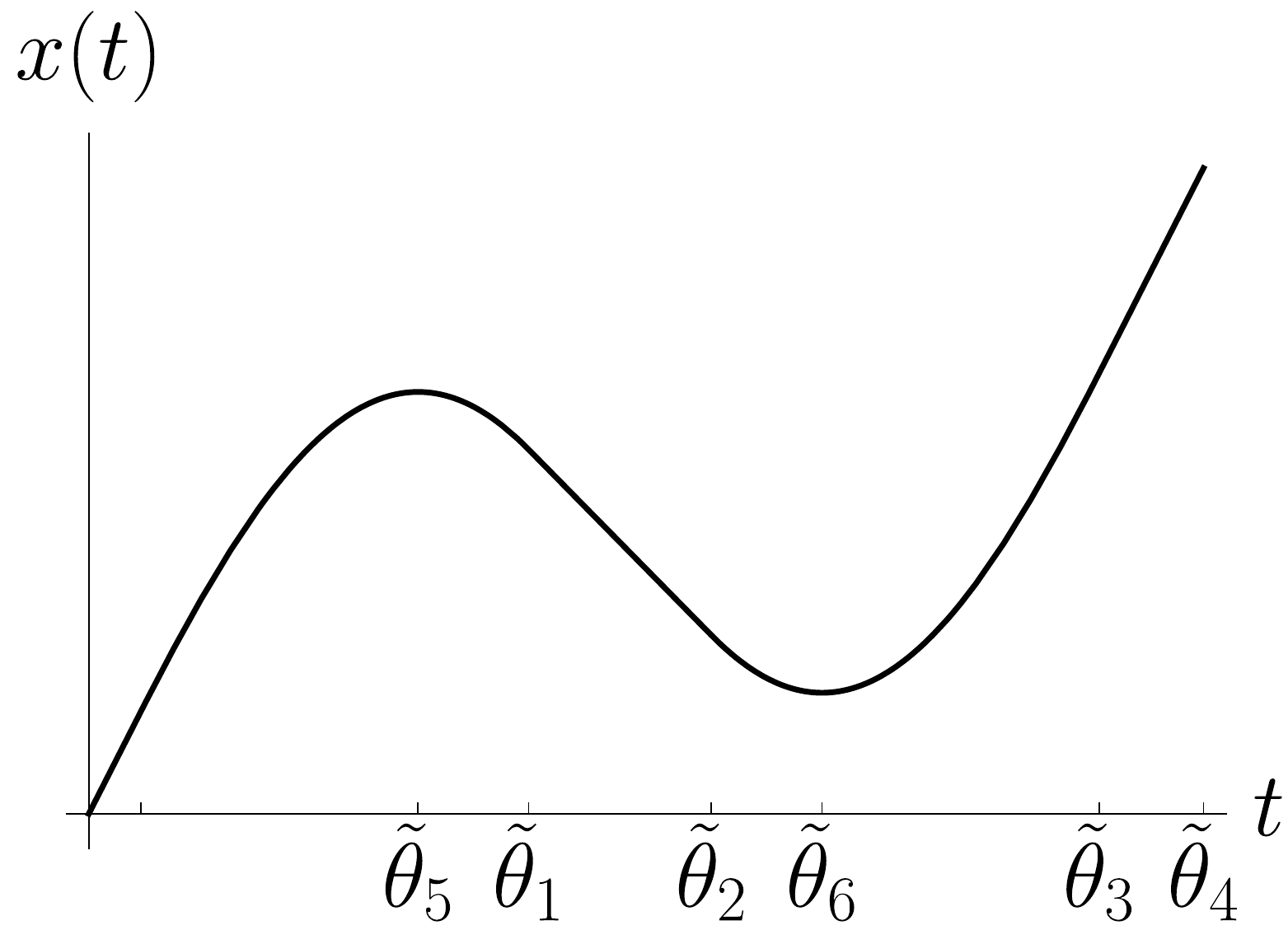} \qquad \qquad  \qquad \includegraphics[width=0.29\textwidth]{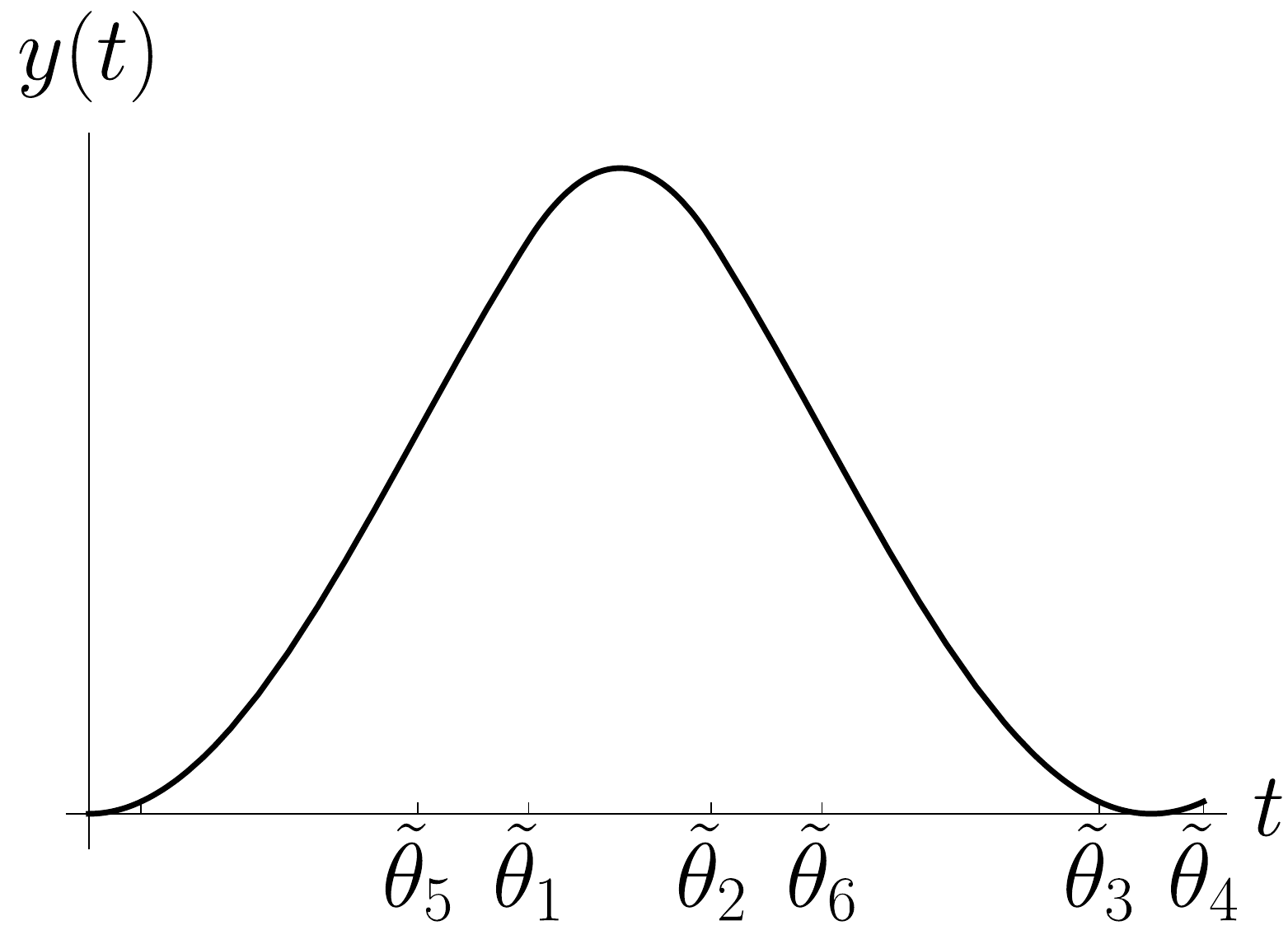} \qquad \qquad \qquad \includegraphics[width=0.105\textwidth]{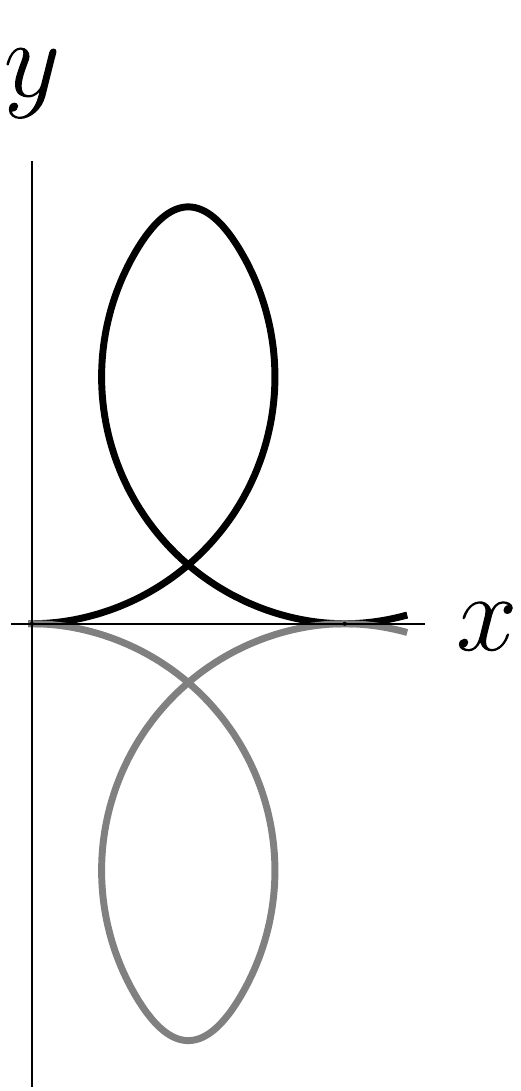}
\caption{Integrated functions $x(t), y(t)$ for $u_2 \equiv 1$ and trajectories on $(x,y)$ for $u_2 \equiv \pm 1$ with $\Omega = \tilde{\Omega}$ given on Fig.~\ref{fig:DomainO}}
	\label{fig:xy1}
\end{figure}

Let $\psi_1^2 + \psi_2^2 \neq 0$. We assume $\psi_1^2 + \psi_2^2 = 1$ w.l.o.g. 
From the maximality condition for $\mathcal{H}$~(\ref{H-short}) we have $$H = \cos_{\tilde{\Omega}} \tilde{\theta} + |\psi_3|, \qquad \dot{\tilde{\theta}}=u_2 = \sgn \psi_3 = \begin{cases}-1, \ &\psi_3<0, \\
[-1,1], \ &\psi_3=0,\\
1, \ &\psi_3>0.\end{cases}$$

The phase portrait of~(\ref{vert-md3}) on the cylinder $(\tilde{\theta}, \psi_3)$ is defined by the maximized Hamiltonian $H$ (see Fig.~\ref{fig:phase_portrait_MD}) and is similar to the one in Euler's elastic problem. Since $\tilde{\Omega}$ is compact, there is a minimal and maximal value of $\cos_{\tilde{\Omega}} \tilde{\theta}$ denoted by $m_1 < 0$, $m_2 > 0$.
The~following cases are possible:
\begin{enumerate}
\item $H < m_1$ --- no solutions.
\item\label{case1} $H = m_1 \Rightarrow \psi_3 \equiv 0, \cos_{\tilde{\Omega}} \tilde{\theta} = m_1$. If $\cos_{\tilde{\Omega}}^{-1} m_1$ is unique up to periodicity, then there is a unique solution defined by $\tilde{\theta} \equiv \cos_{\tilde{\Omega}}^{-1} m_1, u_2 = 0$; otherwise there is a family of solutions $\tilde{\theta}(t) \in \cos_{\tilde{\Omega}}^{-1} m_1 = [\tilde{\theta}_1, \tilde{\theta}_2]$ satisfying $|u_2(t)|\leq 1$ for $\tilde{\theta}(t) \in [\tilde{\theta}_1, \tilde{\theta}_2]$ with $u_2(t) \geq 0$ at $\tilde{\theta}(t) = \tilde{\theta}_1$ and $u_2(t) \leq 0$ at $\tilde{\theta}(t) = \tilde{\theta}_2$ (see Fig.~\ref{fig:phase_portrait_MD}).
\item $H \in (m_1, m_2), \cos_{\tilde{\Omega}} \tilde{\theta} \leq H \Rightarrow \psi_3 \not\equiv 0$, since for $\cos_{\tilde{\Omega}} \tilde{\theta} = H$ there holds	$\sin_{\tilde{\Omega}^\polar} \tilde{\theta}^\polar \neq 0$. Moreover $\cos_{\tilde{\Omega}}^{-1} H \equiv \{\tilde{\theta}_H^-, \tilde{\theta}_H^+\} (\mod 2 \mathbb{S})$, where $\tilde{\theta}_H^- < \tilde{\theta}_1 \leq \tilde{\theta}_2 < \tilde{\theta}_H^+$ (see Fig.~\ref{fig:phase_portrait_MD}). Therefore control $u_2$ is switching between $1$ and $-1$ at points $\tilde{\theta} \in \{\tilde{\theta}_H^-, \tilde{\theta}_H^+\}$, the corresponding time intervals are $T_b, T_1, \dots, T_1, T_e$, where $T_1 = \tilde{\theta}_H^+ - \tilde{\theta}_H^- < 2 \mathbb{S}$ with $T_b \leq T_1$, $T_e \leq T_1$.
\item\label{case3} $H = m_2$.  
When $\psi_3 = 0$ we have $\cos_{\tilde{\Omega}} \tilde{\theta} = m_2$ for $\tilde{\theta}\in [\tilde{\theta}_3, \tilde{\theta}_4] (\mod 2\mathbb{S})$ 
(see Fig.~\ref{fig:DomainO}). In this subcase when $\tilde{\theta} \in [\tilde{\theta}_3, \tilde{\theta}_4] (\mod 2\mathbb{S})$ we have arbitrary $|u_2(t)|\leq 1$ (below, we refer to such a control as uncertain control), it is possible to switch to certain control $u_2 = -1$ at $\tilde{\theta} = \tilde{\theta}_3$ and to certain control $u_2 = 1$ at $\tilde{\theta} = \tilde{\theta}_4$. After passing through the switching points $\tilde{\theta}_3, \tilde{\theta}_4$ we have $\psi_3 \neq 0$ and control $u_2 = \sgn \psi_3 \equiv  \pm 1$ for time $T_1 \equiv (\tilde{\theta}_3-\tilde{\theta}_4) (\mod  2\mathbb{S})$.
\item $H \in (m_2, + \infty), \psi_3 \neq 0$, control $u_2 \in \{-1,1\}$ is constant for arbitrary time similarly to the case $\psi_1=\psi_2=0$.
\end{enumerate}


\begin{thm}\label{Dubins1}
When $\Omega$ is strictly convex, then for any rotated $\tilde{\Omega}$ we have $\tilde{\theta}_1 \equiv \tilde{\theta}_2 (\mod 2\mathbb{S})$ and $\tilde{\theta}_3 \equiv \tilde{\theta}_4 (\mod 2\mathbb{S})$. Optimal control $u_2$ is a piecewise constant function with values $u_2^1, u_2^2, \dots, u_2^n$; corresponding time intervals $T^1, T^2, \dots, T^n$ and is one of two types:
\begin{itemize}
\item $u_2^{2k}=0, u_2^{2 k + 1} \in \{-1,1\}, k \in \N$; $T^{2 k}, k \in \N$ are arbitrary and $T^1 \leq 2 \mathbb{S}, T^n \leq 2 \mathbb{S}, T^{2k+1} = 2 \mathbb{S}$ for $k\neq 0, 2k+1 \neq n$.
\item $u_2^{2k} = \pm 1, u_2^{2 k + 1} = \mp 1, k \in \N$; $T^2 = T^3 = \dots = T^{n-1} = \hat{T} \leq 2 \mathbb{S}$, $T^1 \leq \hat{T}, T^n \leq \hat{T}$. 
\end{itemize}
\end{thm}

\begin{thm}\label{Dubins2}
 When $\Omega$ is not strictly convex, then $\tilde{\Omega}$ for some angle $\alpha$ has a vertical side at the right or at the left, as a consequence $\tilde{\theta}_3 \not \equiv \tilde{\theta}_4 (\mod 2\mathbb{S})$ or $\tilde{\theta}_1 \not \equiv \tilde{\theta}_2 (\mod 2\mathbb{S})$. This case admits not only optimal piecewise constant control described in Theorem~$\ref{Dubins1}$, but also admits optimal control with uncertain pieces corresponding to the edge $\theta (t) \in [\tilde{\theta}_3 + \tilde{\alpha},\tilde{\theta}_4+\tilde{\alpha}]$ or to the edge $\theta (t) \in [\tilde{\theta}_1 + \tilde{\alpha},\tilde{\theta}_2+\tilde{\alpha}]$, which are described in case~$\ref{case3}$ or in case~$\ref{case1}$ correspondingly.

In such situation case~$\ref{case1}$ defines strictly singular trajectories. Meanwhile, case~$\ref{case3}$ provides mixed (combination of singular and nonsingular) trajectories.
\end{thm}

\subsection{Reeds-Shepp problem}\label{subsec:R-S}
An $\Omega$-modification of Reeds-Shepp problem is defined by~(\ref{sys-se2})--(\ref{boundary-se2}) with~(\ref{UfRS}).

Suppose $\psi_1 = \psi_2 = 0$. Then $\psi_3 \equiv \const \neq 0$ and  from the maximum condition for $\mathcal{H} = u_2 \psi_3$ we have $H = |\psi_3|$ with $u_2 \equiv \sgn \psi_3$. For any choice of $u_1 \equiv \pm 1$, we get an optimal trajectory since it is not possible to turn on such an angle faster (up to infinity for ``non-factorized problem''), explicit formulas for $(x,y)$ can be obtained via integration as in (\ref{xy-dubins}) with $c=1$ for $u_1\equiv 1$ and inverse for $u_1\equiv -1$. 

Let $\psi_1^2 + \psi_2^2 \neq 0$. We assume $\psi_1^2 + \psi_2^2 = 1$ w.l.g. 
From the maximality condition for $\mathcal{H}$ we have $$H = |\cos_{\tilde{\Omega}} \tilde{\theta}| + |\psi_3|, \qquad u_1= \sgn \cos_{\tilde{\Omega}} \tilde{\theta}, \qquad \dot{\tilde{\theta}} = u_2 = \sgn \psi_3 = \begin{cases}-1, \ &\psi_3<0, \\
[-1,1], \ &\psi_3=0,\\
1, \ &\psi_3>0.\end{cases}$$

The phase portrait of~(\ref{vert-md3}) on the cylinder $(\tilde{\theta}, \psi_3)$ is defined by the 	maximized Hamiltonian $H$. 
There are only two points $\tilde{\theta}_5 = \theta_{\tilde{\Omega}} (\pi/2)$ and $\tilde{\theta}_6 = \theta_{\tilde{\Omega}} (3 \pi/2)$ satisfying the condition $\cos_{\tilde{\Omega}} \tilde{\theta} = 0$ on $\tilde{\Omega}$ (see Fig.~\ref{fig:DomainO}). 
Since $\tilde{\Omega}$ is compact, there is a minimal and maximal value of $\cos_{\tilde{\Omega}} \tilde{\theta}$ denoted by $m_1 < 0$, $m_2 > 0$. We assume $|m_1| \leq m_2$ without loss of generality. Suppose $\cos_{\tilde{\Omega}}^{-1} m_1 = [\tilde\theta_1, \tilde\theta_2]$ and $\cos_{\tilde{\Omega}}^{-1} m_2 = [\tilde\theta_3, \tilde\theta_4]$.

The~following cases are possible for $H \geq 0$:
\begin{enumerate}
\item $H = 0$, then $\dot{\theta}=\psi_3 \equiv 0$ and $\tilde{\theta} \in \{\tilde{\theta}_5,\tilde{\theta}_6\}$.
We have two fixed points with $u_1 = \pm 1, u_2 = 0$, the corresponding trajectory on $(x,y)$ is a straight line. 
\item\label{case2} $H \in (0, |m_1|)$. From convexity of $\tilde{\Omega}$ we have $\cos_{\tilde{\Omega}}^{-1} H = \{\tilde\theta_{+H}^+,\tilde\theta_{+H}^-\}$ and $\cos_{\tilde{\Omega}}^{-1} (-H) = \{\tilde\theta_{-H}^+,\tilde\theta_{-H}^-\}$. This subcase gives us two families of trajectories with $\tilde\theta \in [\tilde\theta_{+H}^+, \tilde\theta_{-H}^+]$ and $\tilde\theta \in [\tilde\theta_{-H}^-, \tilde\theta_{+H}^-]$. Consider the first one without loss of generality. The control $(u_1,u_2)$ is switching as follows: $\dots, (1,1), (-1,1), (-1,-1), (1,-1), \dots$. The corresponding time intervals $T_{(1,1)}, T_{(-1,1)}, T_{(-1,-1)}, T_{(1,-1)}$ satisfy $T_{(1,1)}=T_{(1,-1)} = \tilde\theta_5 - \tilde{\theta}_{+H}^+, T_{(-1,1)}=T_{(-1,-1)} = \tilde{\theta}_{-H}^+ - \tilde\theta_5$ (see Fig.~\ref{fig:phase_portrait_RS}). The first and the last control can be one of the four types with shortened  time intervals $T_b, T_e$.
\item\label{case3c} $H = |m_1|$. If $\tilde{\theta}\in [\tilde{\theta}_1, \tilde{\theta}_2] (\mod 2\mathbb{S})$ we have $u_1(t) = -1$ and arbitrary $|u_2(t)|\leq 1$ (an uncertain control), it is possible to switch to certain control $u_2 = -1$ while passing through $\tilde{\theta} = \tilde{\theta}_1$ and to certain control $u_2 = 1$ while passing through $\tilde{\theta} = \tilde{\theta}_2$, the time intervals corresponding to these certain controls are equal to $\tilde{\theta}_1-\tilde\theta_5$ and $\tilde{\theta}_6-\tilde{\theta}_2$, then we have $3$ more switchings to a certain controls according to Fig.~\ref{fig:phase_portrait_RS} similarly to case~\ref{case2}. After making full loop we go back to the points $\tilde\theta_1, \tilde\theta_2$, where we can switch to uncertain control again staying in $[\tilde\theta_1, \tilde\theta_2]$ or go for another full loop with certain controls.
\item\label{case4c} $H \in (|m_1|, m_2)$, this case can be treated in the same way as case~\ref{case2}. The control $(u_1, u_2)$ is switching in a sequence as follows: $\dots, (1,1), (-1,1), (1,1), (1,-1), (-1,-1), (1,-1), \dots$ (further repeating with the same pattern). The corresponding time intervals in the pattern are $\dots, T_1, T_2, T_3, T_3, T_2, T_1, \dots$, where $T_3= \tilde\theta_6 - \tilde\theta_5$ and values of $T_1, T_2$ depend on points $\cos_{\tilde\Omega}^{-1} H$.
\item\label{case5c} $H = m_2$, this case can be treated in the same way as case~\ref{case3c}. Here we have an uncertain control $u_2 (t) \in [-1,1]$ for interval $\tilde\theta (t) \in [\tilde\theta_3, \tilde\theta_4]$ with $u_1(t) = 1$. Passing through the point $\tilde\theta_3$ the control $u$ is switching according to certain pattern $(1,-1), (-1,-1), (1,-1)$ with the corresponding time intervals $\ T_1 = \tilde\theta_3-\tilde\theta_6, \ T_3 = \tilde{\theta}_6 - \tilde\theta_5, \ T_2 = \tilde{\theta}_5 - \tilde{\theta}_4 (\mod 2\mathbb{S})$. Passing through the point $\tilde\theta_4$ the control $u$ is switching to certain pattern $(1,1), (-1,1), (1,1)$ with the time intervals $T_2, T_3, T_1$.
\item $H \in (m_2, + \infty)$ with $\psi_3 \neq 0$ we have $u_2 \equiv \pm 1$, the other control $u_1$ is switching between $1$ and $-1$ with the corresponding time intervals $T_+ = \tilde\theta_6 - \tilde\theta_5, T_- = 2\mathbb {S} - T_+$. 
\end{enumerate}

\begin{remark}
When $|m_1| = m_2$ we are missing case~\ref{case4c} and cases~\ref{case3c},~\ref{case5c} join in the one. It provides two uncertain intervals $\tilde\theta \in [\tilde\theta_1, \tilde\theta_2] \cup [\tilde\theta_3, \tilde\theta_4]$ with $\psi_3 = 0$. It is possible to switch to certain controls at points $\tilde\theta_1; \tilde\theta_2; \tilde\theta_3; \tilde\theta_4$ as it is described in cases~\ref{case3c},~\ref{case5c}, however, after passing two certain intervals with controls $(-1,-1),(1,-1); (-1,1),(1,1); (1,-1),(-1,-1); (1,1),(-1,1)$ correspondingly, we approach to the corresponding points $\tilde\theta_4; \tilde\theta_3; \tilde\theta_2; \tilde\theta_1$, where we have can switch to uncertain control or continue with certain control according to the same pattern.
\end{remark}

\begin{figure}[htbp]
\includegraphics[width=0.49\textwidth]{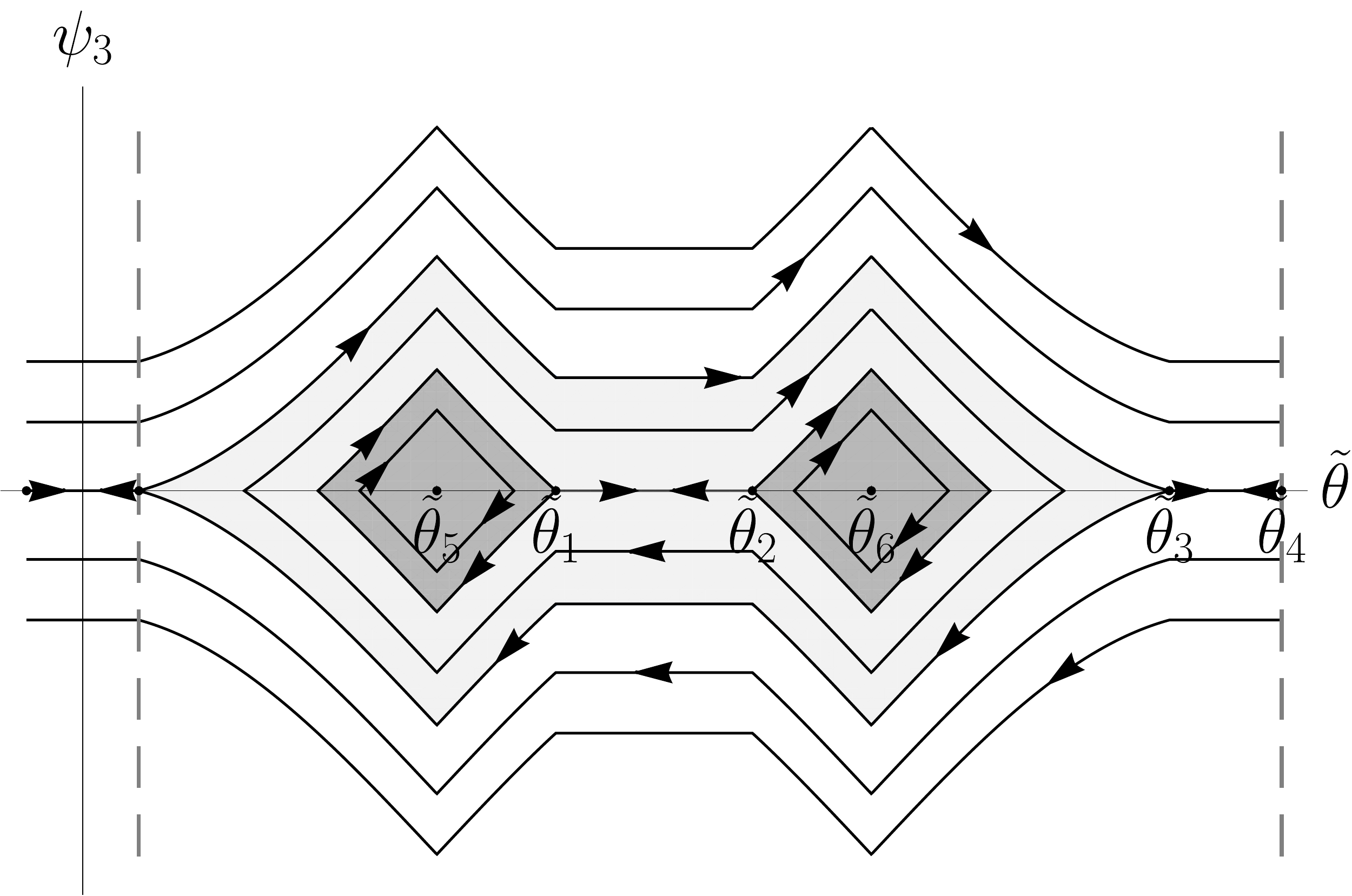}
\hfill
\includegraphics[width=0.49\textwidth]{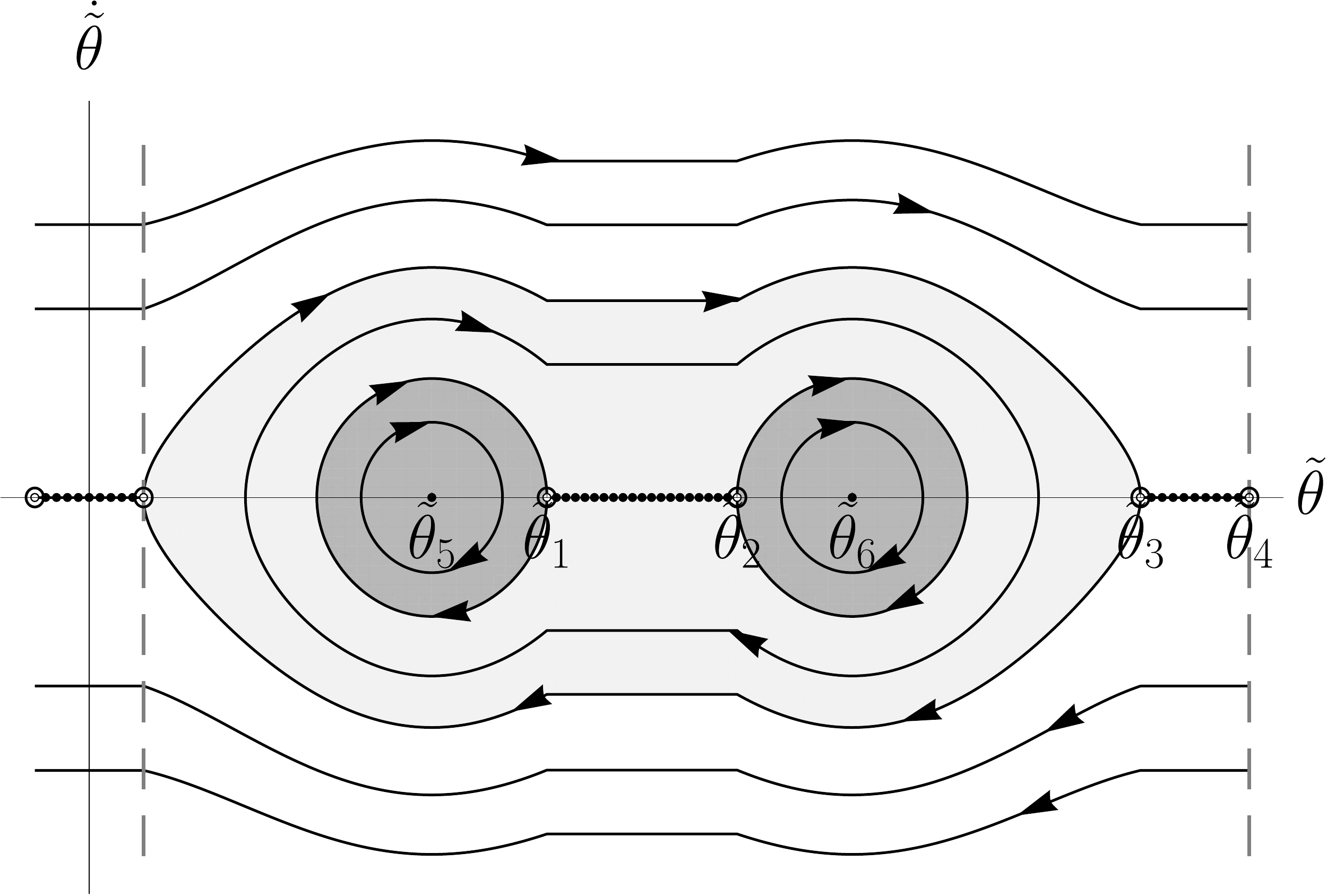}
\\
\parbox[t]{0.495\textwidth}{\caption{Phase portrait for $\Omega$-modifications of Reeds-Shepp problem in the cylinder $(\tilde{\theta}, \psi_3)$ for $\tilde{\Omega}$ presented in Fig.~\ref{fig:DomainO}}\label{fig:phase_portrait_RS}}
\hfill
\parbox[t]{0.495\textwidth}{\caption{Phase portrait for $\Omega$-modifications of Sub-Riemannian problem on $\SE(2)$ in the cylinder $(\tilde{\theta}, \dot{\tilde{\theta}})$ for $\tilde{\Omega}$ presented in Fig.~\ref{fig:DomainO}}\label{fig:phase_portrait_SR}}
\end{figure}

\subsection{Sub-Riemannian problem on $\SE(2)$}
An $\Omega$-modification of the sub-Riemannian problem on $\SE(2)$ is defined by~(\ref{sys-se2})--(\ref{boundary-se2}) with~(\ref{UfSR1}).

The abnormal case $\psi_0 = 0$ is possible only as a trivial solution $u_1 = u_2 = 0$.

Normal case: $\psi_0 = -1$. From the maximality condition for $\mathcal{H}$ we have 
\[
u_1 = \psi_1 \cos_{\Omega} \theta + \psi_2 \sin_{\Omega} \theta, \qquad \dot{\tilde{\theta}} = u_2 = \psi_3.
\] 

Suppose $\psi_1 = \psi_2 = 0$, then $u_1 \equiv 0$ and $\dot{\psi}_3 \equiv 0$, therefore $u_2  \equiv \const$, moreover we can assume w.l.o.g. $u_2 \equiv \pm 1$. The corresponding trajectory is optimal up to infinity for ``non-factorized problem''.

If $\psi_1^2 + \psi_2^2 \neq 0$, then we assume $\psi_1^2 + \psi_2^2 = 1$ w.l.g. 
Using the same notation as in~(\ref{vert-md3}) we get
\begin{equation}\label{vert-rs2}
\dot{\psi}_3 = \cos_{\tilde{\Omega}} \tilde{\theta} \sin_{\tilde{\Omega}^\polar} \tilde{\theta}^\polar 
\end{equation}
with the maximized Hamiltonian $H = \frac{1}{2} (\cos_{\tilde{\Omega}}^2 \tilde{\theta} + \dot{\tilde{\theta}}^2)$, which determines the phase portrait of system~(\ref{vert-rs2}) (see Fig.~\ref{fig:phase_portrait_SR}). This Hamiltonian system (up to change $\tilde{\theta}$ for $\theta^\circ$ and $\tilde{\Omega}$ for $\Omega^\circ$) was considered in details in Sec.~\ref{sec:3DLie} and at the same time it is similar to the one considered in the previous Subsection~\ref{subsec:R-S}, further we remark the main difference between them.	


Since $\tilde{\Omega}$ is compact, we denote the minimal and maximal values of $\cos_{\tilde{\Omega}} \tilde{\theta}$ by $m_1$ and $m_2$ and assume $m_1^2 \leq m_2^2$ without loss of generality. The following cases are possible for $H \geq 0$:
\begin{enumerate}
\item $H = 0$, then we have only a trivial solution $u_1 = u_2 = 0$. 
\item $H \in (0,m_1^2)$, using similar to case~\ref{case2} (see Subsec.~\ref{subsec:R-S}) notation $\cos_{\tilde{\Omega}}^{-1} H = \{\tilde\theta_{+H}^+,\tilde\theta_{+H}^-\}$, $\cos_{\tilde{\Omega}}^{-1} (-H) = \{\tilde\theta_{-H}^+,\tilde\theta_{-H}^-\}$
we get inflectional solutions with small amplitudes for the angle parameter $\tilde\theta \in [\tilde\theta_{+H}^+, \tilde\theta_{-H}^+]$ and $\tilde\theta \in [\tilde\theta_{-H}^-, \tilde\theta_{+H}^-]$.
\item\label{case3-d} $H = m_1^2$. If $\psi_3 = 0$, then we have fixed points for $\tilde\theta \in (\tilde\theta_1,\tilde\theta_2)$, points $\tilde\theta \in \{\tilde\theta_1,\tilde\theta_2\}$ are also fixed when $\Omega$ is $C^2$ in the neighborhood. If $\psi_3 \neq 0$, then we have two small separatrices approaching to points $\tilde\theta \in \{\tilde\theta_1,\tilde\theta_2\}$ (for infinite time when $\Omega$ is $C^2$ in the neighborhood of the approaching points). 
\item $H \in (m_1^2, m_2^2)$, this case corresponds to inflectional solutions with big amplitude for angle $\tilde\theta$.
\item $H = m_2^2$, similarly to case~\ref{case3-d} we have fixed points at $\tilde\theta \in (\tilde\theta_3, \tilde\theta_4)$, two big separatrices for $\tilde\theta \in (\tilde\theta_4, \tilde\theta_3) (\mod 2\mathbb{S})$, which approach to points $\tilde\theta \in \{\tilde\theta_3, \tilde\theta_4\}$ for infinite time when $\Omega$ is $C^2$ in the neighborhood.
\item $H \in (m_2^2, + \infty)$; corresponding trajectories on $(x,y)$ are non-inflectional ones, since $\dot\theta = \psi_3 \neq 0$. 
\end{enumerate}
Detailed analysis for cases $H = m_1^2$ and $H = m_2^2$ can be found in Sec.~\ref{sec:second-order} (up to change $\tilde\theta$ to $\theta^\polar$, $H$ to $\mathbb{E}$ and $\cos_{\tilde{\Omega}}^2 \tilde{\theta}$ to $\mathcal{U}(\theta^\polar)$).


\section{Plane dynamic motion}
\label{sec:plane_dynamic}

In this section, we obtain exact formulae for extremals in the following problem with drift:
\[
T\to\min,\qquad\ddot x=u\in\Omega,
\]
\noindent with some given initial conditions. Here $x=(x_1,x_2)$, $u=(u_1,u_2)$, and $\Omega\subset\R^2$ is a compact convex set with $0\in\mathrm{int}\,\Omega$. Put $\dot x=y$. The Pontryagin maximum principle gives $\mathcal{H}=p_1y_1 + p_2y_2 + q_1u_1 + q_2u_2$, where $p$ and $q$ are conjugate variables to $x$ and $y$ correspondingly. Since $\dot p=-\mathcal{H}_x$ and $\dot q=-\mathcal{H}_y$, we have $p=p_0=\const$ and $q=-p_0t + q_0$. Hence, for a.e.\ $t$, the control $u(t)$ moves along $\partial\Omega$ and the supporting half-plane to $\Omega$ at $u(t)$ is determined by $q(t)$. The main difficulty here is to describe this motion conveniently.

From the integrability point of view, we have a Hamiltonian system with 4 degrees of freedom, and it is easy to find 4 independent first integrals: $p_1$, $p_2$, $E=[p\times q]=p_1q_2-p_2q_1$, and the following nonsmooth Hamiltonian:
\[
	H=\max_{u\in\Omega}\mathcal H = p_1y_2 + p_2y_2 + s_\Omega(q_1,q_2).
\]
\noindent Obviously, the first integrals $p_1$, $p_2$, and $E$ are in involution. So, if $H$ is a smooth function outside $q_1=q_2=0$, then the system is integrable by Liouville-Arnold theorem. Unfortunately, $H$ is nonsmooth outside $q_1=q_2=0$ if $\Omega$ is not strongly convex. Moreover, to construct angle-momentum coordinates one should compute some explicit integration along basic loops on common level surfaces of the listed first integrals, which is not very simple procedure. Nonetheless, we are able to compute convenient formulae for extremals in terms of convex trigonometry. Put
\[
	q_1 = R\cos_{\Omega^\polar}\theta^\polar
	\quad\mbox{and}\quad
	q_2 = R\sin_{\Omega^\polar}\theta^\polar,
\]
\noindent where $R=s_\Omega(q_1,q_2)=H-p_1y_1-p_2y_2\ge 0$ is not constant along an extremal (in contrary to all previously considered problems). Nonetheless, the angle $\theta^\polar$ depends (locally) Lipschitz continuously on $t$ while $R\ne 0$. Using the generalized Pythagorean identity (see Subsec.~\ref{subsec:CT_properties}), we obtain $u_1=\cos_\Omega\theta$ and	$u_2=\sin_\Omega\theta$ for some $\theta\leftrightarrow\theta^\polar$ if $R\ne 0$, since $q_1u_1+q_2u_2=R$. So we need to compute $\theta^\polar(t)$ to find the control~$u$. Obviously,
\[
	E = [p\times q] = R(p_1\sin_{\Omega^\polar}\theta^\polar - p_2\cos_{\Omega^\polar}\theta^\polar).
\]

We start with the simplest case $E=0$, i.e.,\ when $p\parallel q$. If $p=0$, then $q=\const$ and $\theta^\polar=\const$. Suppose $p\ne 0$. Then there exists a unique instant $t_0$, such that $q(t_0)=0$. Hence, $\theta^\polar(t)$ takes only two possible values -- one for $t<t_0$ and another for $t>t_0$. Both are solutions of the equation $p_2\cos_{\Omega^\polar}\theta^\polar-p_1\sin_{\Omega^\polar}\theta^\polar=0$.

Now suppose that $E\ne 0$. In this case, let us express $R$ in $\theta^\polar$ by computing $\dot\theta^\polar$. The angle derivative formula in polar change of coordinates is the same in convex and classical trigonometries (see the inverse polar change of coordinates in Subsec.~\ref{subsec:CT_properties}). Hence, $\dot\theta^\polar = (q_1\dot q_2-\dot q_1q_2)/R^2 = E/R^2$, and we obtain the following form of first integral $E$ for the Pontryagin system:
\begin{equation}
\label{eq:plane_dynamic_integral}
	E\dot\theta^\polar = (p_2\cos_{\Omega^\polar}\theta^\polar - p_1\sin_{\Omega^\polar}\theta^\polar)^2.
\end{equation}
\noindent Hence, if $\Omega^\polar$ has $C^k$ boundary, then $\theta^\polar(t)\in C^{k+1}$ by Proposition~\ref{prop1}. Integrating equation~\eqref{eq:plane_dynamic_integral}, we obtain
\[
	t/E = \int \frac{d\theta^\polar}{(p_2\cos_{\Omega^\polar}\theta^\polar - p_1\sin_{\Omega^\polar}\theta^\polar)^2}.
\]

Let us consider in detail two important shapes of $\Omega^\polar$. If $\Omega^\polar$ is an ellipse (see Subsec.~\ref{subsec:explicit_sets}), then $\cos_{\Omega^\polar}\theta^\polar = a\cos s + x_1^0$ and $\sin_{\Omega^\polar}\theta^\polar = b\sin s + x_2^0$ where ${\theta^\polar}'_s = ab + ax_2^0\sin s + b x_1^0\cos s$. Thus,
\[
	t/E = \int
	\frac{ab + ay_0\sin s + b x_0\cos s}
	{\bigl(ap_2\cos s  - bp_1\sin s + p_2x_1^0 - p_1x_2^0\bigr)^2}
	ds.
\]
\noindent The last integral can be easily taken in elementary functions. Moreover,
\[
	u_1=\cos_\Omega\theta = \frac{d\sin_{\Omega^\polar}\theta^\polar}{d\theta^\polar}=
	\frac{d(b\sin s + x_2^0)}{ds}\frac{ds}{d\theta^\polar}=
	\frac{b\cos s}{ab + ax_2^0\sin s + b x_1^0\cos s},
\]
\[
	u_2=\sin_\Omega\theta = \frac{-d\cos_{\Omega^\polar}\theta^\polar}{d\theta^\polar}=
	\frac{d(-a\cos s - x_1^0)}{ds}\frac{ds}{d\theta^\polar}=
	\frac{a\sin s}{ab + ax_2^0\sin s + b x_1^0\cos s}.
\]

If $\Omega^\polar$ is a polygon, then $\cos_{\Omega^\polar}\theta^\polar$ and $\sin_{\Omega^\polar}\theta^\polar$ depend piecewise linearly on $\theta^\polar$, so equation~\eqref{eq:plane_dynamic_integral} can be easily integrated in elementary functions as well.

\section*{Conclusion}

Summarizing, we have derived formulae for extremals in a series of optimal control problems with two-dimensional control lying in an arbitrary compact convex set $\Omega\subset \R^2$ with $0\in\mathrm{int}\,\Omega$. 
Until now, formulae for extremals in these problems were known only for the case $\Omega$ being ellipse cetered at the origin (there is a small amount of exceptions mentioned in the Introduction).
Investigation of optimal control problems involves a lot of work in addition to deriving formulae for extremals (such as verification of second order conditions, computing conjugate and cut times, finding closed trajectories and etc.). Nonetheless, precise investigations of optimal syntheses are almost always based on handy solutions to the Pontryagin maximum principle. Thus, we believe that the results obtained in the present paper open up a wide opportunity for precise investigations in considered optimal control problems, which were previously impossible.

\section*{Acknowledgement}
The authors thank anonymous reviewers whose comments on presentation were taken into account.


\end{document}